\documentclass[a4paper, 11pt]{article}
\usepackage{amsfonts}
\usepackage {amssymb}
\usepackage {amsmath}
\usepackage {amsmath}
\usepackage {amsthm}
\usepackage{graphicx}
\usepackage{xypic}
\usepackage {amscd}
\usepackage{mathrsfs}
\usepackage[colorlinks, linkcolor=red, anchorcolor=black, citecolor=red]{hyperref}
\usepackage{enumerate}

\usepackage{geometry}  

\newcommand{\Fut}{{\rm Fut}}

\newcommand{\sddb}{{\sqrt{-1}\partial\bar{\partial}}}

\newcommand{\bL}{{\bf L}}
\newcommand{\mcY}{{\mathcal{Y}}}
\newcommand{\mcL}{{\mathcal{L}}}
\newcommand{\mcX}{{\mathcal{X}}}
\newcommand{\mcD}{{\mathcal{D}}}
\newcommand{\mcM}{{\mathcal{M}}}

\newcommand{\vol}{{\rm vol}}
\newcommand{\ord}{{\rm ord}}
\newcommand{\lct}{{\rm lct}}
\newcommand{\vphi}{\varphi}

\newcommand{\bC}{{\mathbb{C}}}

\newcommand{\FS}{{\rm FS}}
\newcommand{\bP}{{\mathbb{P}}}

\newcommand{\mcW}{\mathcal{W}}
\newcommand{\chw}{{\rm CW}}

\newcommand{\mcQ}{{\mathcal{Q}}}

\newcommand{\sslash}{{/\!/}}
\newcommand{\bG}{\mathbb{G}}
\newcommand{\bK}{\mathbb{K}}
\newcommand{\bT}{\mathbb{T}}

\newcommand{\NA}{{\rm NA}}
\newcommand{\MA}{{\rm MA}}

\newcommand{\mcJ}{{\mathcal{J}}}
\newcommand{\mcZ}{\mathcal{Z}}

\newcommand{\triv}{{\rm triv}}
\newcommand{\bQ}{{\mathbb{Q}}}
\newcommand{\reg}{{\rm reg}}

\newcommand{\mcB}{{\mathcal{B}}}
\newcommand{\bR}{{\mathbb{R}}}

\newcommand{\mcH}{{\mathcal{H}}}
\newcommand{\mcO}{{\mathcal{O}}}
\newcommand{\mcF}{{\mathcal{F}}}
\newcommand{\bZ}{{\mathbb{Z}}}

\newcommand{\bD}{{\mathbb{D}}}

\newcommand{\sing}{{\rm sing}}

\newcommand{\bN}{{\mathbb{N}}}

\newcommand{\la}{\langle}
\newcommand{\ra}{\rangle}
\newcommand{\mcE}{\mathcal{E}}
\newcommand{\bfL}{{\bf L}}
\newcommand{\mcI}{\mathcal{I}}

\newcommand{\Val}{{\rm Val}}
\newcommand{\msc}{\mathscr}
\newcommand{\DHM}{{\rm DH}}
\newcommand{\mcR}{\mathcal{R}}

\newcommand{\ccL}{\check{\mcL}}
\newcommand{\wt}{\mathrm{wt}}
\newcommand{\ccZ}{{\check{\mcZ}}}
\newcommand{\ccQ}{{\check{\mcQ}}}

\newcommand{\bfD}{{\bf D}}
\newcommand{\mcU}{{\mathcal{U}}}
\newcommand{\rVal}{\mathring{\Val}}
\newcommand{\bfM}{{\bf M}}
\newcommand{\bfH}{{\bf H}}
\newcommand{\bfE}{{\bf E}}
\newcommand{\bfJ}{{\bf J}}
\newcommand{\lc}{{\rm lc}}
\newcommand{\bfF}{{\bf F}}
\newcommand{\ac}{{\rm ac}}
\newcommand{\bfj}{{\bf j}}

\newcommand{\Aut}{{\rm Aut}}
\newcommand{\PSH}{{\rm PSH}}
\newcommand{\bd}{{\rm bd}}
\newcommand{\cZ}{\mathcal{Z}}
\newcommand{\Lam}{{\bf \Lambda}}
\newcommand{\cF}{\mathcal{F}}
\newcommand{\bfI}{{\bf I}}

\newcommand{\cI}{\mathcal{I}}
\newcommand{\cO}{\mathcal{O}}

\newcommand{\Xdiv}{X^{\mathrm{div}}_\bQ}
\newcommand{\Zdiv}{Z^{\mathrm{div}}_\bQ}

\newcommand{\Blue}{}

\newcommand{\Ad}{{\rm Ad}}

\newcommand{\fg}{\mathfrak{g}}
\newcommand{\ft}{\mathfrak{t}}
\newcommand{\fk}{\mathfrak{k}}
\newcommand{\fa}{\mathfrak{a}}
\newcommand{\cE}{\mathcal{E}}

\newtheorem{thm}{Theorem}[section]
\newtheorem{prop}[thm]{Proposition}
\newtheorem{defn}[thm]{Definition}

\newtheorem{cor}[thm]{Corollary}
\newtheorem{rem}[thm]{Remark}

\newtheorem{exmp}[thm]{Example}
\newtheorem{lem}[thm]{Lemma}

\newtheorem{defn-prop}[thm]{Definition-Proposition}

\begin{document}

\title{$G$-uniform stability and K\"{a}hler-Einstein metrics on Fano varieties}
\author{Chi Li\\
\; \\
\textit{with an appendix by Jun Yu}}
\date{}

\maketitle

\abstract{
Let $X$ be any $\bQ$-Fano variety and $\Aut(X)_0$ be the identity component of the automorphism group of $X$. \Blue{Let $\bG$ be a connected reductive subgroup of $\Aut(X)_0$ that contains a maximal torus of $\Aut(X)_0$. 
We prove that $X$ admits a K\"{a}hler-Einstein metric if and only if $X$ is $\bG$-uniformly K-stable. } This proves a version of Yau-Tian-Donaldson conjecture for arbitrary singular Fano varieties.  A key new ingredient is a valuative criterion for $\bG$-uniform K-stability. 
}

\setcounter{tocdepth}{2}
\tableofcontents

\section{Introduction}

\Blue{In this paper, we always work over the field $\bC$ of complex numbers.} A log Fano pair $(X, D)$ is a normal projective variety $X$ together with an effective $\bQ$-Weil divisor $D$ such that $L:=-(K_X+D)$ is an ample $\bQ$-Cartier divisor and $(X, D)$ has at worst klt singularities \Blue{(see Definition \ref{def-klt})}. If $D=0$, then $X$ is called a $\bQ$-Fano variety.
In \cite{LTW19}, the author together with G. Tian and F. Wang proved the uniform version of Yau-Tian-Donaldson conjecture: a $\bQ$-Fano variety $X$ with a discrete automorphism group admits a K\"{a}hler-Einstein metric if and only if $X$ is uniformly K-stable, if and only if $X$ is uniformly Ding-stable. \Blue{Note that the klt condition is necessary both for the existence of K\"{a}hler-Einstein metrics (see \cite[Proposition 3.8]{BBEGZ}) and for K-(semi)stability (\cite[Theorem 1.3]{Oda13})}.  

In this paper, we consider the case when the automorphism group is not discrete. In this case, Hisamoto \cite{His16b} introduced a $\bG$-uniform stability condition (he called it relatively uniform stability for $\bG$) and made an insightful observation that this stability condition corresponds nicely with an analytic criterion for equivariant \Blue{coercivity} which he obtained by using Darvas-Rubinstein's principle. 
We will derive a refined analytic criterion for the existence of K\"{a}hler-Einstein metric and use Hisamoto's stability condition to formulate our main results. %in this work.
% will play a basic role in our work.
\vskip 2mm
{\bf Notation:} In this paper, we will use the following notations:
\begin{enumerate}[(i)]
\item $\Aut(X, D)$ denotes the automorphism group of $(X, D)$ (i.e. the automorphism of $X$ that preserves $D$). $\Aut(X, D)_0$ is its identity component. 
\item $\bG$ is a connected reductive subgroup of $\Aut(X, D)_0$. $C(\bG)$ is the center of $\bG$ and $\bT:=C(\bG)_0$ is the identity component of $C(\bG)$. \Blue{$\bT$ is a connected, commutative and reductive algebraic group, and is well-known to be isomorphic to an algebraic torus $(\bC^*)^r=((S^1)^r)^{\bC}$ (see \cite[Corollary 16.15]{Mil17}). }

\item $\bK$ is a maximal compact subgroup of $\bG$ that contains $(S^1)^r$.
\end{enumerate}

% and by $\bK$ a maximal compact subgroup of $\bG$ that contains $(S^1)^r$. 
\vskip 2mm
\begin{defn}[{see \cite{His16b, His19}}]\label{defn-bGKstable}
With the above notation, $(X, D)$ is called $\bG$-uniformly K-stable if there exists $\gamma>0$ such that for any $\bG$-equivariant test configuration $(\mcX, \mcD, \mcL)$ for $(X, D, -(K_X+D))$ (see Definition \ref{defn-GequiTC}), the following inequality holds true:
\begin{equation}\label{eq-bGKstable}
\Blue{\bfM^\NA}(\mcX, \mcL)\ge \gamma\cdot \bfJ^\NA_{\bT}(\mcX, \mcL).
\end{equation}
\Blue{See \eqref{eq-MNA} for the definition of $\bfM^\NA$} and \eqref{eq-JNAT} for $\bfJ^\NA_{\bT}$.
If one replace the $\bfM^\NA$ invariant by $\bfD^\NA$ (see \eqref{eq-bfDTC}), then one defines $\bG$-uniform Ding-stability of $(X, D)$ (called relatively uniform $\bfD$-stability for $\bG$ in \cite{His16b}). 
\end{defn}

We will prove the following general existence result:
\begin{thm}\label{thm-YTD}
Let $(X, D)$ be a log Fano pair. $(X, D)$ admits a K\"{a}hler-Einstein metric if $(X, D)$ is $\bG$-uniformly K-stable, or equivalently if $(X, D)$ is $\bG$-uniformly Ding-stable. 
\end{thm} 

In the case when $X$ is a smooth Fano manifold and $D=\emptyset$, the above result can be derived from the work \cite{DaS16} (see Remark \ref{rem-unipoly}), which depends on the method of partial $C^0$-estimates. Here we don't require extra constraint on the singularities of $(X, D)$.

To prove Theorem \ref{thm-YTD}, we first need to derive a valuative criterion for $\bG$-uniform Ding/K-stability. To state this criterion, \Blue{recall that $\bT\cong (\bC^*)^r$ and set}
\begin{equation}\label{eq-Nlattice}
N_\bZ={\rm Hom}(\bC^*, \bT),\quad N_{\bQ}=N_\bZ\otimes_{\bZ}\bQ,\quad N_\bR=N_\bZ\otimes_\bZ \bR.
\end{equation}
Also set $M_\bZ=N_\bZ^{\vee}, M_{\bQ}=M_\bZ\otimes_{\bZ}\bQ, M_\bR=M_\bZ\otimes_\bZ \bR$.

Denote by $\Val(X)$ the set of (real) valuations on $X$. For any valuation $v\in\Val(X)$, denote by $A_{(X,D)}(v)$ the log discrepancy of $v$ as defined in \cite{JM12, BFFU15}. \Blue{$\Val(X)$ contains the subset $\Xdiv$ of all divisorial valuations. Recall that any valuation in $\Xdiv$ is of the form $\lambda \cdot \ord_E$ for a prime divisor $E$ on a birational model of $X$ and $\lambda \in \bQ_{>0}$.} %, denoted by $\DVal$.
Denote by $\Val(X)^\bT$ (resp. $\Val(X)^\bG$) the set of $\bT$-invariant (resp. $\bG$-invariant) valuations on $X$. 
Similarly we denote by $(\Xdiv)^\bT$ (resp. $(\Xdiv)^\bG$) the set of $\bT$-invariant (resp. $\bG$-invariant) divisorial valuations on $X$. 
We observe that $N_\bR$ acts on $(\Val(X))^\bT$: $(\xi, v)\mapsto v_\xi$ (see section \ref{sec-NRact}). If we choose any $\ell_0$ such that \Blue{$\ell_0 L=-\ell_0 (K_X+D)$ is Cartier, then $v$ induces a filtration $\mcF_v=\mcF_v R_\bullet$ on $R:=R^{(\ell_0)}:=\bigoplus_{m=0}^{+\infty}H^0(X, m\ell_0 L)$} (see \eqref{eq-filval}). Define an invariant (see \eqref{eq-defSLv}):
\begin{equation}
\Blue{S_L}(v):=\frac{1}{\ell_0^n (-(K_X+D))^{\cdot n}}\int_0^{+\infty} \vol\left(\mcF^{(x)}_v\right)dx.
\end{equation}
Given $(X, D)$, this is an invariant of $v$ and does not depend on the choice of $\ell_0$. 

Let $\mathfrak{t}$ denote the Lie algebra of $(S^1)^r\subset \bT$ which is identified with the set of \Blue{holomorphic vector fields} generated by the elements of $\mathfrak{t}$. Note that there is a natural isomorphism $\mathfrak{t}\cong N_\bR$. 
\Blue{Denote by $\Fut: N_\bR\rightarrow \bR$ the well-known Futaki invariant which can be defined either analytically (see \eqref{eq-Futana}) or algebraically (see \eqref{eq-CWFut}). }

\begin{thm}\label{thm-Gvalcriterion}
With the above notation, the following statements are equivalent.
\begin{enumerate}[(1)]
\item $(X, D)$ is $\bG$-uniformly K-stable;
\item $(X, D)$ is $\bG$-uniformly Ding-stable;
\item $\Fut\equiv 0$ on $N_\bR$ and there exists $\delta_\bG>1$ such that for any $\bG$-invariant divisorial valuation $v$ over $X$ there exists $\xi\in N_\bR$ satisfying $A_{(X,D)}(v_\xi)\ge \delta_\bG \cdot \Blue{S_L}(v_\xi)$.
%\item $\Fut\equiv 0$ on $N_\bR$ and there exists $\delta_\bG>1$ such that for any $v\in (\Blue{\Xdiv})^\bG$ there exists $\xi \in N_\bR$ satsifying $A_{(X,D)}(v_\xi)\ge \delta_\bG\cdot \Blue{S_L}(v_\xi)$.
\item $(X, D)$ is $\bG$-uniformly K-stable \Blue{with respect to} $\bG$-equivariant special test configurations.
%\item For any filtration $\mcF$, $\lct(\fb_\bullet(\mcF_\xi))/S(\mcF_\xi)\ge \delta_\bG>1$.
\end{enumerate}
\end{thm}
Here the last condition (4) means that in Definition \ref{defn-bGKstable} the inequality \eqref{eq-bGKstable} is required only for $\bG$-equivariant special test configurations (see Definition \ref{defn-TC} and \ref{defn-GequiTC}). 

In practice, we have the following result that serves the same purpose as what a result from \cite{DaS16} does for obtaining K\"{a}hler-Einstein metrics on varieties with large symmetries. Again the advantage of our result is that it works for all singular Fano varieties.
\begin{cor}
Assume that there are only finitely many $\bG$-equivariant special degenerations of $(X, D)$. If $(X, D)$ is $\bG$-equivariantly K-polystable, then $(X, D)$ is $\bG$-uniformly K-stable. Hence $(X, D)$ admits a K\"{a}hler-Einstein metric. 
\end{cor}
Here by a $\bG$-equivariant special degeneration we mean a special test configuration but \Blue{without the data of $\bC^*$-action}. 

We will then show that the converse to Theorem \ref{thm-YTD} holds true if $\bG$ contains a maximal torus of $\Aut(X)_0$. This is true because the existence of K\"{a}hler-Einstein metrics implies a \Blue{coercivity} condition involving such $\bG$, 
which we will prove by combining the works of Darvas-Rubinstein and Hisamoto, together with some properties of reductive groups proved in Appendix \ref{app-Yu}. 
%The converse to Theorem \ref{thm-YTD} follows from Darvas and Hisamoto's work \cite{His16b} (see Theorem \ref{thm-analytic}, Theorem \ref{thm-JTslope}) and Theorem \ref{thm-Gvalcriterion}. 
So we get a sufficient and necessary algebraic conditions for the existence of K\"{a}hler-Einstein metrics for any (singular) Fano variety.
\begin{thm}\label{thm-AutYTD}
Let $(X, D)$ be a log Fano pair. Then $(X, D)$ admits a K\"{a}hler-Einstein metric if and only if $\Aut(X, D)_0$ is reductive and $(X, D)$ is $\bG$-uniformly \Blue{K}-stable, where $\bG$ is any
connected reductive group of $\Aut(X, D)_0$ that contains a maximal torus of $\Aut(X, D)_0$. 
\end{thm}
%The other case is when $\bG$ is a maximal torus of $\Aut(X, D)_0$. By using Darvas-Rubinstein's principle and some properties of reductive groups in Appendix \ref{app-Yu}, we can prove an analogous analytic criterion in this case (Theorem \ref{thm-torusproper}), from which we get a new equivariantly uniform version of Yau-Tian-Donaldson conjecture for any (singular) Fano variety.
%\begin{thm}\label{thm-maxtorus}
%Let $(X, D)$ be a log Fano pair and $\bG$ be any maximal torus in $\Aut(X, D)_0$. Then $(X, D)$ admits a K\"{a}hler-Einstein metric if and only if $(X, D)$ is $\bG$-uniformly K-stable, if and only if $(X, D)$ is $\bG$-uniformly Ding-stable.
%\end{thm}
Theorem \ref{thm-AutYTD} is the first \Blue{version} of Yau-Tian-Donaldson conjecture for arbitrary Fano varieties.
We make some remarks about the above results.
\begin{rem}\label{rem-unipoly}
In this remark we use \Blue{Definition} \ref{defn-Gequivstability} and Remark \ref{rem-Gequivstability}. 
\begin{enumerate}
\item
By \Blue{definition}, $\bG$-equivariantly uniform K-stability implies $\bG$-uniform K-stability  (since $\bfJ^\NA\ge \bfJ^\NA_\bT$). The converse is not true in general.
In fact, it is easy to show that $\bG$-equivariantly uniform K-stability is equivalent to two conditions together: $\bG$-uniform K-stability plus the center $C(\bG)$ being discrete.
So for the above results, if $C(\bG)$ is discrete, we can replace $\bG$-uniform K-stability (Ding-stability) by $\bG$-equivariantly uniform K-stability (Ding-stability). 
We note that $\bG$-equivariantly uniform K-stability was considered recently in \cite{Gol19} and \cite{Zhu19}. %, and as noted there, the theory developed previously carries over to this case.
\item
It can be shown that $\bG$-uniform K-stability implies $\bG$-equivariant $K$-polystability (Lemma \ref{lem-uni2poly}).
Conversely $\bG$-equivariant K-polystability does not in general imply $\bG$-uniform K-stability if $\bG$ is too small compared to $\Aut(X, D)_0$ (e.g. take $X=\bP^n$ and $\bG=\{e\}$). With our result, it is natural to expect that for any $\bG$ containing a maximal torus, $\bG$-equivariant K-polystability (or just K-polystability) is equivalent to $\bG$-uniform K-stability (see also \cite{LTZ19}). This is known in the smooth case by the works in \cite{DaS16} and \cite{His16b} through the existence of K\"{a}hler-Einstein metrics. \Blue{When $\bG$ is a maximal torus of $\Aut(X, L)_0$, this has been confirmed for general $\bQ$-Fano varieties by Liu-Xu-Zhuang in \cite{LXZ21} by using deep techniques from birational algebraic geometry. More precisely, in \cite[Theorem 5.2]{LXZ21} it is shown that if $\tilde{\bT}$ denotes a maximal torus of $\Aut(X, D)_0$, then K-polystability (or just $\tilde{\bT}$-equivariant K-polystability) implies $\tilde{\bT}$-uniform K-stability ($\tilde{\bT}$-uniform stability is called reduced uniform stability in \cite{LXZ21}). 
Moreover in this case, we know by Theorem \ref{thm-YTD} that there exists a K\"{a}hler-Einstein metric on $(X, D)$, which in turn implies K-polystability of $(X, D)$ (by \cite{Berm15}) and also $\bG$-uniform stability for any $\bG$ containing a maximal torus of $\Aut(X, D)_0$ (by Theorem \ref{thm-AutYTD}). }
 % for Yau-Tian-Donaldson conjecture for smooth Fano manifolds (see \cite{Tia97, Berm15, CDS15, Tia15, DaS16, His16b, LWX18}).% (or \cite{CDS15, Tia15}) and \cite{His16b}.
% In the case when $\bG$ is \Blue{an algebraic torus}, one can show that $\bG$-equivariant K-polystability is equivalent to K-polystability. On the other hand, note that K-polystability indeed implies reductivity of automorphism groups by the recent work \cite{ABHLX} (see also \cite{LTZ19}).
%\item
%The connectedness assumption of $\bG$ is only essentially used in the proof of the implication from K-stability to Ding-stability, where we used the $\bC^*\times\bG$-equivariant MMP process to get special test configurations where K-stability coincides with Ding-stability. The existence part of the proof still goes through and hence a version of Yau-Tian-Donaldson conjecture via Ding-stability holds true for more general case of disconnected subgroups of $\Aut(X)$ by probably invoking a more general version of analytic criterion than Theorem \ref{thm-torusproper}. We leave to the reader to write down the valid statement.
%One can work with general reductive (possibly non-connected) subgroup $\bG$ of $\Aut(X, D)$. This only affects the proof of the implication of $\bG$-uniform Ding-stability by $\bG$-uniform K-stability, since we need to run $\bG$-equivariant MMP which is known if $\bG$ is connected. 
\end{enumerate}
\end{rem}

We end the introduction with a short discussion of proofs.
The general idea for the proof of Theorem \ref{thm-Gvalcriterion} parallels the idea for the proof of valuative criterion by Fujita and the author in \cite{Fuj18, Li17}, which uses the \Blue{equivariant, relative} MMP process from \cite{LX14} (see also section \ref{sec-alternative}). However, we need to understand in detail how to relate the twists of valuations to the twists of \Blue{non-Archimedean potential}s including those from test configurations. Note that the notion of twist of test configurations appeared in Hisamoto's work \cite{His16a,His16b}.
We also need to establish that the $\bfJ^\NA_\bT$ energy for filtration (associated to valuations) can be approximated by $\bfJ^\NA_\bT$ for test configurations. The other observation is that \Blue{some calculations in \cite{Fuj18}, showing that $\bfD^\NA-\epsilon \bfJ^\NA$ is decreasing along the MMP (for $\epsilon \in [0,1]$), are compatible with twists}.

In addition to the valuative criterion in Theorem \ref{thm-Gvalcriterion}, the work here is a synthesis of ideas from \cite{BBJ18}, \cite{His16b} and \cite{LTW19}, and further carries out Berman-Boucksom-Jonsson's program of variational approach (proposed in \cite{BBJ15, BBJ18}) to Yau-Tian-Donaldson conjecture for all $\bQ$-Fano varieties. However compared with all these previous works, we need to find new ways to deal with difficulties arising from singularities and continuous automorphism groups. To overcome the difficulties caused by singularities, we use the perturbative idea from our previous work (\cite{LTW17, LTW19}). But instead of directly proving $\bG$-uniform stability on the resolution as in these works, we will work with valuations that approximately calculate the $\bL^\NA$ part of the non-Archimedean Ding energy. This will also allow us to effectively use a key identity (see \eqref{eq-phiFtwist} and \eqref{eq-A+phivxi}) about twists of \Blue{non-Archimedean potential}s in order to deal with the case with continuous automorphism groups. In addition, our proof depends on monotonicity of both parts of the $\bfJ$ energy functional and some delicate uniform estimates of non-Archimedean quantities. The main line of arguments is essentially contained in a chain of (in)equalities in section \ref{sec-step4}. %In particular our way to overcome difficulties caused by continuous automorphism groups is \Blue{different from} Hisamoto's argument (see Remark \ref{rem-Hisamoto}).
%In the first part of the paper, we will derive a valuative criterion for equivariant uniform Ding-stability. Then we will prove the equivariant uniform version of Yau-Tian-Donaldson conjecture. The proof will be a combination of arguments from Berman-Boucksom-Jonsson, Hisamoto and Li-Tian-Wang.

\vskip 3mm
\noindent
{\bf Acknowledgement:} 
The author is partially supported by NSF (Grant No. DMS-1810867) and an Alfred P. Sloan research fellowship. I would like to thank Gang Tian for constant support and his interest in this work, and thank Xiaowei Wang for helpful discussions on related topics, Feng Wang, Xiaohua Zhu and Chenyang Xu for helpful comments, S\'{e}bastien Boucksom for his interest in our work, and Tomoyuki Hisamoto for communications concerning Remark \ref{rem-Hisamoto}. I would like to thank Yuchen Liu for comments and clarifications, which motivate me to write down the results for more general reductive subgroups, and thank Jiyuan Han and Kuang-Ru Wu for attending my lectures about this work patiently and give valuable \Blue{feedback} which allows me to improve the \Blue{presentation}. I am particularly grateful to Jun Yu for answering my questions in the appendix concerning reductive groups. Some parts of this paper were written during the author's visit to BICMR at Peking University, School of Mathematical Sciences at Capital Normal University and Shanghai Center for Mathematical Sciences at Fudan University. I would like to thank these institutes for providing wonderful environment of research. In particular, I would like to thank Gang Tian, Zhenlei Zhang and Peng Wu for their hospitality. I would also like to thank anonymous referees for their careful reading and providing very helpful comments for improving the paper.

\section{Preliminaries}

%\subsection{Energy functions}\label{sec-energy}
\subsection{Space of K\"{a}hler metrics over singular projective varieties}
Let $Z$ be an $n$-dimensional normal \Blue{projective} variety. 
We use the following convention: a smooth differential form $\theta$ (of any bi-degree $(p,q)$) on $Z$ is by definition a smooth differential form on the regular locus $Z^{\text{reg}}$ such that for any point $z\in Z$ there exist an open neighborhood $U\subset Z$, a local holomorphic embedding $\iota: U\rightarrow \bC^N$ (for some $N\gg 1$) and a smooth differential form $\Theta$ of bidegree $(p,q)$ on a neighborhood $W$ of $\iota(U)$ such that $\theta=\Theta_{W\cap U_{\reg}}$.  
We also recall the standard definition for plurisubharmonic functions on $Z$. 

\begin{defn}
Let $U$ be an open subset of $Z$. A function $\psi: U\rightarrow [-\infty, \infty)$ is called plurisubharmonic (psh) on $U$ if for any $z\in U$, there exist a neighborhood $z\in U_1\Subset U$, a local holomorphic embedding $\iota: U_1\rightarrow \bC^N$ and a plurisubharmonic function $\Psi$ on a neighborhood of $\iota(U_1)$ inside $\bC^N$ such that $\psi=\Psi\circ \iota$. 

We say that $\psi$ is smooth (resp. continuous, or bounded) if we can furthermore choose $\Psi$ to be smooth (resp. continuous, or bounded). 
\end{defn}
\begin{rem}
By a basic result of Fornaess-Narasimhan (\cite[Theorem 5.3.1]{FN80}), we know that a function $\psi: U\rightarrow [-\infty, \infty)$ is plurisubharmonic if and only if the following two conditions are satisfied:
\begin{enumerate}
\item[(i)] $\psi$ is upper semi-continuous at any point $z\in U$.
\item[(ii)] For any holomorphic map $\tau: \Delta=\{w\in \bC; |w|<1\}\rightarrow U$, the function $\psi\circ \tau$ is subharmonic on $\Delta$.
\end{enumerate}
\end{rem}

\begin{defn}
Let $L$ be an ample line bundle \Blue{on $Z$}. A psh (i.e. plurisubharmonic) Hermitian metric on $L$ is a collection $e^{-\psi}=\{e^{-\psi_\alpha}\}$ where $\psi_\alpha$ are locally defined psh functions, called local potential functions of the Hermitian metric, that are compatible with transition functions of local trivializations of $L$ (in a standard way). \Blue{The psh Hermitian metric $e^{-\psi}$} is called smooth (resp. continuous, bounded) if all $\psi_\alpha$ are smooth (resp. continuous, bounded). 

If $L$ is an ample $\bQ$-line bundle, a psh Hermitian metric on $L$ is a collection $e^{-\psi}=\{e^{-\psi_\alpha}\}$ satisfying the property that there exists $\ell\in \bZ_{>0}$ such that $\ell L$ is a line bundle and $e^{-\ell \psi}=\{e^{-\ell \psi_\alpha}\}$ is a psh Hermitian metric on $\ell L$. 
\end{defn}

A convenient way to get smooth psh Hermitian metrics on an ample $\bQ$-line bundle is as follows. Choose $\ell\in \bZ_{>0}$ sufficiently divisible such that $\ell L$ is a very ample line bundle. Choose a basis ${\bf \mathfrak{s}}:=\{s^{(\ell)}_1, \dots, s^{(\ell)}_{N_\ell}\}$ of $H^0(Z, \ell L)$. Denote by $\iota=\iota_{\bf \mathfrak{s}}: Z\rightarrow \bP^{N_\ell-1}$ the Kodaira embedding induced by the chosen basis such that $\iota^*H=\ell L$ where $H$ is the hyperplane bundle of $\bP^{N_\ell-1}$. 
Define a Hermitian metric on $L$ by pulling back the standard Fubini-Study Hermitian metric:
\begin{equation}\label{eq-refpsi}
e^{-\psi}=(\iota^*h_{\FS})^{1/\ell}=\left(\frac{1}{\sum_i\left|s^{(\ell)}_i\right|^2}\right)^{1/\ell}.
\end{equation}
Set $\omega=\sddb \psi=\frac{1}{\ell}\iota^*\omega_{\FS}$ where $\omega_{\FS}$ is the standard Fubini-Study K\"{a}hler metric on $\bP^{N_\ell-1}$. Then $\omega$ is a smooth positive $(1,1)$-form representing the first Chern class of $2\pi c_1(L)$. %Note that the first Chern classes of $\bQ$-line bundles are well-defined because $Z$ is normal. 
Moreover, in this construction, if a compact Lie group $\bK$ acts holomorphically on $(Z, L)$, then we can choose the data $\{e^{-\psi}, \omega\}$ to be $\bK$-invariant. Indeed, in this case, because $\bK$ naturally acts on $H^0(Z, \ell L)$, we can choose a $\bK$-invariant Hermitian inner product on $H^0(Z, \ell L)$ and choose the above basis ${\bf \mathfrak{s}}$ to be orthonormal. Then it is easy to see that $e^{-\psi}$ and $\omega$ are $\bK$-invariant. 

From now on, we fix such a reference metric $e^{-\psi}$ and $\omega=\sddb \psi$. 
A function $u: Z\rightarrow [-\infty, \infty)$ is called $\omega$-psh if for any point $z\in Z$, there exist open subsets $\Blue{U\subset U_\alpha\subset Z}$ such that $\psi_\alpha+u$ is psh on $U$. 
%Choose a smooth Hermitian metric $e^{-\psi}$ on $L$ with a smooth semi-positive curvature form $\omega=\sddb\psi\in 2\pi c_1(L)$. 
We will use the following spaces:
\begin{align}
&{\rm PSH}(\omega):=\PSH(Z, \omega)=\left\{\omega\text{-psh functions} \right\};\\
&\bar{\mcH}(\omega):=\bar{\mcH}(Z, \omega)={\rm PSH}(\omega)\cap C^\infty(Z);\\
&{\rm PSH}_\bd(\omega):=\PSH_\bd(Z, \omega)=\PSH(\omega)\cap \{\text{bounded functions on } Z\};\\
&{\rm PSH}(L)\Blue{:=}\PSH(Z, L)=\left\{\psi+u; u\in {\rm PSH}(\omega)\right\};\\
&{\rm PSH}_\bd(L)\Blue{:=}\PSH_\bd(Z, L)=\left\{\psi+u; u\in {\rm PSH}_\bd(\omega)\right\}.
\end{align}
\Blue{$\PSH(L)$ is the same as the space of psh Hermitian metrics $\{e^{-\vphi}=e^{-\psi-u}\}$ on the $\bQ$-line bundle $L$. Note that, rigorously speaking, $\psi+u$ is not a globally defined function, but rather a collection of local psh functions that satisfy the obvious compatible condition with respect to the transition functions of the $\bQ$-line bundle. However for the simplicity, we will adopt this notation and call any $\vphi\in \PSH(L)$ a psh potential. }
%We endow $\PSH(\omega)$ with the weak topology.
%$\mcE^1$ contains all bounded $\omega$-psh functions.% According to \cite{BBEGZ}, $u=\vphi-\psi\in \mcE^1(\omega)$ if and only if $\bfE_{\psi}(\vphi)>-\infty$.

By Hironaka's theorem there exists a resolution of singularities $\mu: Y\rightarrow Z$ which is obtained via an imbedding $\iota: Z\rightarrow \bP^N$ and then by taking the strict transform of $X$ under a sequence of blowups along smooth centers. It is well-known that such resolution of singularities can be made functorial. In particular if there is a group $G$ acting holomorphically on $X$, one can guarantee the existence of $G$-equivariant resolution of singularities (see \cite{Kol07} for more details). \Blue{Because the composition $\iota\circ \mu: Y\rightarrow \bP^N$ is a holomorphic map, the} pullback of the Hermitian metric $e^{-\psi}$ defined in \eqref{eq-refpsi} is a smooth psh Hermitian metric $e^{-\mu^*\psi}$ on $\mu^*L$, whose Chern curvature is a smooth semipositive closed $(1,1)$-form $\tilde{\omega}:=\mu^*\omega$ satisfying $\int_Y \tilde{\omega}^n=\int_Z \omega^n>0$. Because $Z$ is normal, $\mu$ has connected and compact fibers. So every $\tilde{\omega}$-psh function on $Y$ is of the form $u\circ \mu$ for a unique $\omega$-psh function $u$ on $Z$. So we have the identification 
\begin{equation}\label{eq-PSHY}
\PSH(Z, \omega)\cong \PSH(Y, \tilde{\omega}).
\end{equation} 
This is a homeomorphism if we endow both sides with \Blue{the} weak topology which coincides with the $L^1$-topology with respect to the smooth volume form $\omega^n$ (resp. $\tilde{\omega}^n$). If $u_j$ converges to $u$ in $\PSH(Z, \omega)$, then $\sup(u_j)\rightarrow \sup(u)$ by Hartogs' lemma for plurisubharmonic functions \cite[Theorem 1.46]{GZ17}. Moreover, we have the following lemma, which in the smooth case can be proved by using Green's formula.
\begin{lem}\label{lem-Hartogs}
There exists $C=C(\omega)>0$ such that for any $u\in \PSH(\omega)$ with $\sup(u)=0$, we have 
\begin{equation}
\int_Z u\omega^n\ge -C.
\end{equation}
%Let $u_j\in {\PSH}(\omega)$ be a sequence such that $\sup (u_j)=0$. Then there exists a constant $C>0$ independent of $j$, such that:
%\begin{equation}
%\int_X u_j \omega^n\ge -C.
%\end{equation}
\end{lem}
\begin{proof}
If this is not true, then there exists a sequence $u_j\in \PSH(\omega)$ with $\sup(u_j)=0$ which satisfies 
\begin{equation}\label{eq-nobd}
\int_Z u_j\omega^n\le -j.
\end{equation}
However, by Hartogs lemma in \cite[Theorem 1.46]{GZ17}, applied on the resolution $\mu: Y\rightarrow Z$ as above, we know that $\mu^* u_j$ converges in $L^1(Y, \tilde{\omega}^n)$ to $\tilde{u}=u\circ \mu \in \PSH(\tilde{\omega})$ (see \eqref{eq-PSHY}). So we get $\int_Z u_j \omega^n=\int_Y (\mu^*u_j) \tilde{\omega}^n \rightarrow \int_Y \tilde{u} \tilde{\omega}^n=\int_Z u\omega^n>-\infty$ which contradicts \eqref{eq-nobd}.
\end{proof}
The following global regularization result will be useful to us.
\begin{prop}[{\cite[Corollary C]{CGZ13}}]\label{prop-smapp}
For any $u\in \PSH(Z, \omega)$ there exists a sequence of smooth functions $u_j\in \PSH(Z, \omega)$ which decrease pointwise on $Z$ so that $\lim_{j\rightarrow+\infty}u_j=u$ on $Z$.
\end{prop}
For any $u\in \PSH(Z, \omega)$, set $\tilde{u}=\mu^*u\in \PSH(Y, \tilde{\omega})$ and define:
\begin{equation}
\tilde{\omega}_{\tilde{u}}^n:=\lim_{j\rightarrow+\infty} {\bf 1}_{\{\tilde{u}>-j\}}\left(\tilde{\omega}+\sddb \max(\tilde{u}, -j)\right)^n.
\end{equation}
This is always well-defined by \cite[Proposition 1.6]{BEGZ}. 
Set $\omega_u^n=\mu_*\tilde{\omega}_{\tilde{u}}^n$ such that $\int_Z \omega_u^n=\int_Y \tilde{\omega}_{\tilde{u}}^n$. 
More generally, for any $\{\vphi_k; 1\le k\le n\}\subset \Blue{\PSH(L)}$, we define their mixed complex Monge-Amp\`{e}re measure as:
\begin{equation}\label{eq-mixed}
(\sddb\vphi_1)\wedge \cdots (\sddb\vphi_n)=\mu_*\left(\left\langle \sddb\mu^*\vphi_1 \wedge \cdots\wedge \sddb\mu^*\vphi_n \right\rangle\right)
\end{equation}
where we used the non-pluripolar product of closed positive currents on compact K\"{a}hler manifolds introduced in \cite{BEGZ}: set $u_j=\vphi_j-\psi$, $\tilde{u}_j=\mu^*u_j$ and define
\begin{eqnarray*}
&&\left\la \sddb\mu^*\vphi_1\wedge \cdots\wedge \sddb\mu^*\vphi_n \right\ra\\
&&=
\lim_{j\rightarrow+\infty} {\bf 1}_{\bigcap_{k=1}^n \{u_k>-j\}} (\tilde{\omega}+\sddb \max\{\tilde{u}_1, -j\})\wedge \cdots \wedge (\tilde{\omega}+\sddb \max\{\tilde{u}_n, -j\}).
\end{eqnarray*} 
This non-pluripolar product generalizes the mixed Monge-Amp\`{e}re measure of bounded psh metrics (due to Bedford-Taylor) and is again always well-defined by \cite[Proposition 1.6]{BEGZ}. 
%\begin{equation}
%\omega_u^n:=\lim_{j\rightarrow+\infty}  {\bf 1}_{\{u>-j\}}\left(\omega+\sddb \max(u, -j)\right)^n.
%\end{equation}

We will use the space $\mcE^1$ of finite energy $\omega$-psh functions (see \cite{GZ07, BEGZ, BBEGZ}):
\begin{align}
&\mcE(\omega):=\mcE(Z, \omega)=\left\{u\in {\rm PSH}(Z, \omega); \int_Z \omega_u^n=\int_Z \omega^n\right\};\\
&\mcE^1(\omega):=\mcE^1(Z,\omega)=\left\{u\in \mcE(Z, \omega); \int_Z |u| \omega_u^n<\infty\right\};\\
&\mcE^1(L):=\mcE^1(Z, L)=\left\{\psi+u; u \in \mcE^1(Z, \omega) \right\}.
\end{align}
We have the inclusion $\PSH_\bd(\omega)\subset \mcE^1(\omega)\subset \mcE(\omega)$.

Set $V=(L^{\cdot n})$. For any $\vphi \in {\rm PSH}_\bd(Z, L)$, we have the following important functional: %$\chi$ is a closed $(1,1)$-current
\begin{eqnarray}
\bfE(\vphi)&:=&\bfE_{\psi}(\vphi)=\frac{1}{(n+1)(2\pi)^nV} \sum_{i=0}^n \int_Z (\vphi-\psi) (\sddb\psi)^{n-i}\wedge (\sddb\vphi)^{i} \label{eq-Ephibd}\\
&=&\frac{1}{(n+1)(2\pi)^n V}\sum_{i=0}^n \int_Y (\mu^*(\vphi-\psi))(\sddb \mu^*\psi)^{n-i}\wedge (\sddb \mu^*\vphi)^i. \nonumber
\end{eqnarray}
Following \cite[2.2]{BEGZ}, for any $\vphi \in \PSH(Z, L)$, define:
\begin{eqnarray}
\bfE(\vphi)&=&\inf\left\{\bfE(\tilde{\vphi}); \vphi \in \PSH_\bd(Z, L), \tilde{\vphi}\ge \vphi \right\}.
\end{eqnarray}
Then $\bfE$ is a concave, non-decreasing and finite-valued function on $\cE^1$. See \eqref{eq-Ephi} for an explicit formula generalizing \eqref{eq-Ephibd}.  
Following \cite{BBEGZ}, we endow $\mcE^1$ with the strong topology. 
\begin{defn}\label{defn-strong}
The strong topology on $\mcE^1$ is defined \Blue{as} the coarsest refinement of the weak topology such that $\bfE$ is continuous.
\end{defn}
%We will use the following monotone and rescaling property of %$\Lambda$ and 
%$\bfE$ functional:
%\begin{eqnarray}\label{eq-monotoneE}
%\vphi_1\le \vphi_2 &\Longrightarrow& %\quad {\bf \Lambda}(\vphi_1)\le {\bf \Lambda}(\vphi_2)\quad {\text and }
% \bfE(\vphi_1)\le \bfE(\vphi_2); \quad \bfE_{\lambda \psi}(\lambda \vphi)=\lambda^{n+1} \cdot \bfE_{\psi}(\vphi) \text{ for any } \lambda\in \bR_{>0}. \label{eq-rescaleE}
%\end{eqnarray}

For any interval $I\subset \bR$, denote the Riemann surface $$\bD_I=I\times S^1=\{\tau\in \bC^*; s=-\log|\tau|^2\in I\}.$$ 
\begin{defn}[{see \cite[Definition 1.3]{BBJ18}}]\label{defn-pshpath}
A $\omega$-psh path, or just called a psh path, on an open interval $I$ is a map $U=\{u(s)\}: I\rightarrow \PSH(\omega)$ such that the $U(\cdot, \tau):=U(\log|\tau|)$ is a $p_1^*\omega$-psh function on $X\times \bD_I$. A psh ray (emanating from $u_0$) is a psh path on $(0, +\infty)$ (with $\lim_{t\rightarrow 0}u(s)=u_0$). Note in the literature, psh path (resp. psh ray) are also called subgeodesic (resp. subgeodesic ray). 

In the above situation, we also say that $\Phi(s)=\{\psi+u(s)\}$ is a psh path (resp. a psh ray). 
\end{defn}
%\begin{prop}
%Let $U$ denote the geodesic connecting $u_0$ and $u_1$ in $\mathcal{H}(\omega)$ and let $\tilde{U}_\epsilon$ denote the corresponding geodesic in the space $\mathcal{H}(\omega_\epsilon)$. Then map $\epsilon\mapsto U_\epsilon$ is increasing and $U_\epsilon$ decreases to $U$ as $\epsilon\rightarrow 0$.
%\end{prop}
We will use geodesics connecting bounded potentials. 
\begin{prop}[{\cite[Proposition 1.17]{DNG18}}]\label{prop-geod}
Let $u_0, u_1\in \PSH_\bd(\omega)$. Then  
\begin{equation}\label{eq-envelope}
U=\sup\left\{u; u\in \PSH(Z\times\bD_{[0,1]}, p_1^*\omega);\quad U\le u_{0,1} \text{ on } \partial (Z\times \bD_{[0,1]})\right\}.
\end{equation}
is the unique bounded $\omega$-psh function on $Z\times\bD_{[0,1]}$ that is the solution of the Dirichlet problem:
\begin{equation}
(p_1^*\omega+\sddb U)^{n+1}=0 \text{ on } Z\times\bD_{[0,1]}, \quad U|_{Z\times\partial \bD_{[0,1]}}=u_{0,1}.
\end{equation}
%If $u_0, u_1\in \bar{\mcH}(\omega)$, then $U$ is $C^{1,1}$ on $X^{\rm reg}\times \bS$.
\end{prop}
We will call $\Phi=\{\vphi(s)=\psi+U(\cdot, s)\}$ the \textit{geodesic} joining $\vphi_0=\psi+u_0$ and $\vphi_1=\psi+u_1$.

For finite energy potentials $u_0, u_1\in \mcE^1(\omega)$, let $u^j_0, u^j_1$ be bounded smooth $\omega$-psh functions decreasing to $u_0, u_1$ (see Proposition \ref{prop-smapp}). Let $u_t^j$ be the bounded geodesic connecting $u^j_0$ to $u^j_1$. \Blue{Using the expression \eqref{eq-envelope}, we know that} $j\rightarrow u^j_{t}$ is non-increasing. Set:
\begin{equation}
u_t:=\lim_{j\rightarrow+\infty}u^j_t.
\end{equation}
\Blue{We call $U=\{u_t\}$ to be a (finite-energy) geodesic joining $u_0$ to $u_1$. By \cite[Theorem 1.7]{BBJ18},
%\cite[Proposition 4.6]{DNG18},
the map $t\mapsto u_t$ associated to the geodesic $U$ is a continuous map from $[0, 1]$ to $\cE^1$ with respect to the strong topology. }

Generalizing Darvas' result in the smooth case (\cite{Dar15}), the works in \cite{Dar17, DNG18} showed that %any $\vphi_0, \vphi_1\in \mcE^1$ can be connected by a geodesic segment in $\mcE^1$ and 
$\mcE^1$ can be characterized as the metric completion of $\bar{\mcH}(\omega)$ under a \Blue{length metric} $d_1$ which can be defined as follows.  
Fix a log resolution $\mu: Y\rightarrow Z$ and a K\"{a}hler form $\omega_P>0$ on $Y$. Then for any $\epsilon>0$,
\begin{equation}\label{eq-omegaep}
\omega_\epsilon:=\mu^* \omega+\epsilon \omega_P
\end{equation}
is a K\"{a}hler form and one can define Darvas' \Blue{$d_{1,\epsilon}$-metric} on $\bar{\mcH}(Z, \omega_\epsilon)$. Note that $u\in \bar{\mcH}(Z, \omega)$ implies $u\in \bar{\mcH}(Y, \omega_\epsilon)$. One then defines (see \cite[Definition 1.10]{DNG18})
\begin{eqnarray*}
d_1(u_0, u_1)=\liminf_{\epsilon\rightarrow 0}d_{1,\epsilon}(u_0, u_1).
\end{eqnarray*}

It is known that $u_j\rightarrow u$ in $\mcE^1$ under the strong topology %if and only if $I_{\psi+u}(\psi+u_j)\rightarrow 0$, 
if and only if $d_1(u_j, u)=0$.  Moreover in this case the Monge-Amp\`{e}re measures $(\sddb(\psi+u_j))^n$ converges weakly to $(\sddb(\psi+u))^n$ (see \cite[Proposition 2.6]{BBEGZ}).

\subsection{Energy functionals and K\"{a}hler-Einstein metrics}

\Blue{Let $e^{-\psi}$ be again the smooth reference Hermitian metric on $L$ as defined in \eqref{eq-refpsi}.} 
For any $\Blue{\vphi \in\;} \PSH(L)$ such that $\vphi-\psi\in \mcE^1(\omega)$, we use the following well-studied functionals: %$\chi$ is a closed $(1,1)$-current
\begin{eqnarray}
\bf\bfE(\vphi)&:=&\bfE_\psi(\vphi)\nonumber \\
&=&\frac{1}{(n+1)(2\pi)^n V} \sum_{i=0}^n \int_Z (\vphi-\psi) (\sddb\psi)^{n-i}\wedge (\sddb\vphi)^{i},\label{eq-Ephi}\\
{\bf \Lambda}(\vphi)&:=&{\bf \Lambda}_\psi(\vphi)=\frac{1}{(2\pi)^nV}\int_Z (\vphi-\psi)(\sddb \psi)^n \label{eq-Kphi}, \\
%E^\chi(\vphi)&:=&E^\chi_\psi(\vphi)=\frac{1}{n}\sum_{i=0}^{n-1}\int_X \chi\wedge (\vphi-\psi)(\sddb \psi)^{n-1-i}\wedge (\sddb \vphi)^i\\
\bfJ(\vphi)&:=&J_\psi(\vphi)={\bf \Lambda}_\psi(\vphi)-\bfE_\psi(\vphi)\nonumber \\
&=&\frac{1}{(2\pi)^nV}\int_Z (\vphi-\psi)(\sddb\psi)^n-\bfE_\psi(\vphi), \label{eq-Jphi}\\
\bfI(\vphi)&:=&\bfI_\psi(\vphi)=\bfI(\psi, \vphi)= \int_X (\vphi-\psi)\left((\sddb\psi)^n-(\sddb\vphi)^n\right), \label{eq-Iphi}\\
({\bf I}-\bfJ)(\vphi)&=&({\bf I}_\psi-\bfJ_\psi)(\vphi)=\bfE_\psi(\vphi)-\frac{1}{(2\pi)^n V}\int_Z (\vphi-\psi)(\sddb \vphi)^n.
\end{eqnarray}
These functionals first appeared in K\"{a}hler geometry in the smooth setting (see \cite{Tia00}). They have since been well studied in much more generality in \cite{BEGZ} for any big classes on a compact K\"{a}hler manifold. In particular, the formula \eqref{eq-Ephi} (in which we have used the mixed Monge-Amp\`{e}re measures defined in \eqref{eq-mixed})
is valid according to \cite[Corollary 2.18]{BEGZ} after equating the integrals with corresponding integrals on the resolution of singularities (and using the identification \eqref{eq-PSHY}). 
\Blue{We recall some basic inequalities that will be useful later
\begin{lem}\label{eq-lemIJ}
For any $\vphi\in \cE^1(L)$, we have the inequalities:
\begin{align}\label{eq-IJineq}
\frac{1}{n+1}\bfI(\vphi)\le \bfJ(\vphi)\le \frac{n}{n+1} \bfI(\vphi), %\quad \frac{1}{n}\bfJ(\vphi) \le (\bfI-\bfJ)(\vphi)\le n\bfJ(\vphi).
\end{align}
Moreover for any $t\in [0, 1]$, we have:
\begin{align}\label{eq-Dingineq}
\bfJ(t\vphi+(1-t)\psi)\le t^{1+\frac{1}{n}}\bfJ(\vphi).
\end{align}
\end{lem}
The inequality \eqref{eq-IJineq} is well-known (see \cite[6.1]{Tia00} or \cite[section 1.4]{BBEGZ}). The inequality \eqref{eq-Dingineq} is first proved in \cite{Din88} by integrating the following inequality:
\begin{align*}
\frac{d}{dt}\bfJ(t\vphi+(1-t)\psi)&=\frac{1}{(2\pi)^nL^{\cdot n}}\int_X (\vphi-\psi)((\sddb\psi)^n-(\sddb(t\vphi+(1-t)\psi))^n)\\
&=\frac{1}{t}\bfI(t\vphi+(1-t)\psi)\ge \frac{1}{t}(1+\frac{1}{n})\bfJ(t\vphi+(1-t)\psi). 
\end{align*}}
Another property we will need is the monotonicity of ${\bf \Lambda}$ and $\bfE$ functionals:
\begin{equation}\label{eq-monotone}
\vphi_1\le \vphi_2 \quad \Longrightarrow \quad {\bf \Lambda}(\vphi_1)\le {\bf \Lambda}(\vphi_2)\quad \text{ and }\quad \bfE(\vphi_1)\le \bfE(\vphi_2).
\end{equation}
From now on, let $Q$ be a Weil \Blue{$\bQ$-divisor on $Z$ that is not necessarily effective. Assume that $K_Z+Q$ is $\bQ$-Cartier. }
Let $\mu: Y\rightarrow Z$ be a log resolution of singularities of $(Z, Q)$ such that $\mu^{-1}Z^\sing=\sum_k E_k$ is the reduced exceptional simple normal crossing divisor, $Q':=\mu^{-1}_*Q$ is the strict transform of $Q$ and $Q'+\sum_k E_k$ has simple normal crossings. We can write:
\begin{equation}
K_Y+Q'=\mu^*(K_Z+Q)+\sum_k a_k E_k.
\end{equation}
As before, we can assume that the construction of the resolution $\mu$ is functorial. In particular, if a \Blue{connected} group $G$ acts holomorphically on $(Z, Q)$, then we can assume that $\mu$ is $G$-equivariant and the divisors $Q'$ and $E_k$ are all $G$-invariant.
\begin{defn}\label{def-klt}
$(Z, Q)$ is said to have sub-klt singularities if there exists a log resolution of singularities as above such that $a_k>-1$ for all $k$. If $Q$ is moreover effective, then $(Z, Q)$ is said to have klt singularities.
\end{defn}
Fix $\ell_0\in \bN^*$ such that $\ell_0 (K_Z+Q)$ is Cartier. If $\sigma$ is a nowhere-vanishing holomorphic section of the corresponding line bundle over \Blue{an open set} $U$ of $Z$, then there is a pull-back meromorphic volume form on $\mu^{-1}(U)$:
\begin{equation*}
\mu^*\left(\sqrt{-1}^{\ell_0 n^2} \sigma\wedge \bar{\sigma}\right)^{1/\ell_0}=\prod_{i}|z_i|^{2a_i}dV,
\end{equation*}
where $\{z_i\}$ are local holomorphic coordinates satisfying $E_i=\{z_i=0\}$ and $dV$ is a smooth volume for on $Y$. $(Z, Q)$ has sub-klt singularities \Blue{if and only if the above volume form is locally integrable. We assume that this is the case from now on. }
\begin{defn}[{see \cite[section 3]{BBEGZ}}]
Assume $\Blue{L=-K_Z-Q}$ is an ample $\bQ$-line bundle. % for $\lambda>0\in \bQ$.
Let $\vphi \in\cE^1(Z, L)$ be a finite energy psh potential on the $\bQ$-line bundle $L$. \Blue{We define a measure:}
\begin{equation}
\Blue{
\frac{e^{-\vphi}}{|s_Q|^2}:=
\left(\sqrt{-1}^{\ell_0 n^2} \sigma\wedge \bar{\sigma}\right)^{1/\ell_0}{|\sigma^*|_{\ell_0 \vphi}^{2/\ell_0}},}
\end{equation}
where $\sigma^*$ is the dual of $\sigma$ which is a nowhere-vanishing section of $-\ell_0(K_Z+Q)$.
\end{defn}

The Ding- and Mabuchi- functionals on $\mcE^1(Z, L)$ are defined as follows:
\begin{eqnarray}
\bL(\vphi)&=&\bL_{(Z,Q)}(\vphi)=-\log\left(\frac{1}{(2\pi)^nL^{\cdot n}}\int_Y e^{- \vphi}\frac{1}{|s_Q|^2}\right)\label{eq-LB}\\
\bfD(\vphi)&=&\bfD_{(Z,Q),\psi}(\vphi)=\bfD_\psi(\vphi)=-\bfE_\psi(\vphi)+\bL_{(Z,Q)}(\vphi) \label{eq-DB}\\
\bfH(\vphi)&:=&\bfH_{(Z,Q),\psi}(\vphi)=\frac{1}{(2\pi)^n L^{\cdot n}}\int_X \log\frac{(\sddb\vphi)^n}{e^{-\psi}/|s_Q|^2}(\sddb \vphi)^n \label{eq-Hphi}\\
\bfM(\vphi)&:=&\bfM_{(Z,Q),\psi}(\vphi)=\bfM_\psi(\vphi)= \bfH(\vphi)-({\bf I}-\bfJ)_\psi(\vphi) \label{eq-Mphi}.
\end{eqnarray}
\Blue{In the formula \eqref{eq-Hphi}, if $(\sddb\vphi)^n$ is not absolutely continuous with respect to the measure $e^{-\psi}/|s_Q|^2$, then we define $\bfH(\vphi)$ to be $+\infty$. By Jensen's inequality, it is easy to get:
$$
\bfH(\vphi)\ge -\log\left(\frac{1}{(2\pi)^nL^{\cdot n}}\int_X \frac{e^{-\psi}}{|s_Q|^2}\right)>-C>-\infty
$$ 
for a constant $C>0$ independent of $\vphi$, where we used the assumption that $(Z, Q)$ has sub-klt singularities.  
Moreover, because inequality \eqref{eq-IJineq} implies $\bfI-\bfJ\le n\bfJ$, we get the inequality:
\begin{equation}\label{eq-Mge-nJ}
\bfM(\vphi)\ge \bfH(\vphi)-n \bfJ(\vphi)\ge -C-n\bfJ(\vphi). 
\end{equation}
 }

In the rest of this subsection, we will assume $(Z, Q)=(X, D)$ is a log Fano pair. In other words, we assume that $D$ is an effective divisor, $L=-K_X-D$ is an ample $\bQ$-Cartier divisor and $(X,D)$ has klt singularities.
\begin{defn}
A finite energy potential $\Blue{\vphi\in \cE^1}(X, -(K_X+D))$ is a K\"{a}hler-Einstein \Blue{potential} on $(X, D)$ if it satisfies the following equation in the pluripotential sense:
\begin{equation}
(\sddb\vphi)^n=\frac{e^{-\vphi}}{|s_D|^2}.
\end{equation}
\Blue{We then say that the curvature current $\sddb\vphi$ is a K\"{a}hler-Einstein metric.}
\end{defn}
By \cite{BBEGZ}, it is known that any K\"{a}hler-Einstein \Blue{potential} $\vphi$ is automatically bounded, smooth on $X^{\reg}\setminus D$.  
\begin{defn}[{\cite[Definition 1.3]{BBEGZ}}]
A positive measure $\nu$ on $X$ is tame if $\Blue{\nu}$ puts no mass on closed analytic sets and if there is a resolution of singularities $\mu: Y\rightarrow X$ such that the lift $\nu_Y$ of $\nu$ to $Y$ has $L^p$ density for some $p>1$.
\end{defn}
The following compactness result is very important in the variational approach to solving Monge-Amp\`{e}re equations using the pluripotential theory.
\begin{thm}[{\cite[Theorem 2.17]{BBEGZ}}]\label{thm-BBEGZ}
Let $\nu$ be a tame probability measure on $X$. 
%Assume $\chi$ is a smooth K\"{a}hler form.
%Let $p>1$ and suppose $\mu=f \chi^n$ is a probability measure with $f\in L^p(X, \chi^n)$. 
For any $C>0$, the following set is compact in the strong topology:
\[
\left\{u\in \mcE^1(X, \omega); \quad
\sup_{M}u=0, \quad \int_Z \log\frac{\omega_{u}^n}{\nu}\omega^n_{u}\;\Blue{\le}\; C
\right\}.
\]
\end{thm}

\Blue{\subsection{Analytic criterion for the existence under group actions}}

Let $\bG$ be a connected reductive subgroup of $\Aut(X, D)_0$ and $\bT:=C(\bG)_0\cong (\bC^*)^r=((S^1)^r)^\bC$ be the identity component of the center $C(\bG)$.  Any $\xi\in N_\bR$ corresponds to a holomorphic vector field written as $\xi-i J\xi$ where $J$ is the complex structure (on the regular part). In other words, we identify $\xi$ with a real vector field and $J\xi\in \mathfrak{t}$, where $\mathfrak{t}$ is the Lie algebra of $(S^1)^r$. For any $\xi\in N_{\bR}$, let $\Blue{\sigma_\xi}: \bC \rightarrow \bG$ be the one parameter subgroup generated by $\xi$. 
Then we have:
\begin{equation}\label{eq-sigmaxi}
\sigma_\xi(\Blue{\mathfrak{z}}=s+iu)=\exp(s\xi)\cdot \exp(u J\xi).
\end{equation}
If $\xi\in N_\bZ$, then $\sigma_\xi\circ (-\log)=:\hat{\sigma}_\xi: \bC^*\rightarrow \bG$ is a well defined one parameter subgroup. In this paper, we will freely use the change of variables:
\begin{equation}
\bC^*\rightarrow \bR, \quad t\mapsto \Blue{-\log|t|^2}=:s.
\end{equation}

Let $\bK$ be a maximal compact subgroup of $\bG$ containing $(S^1)^r$. Denote by $(\mcE^1)^\bK:=(\mcE^1(L))^\bK$ the set of $\bK$-invariant finite energy psh Hermitian metrics on $L$. For any $\vphi\in (\mcE^1)^\bK$ define:
\begin{equation}\label{eq-JbTvphi}
\bfJ_\bT(\vphi):=\bfJ_{\psi, \bT}(\vphi):=\inf_{\sigma\in \bT}\bfJ_\psi(\sigma^*\vphi).
\end{equation}
\begin{lem}[see {\cite[Lemma 1.9]{His16b}}]\label{lem-infobt}
For any $\vphi\in\; (\cE^1)^\bK$, The function $\sigma\mapsto \bfJ_\psi(\sigma^*\vphi)$ defined on $\bT\cong N_\bR\times (S^1)^r$ is $(S^1)^r$ invariant, convex and proper. As a consequence there always exists $\sigma\in \bT$ that achieves the infimum.
\end{lem}
\begin{proof}
\Blue{
Because $\vphi$ is $\bK$-invariant and $\bK$ contains $(S^1)^r$, $\vphi$ is also $(S^1)^r$-invariant. 
For any $\sigma=\exp(\xi)\exp(i\xi')\in \bT$ with $\xi, \xi'\in N_\bR$, we have $\sigma^*\vphi=\exp(\xi)^*\vphi$.   
So $\sigma\mapsto {\bfJ}_\psi(\sigma^*\vphi)=\bfJ_\psi(\sigma_\xi(1)^*\vphi)$ can be seen as a function of $\xi\in N_\bR\cong \bR^r$. For its convexity, see Proposition \ref{prop-convex1}. To verify its properness, we just need to show the following slope is positive for any $\xi\neq 0\in N_\bR$: 
$$
\bfJ'^\infty(\{\sigma_\xi(s)^*\vphi\}):=\lim_{s\rightarrow+\infty} \frac{\bfJ_\psi(\sigma_\xi(s)^*\vphi)}{s}.
$$ 
Here $\psi$ is a smooth psh potential defined in \eqref{eq-refpsi}. 
We now claim that $a:=\bfJ'^\infty(\{\sigma_\xi(s)^*\vphi\})=\bfJ'^\infty(\{\sigma_\xi(s)^*\psi\})=:b$. Assuming this claim, we just need to prove the second slope is positive: $b>0$. 
Again it is well-known that $s\mapsto \bfJ_\psi(\sigma_\xi(s)^*\psi)=:f(s)$ is a convex function (see Proposition \ref{prop-convex1}). By convexity and $f(0)=0$, it has a positive slope at $+\infty$ as long as it takes a positive value. It is well-known that for any $\vphi'\in \cE^1$, $\bfJ(\vphi')$ is non-negative and is equal to 0 only if $\vphi'-\psi$ is a constant. 
So we just need to show that $\sigma_\xi(s)^*\psi-\psi$ is not a constant function for some $s\in \bR$. To see this,  we note that
$
\left.\frac{d}{ds}\right|_{s=0}\sigma_\xi(s)^*\psi$ is a Hamiltonian function of $\xi$ with respect to the smooth K\"{a}hler metric $\sddb\psi$, which can not be constant unless $\xi$ is 0 (since we assume that the $\bT$-action is faithful). So, for $s$ small enough, $\sigma_\xi(s)^*\psi-\psi$ is not a constant function either. }

\Blue{
Finally we verify the above claim.  By using the definition of $\bfJ$ in \eqref{eq-Jphi} and the co-cycle property of $\bfE$, we get the decomposition:
\begin{align}\label{eq-lemJslope}
&\bfJ_\psi(\sigma_\xi(s)^*\vphi)-\bfJ_\psi(\sigma_\xi(s)^*\psi)=\int_X (\sigma^*\vphi-\sigma^*\psi)(\sddb\psi)^n-\bfE_{\sigma^*\psi}(\sigma^*\vphi)
\end{align}
where, for simplicity of notation, we write $\sigma=\sigma_\xi(s)$.
We need to show that the slope of the left-hand-side at $s=+\infty$ is equal to 0. 
First, by change of variable formula, $\bfE_{\sigma^*\psi}(\sigma^*\vphi)=\bfE_\psi(\vphi)$ does not depend on $s$ and hence its slope is 0. 
If the Hermitian metric $e^{-\vphi}$ is smooth, then we easily see that the first term on the right-hand-side in \eqref{eq-lemJslope} is a bounded function of $s$. For general $e^{-\vphi}$, we further decompose the first term as:
\begin{align}\label{eq-lemJslope2}
&\int_X (\sigma^*\vphi-\sigma^*\psi)((\sddb\psi)^n-(\sddb\sigma^*\psi)^n)+\int_X (\sigma^*\vphi-\sigma^*\psi)(\sddb\sigma^*\psi)^n.
\end{align}
The second term on the right does not dependent on $s$ by change of variables. For the first term, we can use the inequality proved in \cite[Lemma A.1]{BBJ18} (see also \cite[Lemma 1.9]{BBEGZ}) to know that its absolute value is bounded by
$$
\bfI(\sigma^*\vphi, \sigma^*\psi)^{\frac{1}{2^n}}\bfI(\psi, \sigma^*\psi)^{\frac{1}{2^n}}\cdot \max\{\bfI(\psi, \sigma^*\vphi), \bfI(\psi, \sigma^*\psi)\}^{1-\frac{1}{2^{n-1}}}.
$$
By using $\bfI(\sigma^*\vphi, \sigma^*\psi)=\bfI(\vphi, \psi)$, it is easy to see that the above quantity is bounded by $C s^{1-\frac{1}{2^n}}$ for some constant independent of $s$. So if we divide the first term in \eqref{eq-lemJslope2}
by $s$ and let $s\rightarrow+\infty$, we see that its slope is also equal to 0. Combining the above discussions, the proof is then completed. }
\end{proof}
\Blue{
To state the following result, we first introduce the Futaki invariant. For any $\xi\in N_\bR$, let $V_\xi$ be the corresponding holomorphic vector field. The canonical lift of $V_\xi$ on $L=-(K_X+D)$ corresponds to a (Hamiltonian) function $\theta_\xi(\psi)$ that is defined as:
\begin{equation}
\theta_\xi(\psi):=\left.\frac{d}{ds}\right|_{s=0}\sigma_\xi(s)^*\psi:=e^{\psi}\left.\frac{d}{ds}\right|_{s=0}\frac{d}{ds}\sigma_\xi(s)^*e^{-\psi}
\end{equation}
%\begin{equation}
%\theta_\xi(\psi)=-\frac{\mathfrak{L}_{V_\xi}e^{-\psi}}{e^{-\psi}}. %=\sum_i V_\xi^i \partial_{z_i}\psi.
%\end{equation}
%Here we think of the smooth Hermitian metric $e^{-\psi}$ in \eqref{eq-refpsi} as a smooth volume form on $X$ and $\mathfrak{L}_{V_\xi}$ denotes the Lie derivative (on the regular part).  Then $\theta_\xi(\psi)$ is the Hamiltonian function of $\xi$:
and satisfies $\iota_{V_\xi}\sddb \psi=\sqrt{-1}\bar{\partial}\theta_\xi(\psi)$. Define the Futaki invariant:
\begin{equation}\label{eq-Futana}
\Fut(\xi)=-\frac{1}{(2\pi)^n L^{\cdot n}}\int_X \theta_\xi(\psi)(\sddb\psi)^n.
\end{equation}
See \eqref{eq-CWFut} for an algebraic definition of this invariant.  
\begin{lem}
Let $\bfF$ be either $\bfD$ or $\bfM$. Fix $\vphi\in (\cE^1)^\bK$ and $\xi\in N_\bR$. Then for any $s\in \bR$, we have $\bfF(\sigma_\xi(s)^*\vphi)=\bfF(\vphi)-s\cdot (2\pi)^n L^{\cdot n} \cdot \Fut(\xi)$. 
\end{lem}
\begin{proof}
First note that $\bfL(\vphi)=-\log\left(\frac{1}{(2\pi)^nL^{\cdot n}}\int_X e^{-\vphi}/|s_Q|^2\right)$ is invariant under the $\bT$-action:
$\bfL(\sigma^*\vphi)=\bfL(\vphi)$ for any $\sigma\in \bT$. For the $\bfE$ term, by the cocycle condition, with $\sigma=\sigma_\xi(s)$, 
\begin{align*}
\bfE(\sigma^*\vphi)&=\bfE_\psi(\sigma^*\vphi)=\bfE_\psi(\sigma^*\psi)+\bfE_{\sigma^*\psi}(\sigma^*\vphi)=\bfE_\psi(\sigma^*\psi)+\bfE_\psi(\vphi). 
\end{align*}
We then have the identity:
\begin{align*}
\frac{d}{ds}\bfE_\psi(\sigma_\xi(s)^*\psi)&=\int_X \sigma^*\theta_\xi(\psi) (\sddb \sigma^*\psi)^n=\int_X \theta_\xi(\psi)(\sddb\psi)^n\\
&=-(2\pi)^nL^{\cdot n} \cdot \Fut(\xi). 
\end{align*}
This clearly implies the wanted identity for $\bfD=-\bfE+\bfL$. To prove the identity for $\bfM$, note that since $(\bfI-\bfJ)(\vphi)=-\int_X (\vphi-\psi)(\sddb\vphi)^n+\bfE_\psi(\vphi)$, $\bfM$ can be re-written as 
\begin{align*}
\bfM_\psi(\vphi)&=\int_X \log \frac{(\sddb\vphi)^n}{e^{-\psi}/|s_Q|^2}(\sddb\vphi)^n-\int_X \log\frac{e^{-\vphi}/|s_Q|^2}{e^{-\psi}/|s_Q|^2}(\sddb\vphi)^n-\bfE_\psi(\vphi)\\
&=\int_X \log \frac{(\sddb\vphi)^n}{e^{-\vphi}/|s_Q|^2}(\sddb\vphi)^n-\bfE_\psi(\vphi)
\end{align*}
The first term on the right is again invariant under $\bT$-action by the change of variable formula. The last term has been dealt with above. 
\end{proof}
By this lemma, if $\bfF\in \{\bfD, \bfM\}$ is bounded from below on $(\cE^1)^\bK$, then $\Fut\equiv 0$ on $N_\bR$ and $\bfF$ is invariant under the $\bT$-action. 
We now introduce the stronger condition to guarantee the existence of K\"{a}hler-Einstein metrics. 
}
%Recall the following definition:
\begin{defn}[\cite{DR17, His16b}]\label{def-funcoercive}
We say that the energy $\bfF\in \{\bfD, \bfM\}$ is $\bG$-coercive %(usually called coercive in the literature)
 if there exists $\gamma>0$, $C>0$ such that for any $\vphi\in (\mcE^1)^\bK$ we have:
\begin{equation}\label{eq-FGproper}
\bfF(\vphi) \ge \gamma\cdot  \bfJ_\bT(\vphi)-C.
\end{equation}
\end{defn}

\begin{thm}[{\cite{BBEGZ}, \cite{DR17}, \cite{Dar17}, \cite[Theorem 3.4]{His16b}}]\label{thm-analytic}
Let $(X, D)$ be a log Fano pair. Let $\bG$ be a connected reductive subgroup of $\Aut(X, D)_0$, and set $\bT=C(\bG)_0$ and $\bK\subset \bG$ as before.
Consider the following conditions:
\begin{enumerate}[(1)]
\item The Ding energy is $\bG$-coercive. 
\item The Mabuchi energy is $\bG$-coercive.
\item $(X, D)$ admits a $\bK$-invariant K\"{a}hler-Einstein metric.
\end{enumerate}
Then  \Blue{we have the implications (1) $\Rightarrow$ (2) $\Rightarrow$ (3)}. 

Moreover, if we assume that $\Aut(X, D)_0$ is reductive and set $\bG=\Aut(X, D)_0$, then all of the above conditions are equivalent.
\end{thm}
The existence part of the above result can be derived from the work in \cite{BBEGZ, His16b}. For the reader's convenience, we sketch the proof of $(2)\Rightarrow (3)$ and refer the details to \cite{BBEGZ,Dar17, DNG18}. Because Mabuchi energy is bigger than the Ding energy, $(1)\Rightarrow (2)$ also follows.
\begin{proof}[Sketch of the proof of $(2) \Rightarrow (3)$]
 Assume that $\bfM$ is $\bG$-coercive. Then $\bfM$ is bounded from below over $(\mcE^1)^\bK$. Choose a sequence of potentials $\vphi_j\in (\mcE^1)^\bK$ such that $\bfM(\vphi_j)\rightarrow \inf_{(\mcE^1)^\bK}\bfM(\vphi)$. Then $\bfJ_{\bT}(\vphi_j)\le C$ independent of $j$. By Lemma \ref{lem-infobt} there exists $\sigma_j\in \bT$ such that $\tilde{\vphi}_j:=\sigma_j^*\vphi_j$ satisfies $\bfJ(\tilde{\vphi}_j)=\bfJ_\bT(\vphi_j)$. Clearly $\tilde{\vphi}_j\in (\cE^1)^\bK$. Moreover we can assume that $\sup(\tilde{\vphi}_j-\psi)=0$. 
  
If $\bfM$ is $\bG$-coercive, then it is bounded from below on $(\cE^1)^\bK$. 
As mentioned above, this implies that $\bfM$ is invariant under the $\bT$-action on $(\cE^1)^\bK$. 
Moreover, \Blue{using the inequality $\bfM=\bfH-(\bfI-\bfJ)\ge \bfH-n\bfJ$ (see \eqref{eq-Mge-nJ})}, we know that $\bfH(\tilde{\vphi}_j)$ is uniformly bounded from above. So by the compactness Theorem \ref{thm-BBEGZ}, $\tilde{\vphi}_j$ converges strongly to $\vphi_\infty\in (\cE^1)^\bK$. By the lower semicontinuity of $\bfM$ under strong convergence (see \cite[Lemma 4.3]{BBEGZ}), we know that $\vphi_\infty$ is a minimizer of $\bfM$ over $(\cE^1)^\bK$. Now we can easily adapt \cite[Proof of Theorem 4.8]{BBEGZ} to the $\bK$-invariant setting conclude that the $\vphi_\infty$ is a $\bK$-invariant K\"{a}hler-Einstein \Blue{potential}.
 \end{proof}

The last statement of Theorem \ref{thm-analytic} follows from the works of Darvas and Hisamoto via the general framework by Darvas-Rubinstein (in \cite{DR17}) for proving Tian's properness conjecture from \cite{Tia12}. \Blue{Note that the reductivity of $\Aut(X, D)_0$ is a necessary condition for the existence of K\"{a}hler-Einstein metrics (see \cite[Theorem 5.2]{BBEGZ}).} %Note that although Hisamoto's work uses $\bfJ_{C(\bG)}$ instead of $\bfJ_{C(\bG)_0}$, the properness conditions using these two norms will turn out to be equivalent. 
Here we prove a more general result.
\begin{thm}\label{thm-torusproper}
Let $(X, D)$ be a log Fano pair. \Blue{Assume that} $\bG$ is a connected reductive subgroup of $\Aut(X, D)_0$ that contains a maximal torus of $\Aut(X,D)_0$. Then all of the conditions in the above theorem are equivalent.
\end{thm}
\begin{proof}
We just need to show that condition (3) implies (1). For this, we use Darvas-Rubinstein's principle from \Blue{\cite[Theorem 3.4]{DR17}}. In their notations (see also \cite{Dar17}), we consider the data
$$
\mathcal{R}=(\mcE^1)^\bK\cap L^\infty(X), \quad
\overline{\mathcal{R}}=(\mcE^1)^\bK,  \quad \mcM=\{\text{$\bK$-invariant K\"{a}hler-Einstein \Blue{metrics} on } (X, D)\},
$$
where $\bK\subset \bG$ is a maximal compact subgroup. \Blue{One easily verifies} that the data $(\mcR, d_1, \bfD, \bT)$ satisfies the properties (P1)-(P7) in \cite[Hypothesis 3.2]{DR17} except for (P5) which needs more argument. The property (P5) means that the space of $\bK$-invariant K\"{a}hler-Einstein metrics is homogeneous under the action of $\bT$ where $\bT$ is the identity component of the center of $\bG$.

Let $\omega_i, i=1,2$ be any two $\bK$-invariant K\"{a}hler-Einstein metrics and set $$K_i={\rm Isom}(\omega_i)_0=\{g\in \Aut(X, D)_0; g^*\omega_i=\omega_i\}.$$ Then by \cite[section 5]{BBEGZ}, $K_i, i=1,2$ are maximal compact subgroups of $\Aut(X, D)_0$. 
Because $\omega_i$ is $\bK$-invariant, we know that $\bK\subseteq K_1\cap K_2$. By assumption, $\bK$ contains a maximal compact torus of $\bG$. By Proposition \ref{prop-Yu3}, $K_2=t^{-1}K_1 t$ for some $t\in \bT=C(\bG)_0$. 

On the other hand, by Berndtsson's theorem (see \cite[Appendix C]{BBEGZ}), there exists $f\in \Aut(X, D)_0=:\mathfrak{G}$ satisfying $\omega_2=f^* \omega_1$. So we get $f^{-1}K_1f=K_2=t^{-1} K_1 t$. This implies $f t^{-1}\in N_{\mathfrak{G}}(K_1)$. By Proposition \ref{prop-Yu1} (see also \cite[Proposition 2.13]{His19}), $f t^{-1}\in K_1 C(\mathfrak{G})_0$. So $f=k_1 \cdot t \cdot t_1=: k_1\cdot t'$ for $k_1\in K_1$, $t\in \bT$, $t_1\in C(\mathfrak{G})_0\subset \bT$ and $t':=t\cdot t_1\in \bT$. So we get $\omega_2=f^*\omega_1=t'^* k_1^*\omega_1=t'^* \omega_1$. We are done.
\end{proof}
\Blue{
\begin{rem}
Combining the above two theorems, we see that as long as $\bG$ contains a maximal torus of $\Aut(X,D)_0$, the $\bG$-coercivity condition in Definition \ref{def-funcoercive} 
does not depend on the choice of the maximal compact group of $\bG$. In fact, if $\bfM$ is $\bG$-coercive for a maximal compact subgroup $\bK_1$, then by Theorem \ref{thm-analytic} there is a $\bK_1$-invariant K\"{a}hler-Einstein metric $\omega_1$. If $\bK_2$ is another maximal compact subgroup of $\bG$, there exists $g\in \bG$ such that $\bK_2=g^{-1}\bK_1g$. Then it is easy to see that $g^*\omega_1$ is a $\bK_2$-invariant K\"{a}hler-Einstein metric. By Theorem \ref{thm-torusproper} we know that $\bfM$ is indeed $\bG$-coercive for the choice $\bK_2$. 
\end{rem}
}

\subsection{Valuations on $T$-varieties}\label{sec-NRact}

Let $\bT$ be \Blue{an algebraic torus acting faithfully on $Z$}. By the structure theory of $\bT$-varieties , $Z$ can be described using the language of divisorial fans (see \cite[Theorem 5.6]{AHS08}). For us, we just need to know that
$Z$ is birationally a torus fibration over the Chow quotient of $Z$ by $\bT$ which will be denoted by $Z/\!/\bT$. \Blue{Here the Chow quotient is obtained by taking an inverse limit for a system of GIT quotients (see \cite[section 3]{Suss13} for a discussion and references on the relation between these torus quotients).} 
As a consequence the function field $\bC(Z)$ is the quotient field of the Laurent polynomial algebra:
\begin{equation}
\bC(Z\sslash \bT)[M_\bZ]=\bigoplus_{\alpha\in M_\bZ}\bC(Z\sslash \bT)\cdot 1^{\alpha}.
\end{equation}
Given a valuation $\nu$ of the functional field $\bC(Z\sslash \bT)$ and a vector $\lambda\in N_\bR$, we obtain a valuation (\cite[page 236]{AHS08}): 
\begin{equation}
v_{\nu, \lambda}: \bC[Z\sslash \bT][M_\bZ]\rightarrow \bR, \quad \sum_{i} f_i\cdot 1^{\alpha_i}\quad \mapsto\quad \min\left(\nu(f_i)+\la \alpha_i, \lambda \ra \right).
\end{equation}
In particular, for any $\xi\in N_\bR$, $\xi$ determines a valuation which will be denoted by $\wt_\xi:=v_{{\rm triv},\xi}$:
\begin{equation}
\wt_\xi\left(\sum_i f_i \cdot 1^{\alpha_i}\right)=\min_i \la\alpha_i, \xi\ra.
\end{equation}

The vector space $N_\bR$ acts on $\Val(Z)^\bT$ in the following natural way. If $v=\nu_{\nu,\lambda}$, then
\begin{equation}
\xi\circ v= \xi \circ v_{\nu, \lambda}= v_{\nu, \lambda+\xi}=:v_\xi.
\end{equation}
\Blue{
More explicitly, if $f\in \bC(Z)_\alpha$ with $\alpha=(\alpha_1,\dots, \alpha_r)\in \bZ^r$ i.e. if 
$f$ satisfies,  for any $\mathbf{t}=(t_1,\dots, t_r)\in (\bC^*)^r$, 
\begin{equation}\label{eq-falpha}
f\circ \mathbf{t}^{-1}=\mathbf{t}^{\alpha} f\quad \text{ where } \mathbf{t}^\alpha=\prod_{i=1}^r t_i^{\alpha_i},
\end{equation} 
then we have:
\begin{equation}
v_\xi(f)=v(f)+\la \alpha, \xi\ra. 
\end{equation}
%Moreover, under this action, $N_\bQ$ sends $(\Zdiv)^\bT$ to $(\Zdiv)^\bT$. 
}

\subsection{K-stability and Ding-stability}
\subsubsection{Stability via test configurations}

In this section we recall the definition of test configurations and stability of log Fano varieties.
\begin{defn}[{\cite{Tia97, Don02}, see also \cite{LX14}}]\label{defn-TC}
\Blue{Let $Z$ be a normal projective variety and $Q$ be a $\bQ$-divisor. Assume that $L=-(K_Z+Q)$ is an ample $\bQ$-Cartier divisor and $(Z, Q)$ has at worst sub-klt singularities. }
\begin{enumerate}[(1)]
\item A test configuration for $(Z, L)$, denoted by $(\mcZ, \mcL)$, consists of the following data
\begin{itemize}
\item A \Blue{normal} variety $\mcZ$ and a flat projective morphism $\pi: \mcZ\rightarrow \bC$;
\item A $\pi$-semiample $\bQ$-line bundle $\mcL$;
\item A $\bC^*$-action on $(\mcX, \mcL)$ that makes $\pi$ equivariant and induces a $\bC^*$-equivariant isomorphism $ (\mcZ, \mcL)|_{\pi^{-1}(\bC\backslash\{0\})}\cong (Z, L)\times \bC^*$.
\end{itemize}
Let $\mcQ=\mcQ_{(\mcZ, \mcL)}$ denote the closure of $Q\times\bC^*$ in $\mcZ$ under the inclusion $Q\times\bC^*\subset Z\times\bC^*\; {\cong}\;\mcZ\times_\bC\bC^*$. 
\Blue{If we want to emphasize the boundary divisors, we also say that $(\mcZ, \mcQ, \mcL)$ is a test configuration for $(Z, Q, L)$.} 

Denote by $\bar{\pi}: (\bar{\mcZ}, \bar{\mcQ}, \bar{\mcL})\rightarrow \bP^1$ the canonical compactification of $(\mcZ, \mcQ, \mcL)\rightarrow \bC$  adding a trivial fiber over $\{\infty\}\in \bP^1$. 

\item
%A test configuration is called normal if $\mcZ$ is a normal variety. We will always consider \Blue{test configurations} in this paper. 

A test configuration is called a special test configuration, if the following conditions are satisfied:
\Blue{\begin{itemize}
\item $(\mcZ, \mcZ_0+\mcQ)$ has plt singularities. \Blue{In particular, the central fibre $\mcZ_0=\pi^{-1}(\{0\})$ is irreducible and normal. }
%\item $\mcZ$ is normal, and $\mcZ_0=\pi^{-1}(0)$ is an irreducible normal variety;
\item $\mcL=-(K_{\mcZ/\bC}+\mcQ)$, which is thus a $\pi$-ample $\bQ$-line bundle.
%and the $\bC^*$-action on $\mcL_0=-r^{-1}K_{\mcX_0}\rightarrow \mcX_0$ is the induced by the action on $\mcX_0$.
\end{itemize}
}

A test configuration $(\mcZ, \mcQ, \mcL)$ is called dominating if there exists a $\bC^*$-equivariantly birational morphism $\rho: (\mcZ, \mcQ) \rightarrow (Z, Q)\times\bC$.

Two test configurations $(\mcZ_i, \mcQ_i, \mcL_i), i=1, 2$ are called equivalent, if there exists a test configuration $(\mcZ_3, \mcQ_3)$ that $\bC^*$-equivariantly dominates both test configurations via $q_i: (\mcZ_3, \mcQ_3)\rightarrow (\mcZ_i, \mcQ_i)$, $i=1,2$ and satisfies $q_1^*\mcL_1=q_2^*\mcL_2$. \Blue{See \cite{BHJ17} for more details. Note that, by taking fiber product with the trivial test configuration, any test configuration is equivalent to a dominating test configuration. } 

\Blue{A test configuration $(\mcZ, \mcL)$ is called a product test configuration (for the pair $(Z, Q)$) if there is a $\bC^*$-equivariant isomorphism $(\mcZ, \mcQ, \mcL)\cong (Z\times\bC, Q\times\bC, p_1^*L)$ and the $\bC^*$-action on the right-hand-side is given by the product of a $\bC^*$-action on $(Z, Q, L)$ with the standard multiplication on  $\bC$. }

\item
For any \Blue{test configuration} $(\mcZ, \mcQ, \mcL)$ of $(Z, Q, \Blue{L=-(K_Z+Q)})$, define the divisor 
$\Delta_{(\mcZ, \mcQ, \mcL)}$ to be the $\bQ$-divisor supported on $\mcZ_0$ that is given by:
\begin{equation}
\Delta:=\Delta_{(\mcZ, \mcQ, \mcL)}=-K_{\mcZ/\bC}-\mcQ-\mcL.
\end{equation}
%\item By a special degeneration we mean a special test configuration without the data $\eta$.
\end{enumerate}
\end{defn}
Set $V=(L^{\cdot n})$ to be the volume. For any (dominating) \Blue{test configuration} $(\mcZ, \mcL)$, we \Blue{attach the non-Archimedean invariants following the notation of Boucksom-Hisamoto-Jonsson in \cite{BHJ17}}:
\begin{eqnarray}
\bfE^\NA(\mcZ,  \mcL)&=&%\bfE^\NA_L(\phi)%=\frac{(\phi^{n+1})}{n+1}=
\frac{1}{V}\frac{\left(\bar{\mcL}^{\cdot n+1}\right)}{n+1},\label{eq-ENA} \\
{\bf \Lambda}^\NA(\mcZ, \mcL)&=&\frac{1}{V} \left(\bar{\mcL}\cdot \rho^*(L\times\bP^1)^{\cdot n}\right), \label{eq-KNA}\\
\bfJ^\NA(\mcZ, \mcL)&=&%J_L^\NA(\phi)=
%\Phi\cdot \Phi_\triv^n-\frac{(\phi^{n+1})}{n+1} \nonumber \\ &=&= 
\frac{1}{V} \left(\bar{\mcL}\cdot \rho^*(L\times\bP^1)^{\cdot n}\right)-\frac{1}{V} \frac{\left(\bar{\mcL}^{\cdot n+1}\right)}{n+1}, \label{eq-JNA} \\
 %I^\NA(\mcZ,  \mcL)&=&%I_L^\NA(\phi)=
%\left(\phi\cdot \Phi_\triv^n-\Phi^{n+1}+\Phi_\triv\cdot (\phi^n)\right)\nonumber\\&=&
 %\frac{1}{V}\left(\bar{\mcL}\cdot \rho^*(L\times\bP^1)^{\cdot n}\right)-\frac{1}{V}\left(\bar{\mcL}^{\cdot n+1}\right)+\frac{1}{V}\left(\rho^*(L\times\bP^1)\cdot \bar{\mcL}^{\cdot n}\right), \label{eq-INA} \nonumber \\
%(I-J)^\NA(\mcZ,  \mcL)&=&\frac{1}{V}\bar{\mcL}^{\cdot n}\cdot \rho^*(L\times\bP^1)-\frac{1}{V}\frac{n}{n+1}\bar{\mcL}^{\cdot n+1}, \\
\bL^\NA(\mcZ, \mcL)&:=&\bL^\NA(\mcZ, \mcQ, \mcL)=\lct(\mcZ, \mcQ+\Delta; \mcZ_0)-1,\label{eq-LNA}\\
\bfD^\NA(\mcZ, \mcL)&:=&\bfD^\NA(\mcZ, \mcQ, \mcL)=\frac{- \bar{\mcL}^{\cdot n+1}}{(n+1)V}+\left(\lct(\mcZ, \mcQ+\Delta; \mcZ_0)-1\right), \label{eq-bfDTC}\\
%\CM(\mcZ, \mcL)&:=&\CM(\mcZ, \mcQ, \mcL)=\frac{1}{(n+1)V}\left(n \bar{\mcL}^{\cdot n+1}+(n+1) \bar{\mcL}^{\cdot n}\cdot K_{(\bar{\mcZ}, \bar{\mcQ})/\bP^1}\right). \label{eq-defCM}
\Blue{\bfM^\NA(\mcZ, \mcL)}&:=&\bfM^\NA(\mcZ, \mcQ, \mcL)=\frac{1}{(n+1)V}\left(n \bar{\mcL}^{\cdot n+1}+(n+1) \bar{\mcL}^{\cdot n}\cdot K^{\log}_{(\bar{\mcZ}, \bar{\mcQ})/\bP^1}\right). \label{eq-MNA}
\end{eqnarray}
\Blue{where in \eqref{eq-LNA}
$
\lct(\mcZ, \mathcal{Q}+\Delta; \mcZ_0)=\sup\{t; (\mcZ, \mcQ+\Delta+t\mcZ_0) \text{ is sub-log-canonical}\}
$
and in \eqref{eq-MNA}
$
K^{\log}_{(\bar{\mcZ},\bar{\mcQ})/\bP^1}=K_{\bar{\mcZ}}+\bar{\mcQ}+\mcZ_0^{\rm red}-\pi^*(K_{\bP^1}+\{0\})$.}
\begin{rem}\label{rem-calLNA}
There is an explicit and useful formula for $\bL^\NA(\mcZ, \mcQ, \mcL)$.
Choose a $\bC^*$-equivariant log resolution $\pi_\mcZ: \mcU\rightarrow (\mcZ, \mcQ)$ such that $(\mcZ, \mcZ_0+\pi_\mcZ^{-1}(\mcQ))$ is a log smooth pair. Write:
\begin{equation*}
K_{\mcU}=\pi_\mcZ^*(K_{\mcZ}+\mcQ)+\sum_i a_i E_i+\sum_j a'_j E'_j, \quad \pi^*\mcZ_0=\sum_i b_i E_i, \quad \pi^*\Delta=\sum_i c_i E_i,
\end{equation*}
where \Blue{$\{E_i\}_i$ are irreducible components of $\mcU_0$}. 
Then we have the following formula (see \cite[Theorem 3.11]{Berm15} and \cite[Proposition 7.29]{BHJ17}):
\begin{equation}
\bL^\NA(\mcZ, \mcL)=\min_i \frac{a_i-c_i+1}{b_i}-1.
\end{equation}
In particular, this means that $\lct(\mcZ, \mcQ+\Delta; \mcZ_0)$ is calculated by some $E_i$ whose center over $\cZ$ is supported on $\mcZ_0$.

\end{rem}

The following result is now well known:
\begin{prop}[{see \cite{Berm15, BHJ19, PRS08}}]\label{prop-BHJslope}
\Blue{Let $(Z, Q)$ be a log pair with at worst sub-klt singularities. Assume $L=-(K_X+D)$ is an ample $\bQ$-Cartier line bundle. }
Assume that $(\mcZ, \mcQ, \mcL)$ is a \Blue{test configuration} for $(Z, Q, L)$. Let $\Blue{e^{-\Phi}=\{e^{-\vphi(t)}\}}$ be a bounded and psh Hermitian metric on $\mcL$. Then the following limit holds true:
\begin{equation}\label{eq-BHJslope}
\bfF'^\infty(\Phi):=\lim_{t\rightarrow 0} \frac{\bfF(\vphi(t))}{-\log|t|^2}=\bfF^\NA(\mcZ,  \mcL),
\end{equation}
where the energy $\bfF$ is any one from $\{\bfE, {\bf \Lambda}, \bfJ, \bfL, \bfD\}$.
\end{prop}
\Blue{The slope formula \eqref{eq-BHJslope} for $\bfF\in \{\bfE, \Lam, \bfJ\}$ is proved in \cite[Theorem A]{BHJ19} using the method of Deligne pairings (see \cite{PRS08}). For $\bfF\in \{\bfL, \bfD\}$, the slope formula is proved in \cite[Theorem 1.3]{Berm15}. 
From this formula we derive a result that will be used later:
\begin{cor}
Let $(\mcZ, \mcL)$ be a test configuration for $(Z, L)$ that dominates the trivial test configuration $(Z\times \bC, p_1^*L)$ by a morphism $\rho$. Then for any $\delta\in [0, 1]$, the following inequality holds true:
\begin{equation}\label{eq-NADingineq}
\bfJ^\NA(\mcZ, \delta \mcL+(1-\delta)\rho^*L_{\bC})\le \delta^{1+\frac{1}{n}}\bfJ^\NA(\mcZ, \mcL).
\end{equation}
\end{cor}
\begin{proof}
Choose a bounded psh Hermitian metric $e^{-\Phi}$ on $\mcL$. Set $e^{-\Phi_\delta}=e^{-(1-\delta)\Phi-\delta \rho^*p_1^*\psi}=\{e^{-\vphi_\delta(t)}\}$. 
By \eqref{eq-Dingineq}, we get:
\begin{equation*}
\bfJ(\vphi_\delta(t))\le \delta^{1+\frac{1}{n}}\bfJ(\vphi).
\end{equation*}
Dividing $-\log|t|^2$ on both sides and letting $t\rightarrow 0$, we use \eqref{eq-BHJslope} to get the inequality \eqref{eq-NADingineq}. 
\end{proof}
\begin{rem}
The inequality \eqref{eq-NADingineq} is also derived in \cite{BoJ18a} by using non-Archimedean pluripotential theory. 
\end{rem}
}

\begin{defn}\label{defn-unistability}

\begin{enumerate}[(1)]

\item
$(Z, Q)$ is called {\it uniformly K-stable} if there exists $\gamma>0$ such that $\Blue{\bfM^\NA}(\mcZ, \mcL)\ge \gamma\cdot \bfJ^\NA(\mcZ, \mcL)$ for any \Blue{test configuration} $(\mcZ, \mcL)$ of $(Z, L)$.

%\item %The pair $(X, -K_{X})$ 
%$X$ is called K-semistable if $\CM(\mcX, \mcL)\ge 0$ for any \Blue{test configuration} $(\mcX, \mcL)/\bC^1$ of $(X, r^{-1}K_X^{-1})$.
%\item %The pair is $(X, -K_{X})$ 
%$X$ is called K-polystable if $\CM(\mcX, \mcL)\ge 0$ for any \Blue{test configuration} $(\mcX, \mcL)/\bC^1$ of $(X, r^{-1}K_X^{-1})$, and the equality holds if and only if $\mcX\cong %X\times\bC^1$.
\item
$(Z, Q)$ is called {\it uniformly Ding-stable} if there exists $\gamma>0$ such that $\bfD^\NA(\mcZ, \mcL)\ge \gamma\cdot \bfJ^\NA(\mcZ,  \mcL)$ for any \Blue{test configuration} $(\mcZ, \mcL)$ of $(Z, L)$. 

For convenience, we will call $\gamma$ to be a slope constant.

\end{enumerate}
\end{defn}
For any special test configuration $(\mcZ^s,  \mcL^s)$, its \Blue{$\bfM^\NA$ invariant} coincides with its $\bfD^\NA$ invariant, which coincides with the Futaki invariant which is defined in \eqref{eq-Futana} (see also \eqref{eq-CWFut}):
\begin{equation}\label{eq-CMstc}
\bfD^\NA(\mcZ^s, \mcL^s)=
\Blue{\bfM}^\NA(\mcZ^s,  \mcL^s)=-\frac{(-K_{(\overline{\mcZ^s}, \overline{\mcQ^s})/\bP^1})^{\cdot n+1}}{(n+1)L^{\cdot n}}=\Fut_{(\mcZ^s_0, \mcQ^s_0)}(\eta).
\end{equation}

By the work in \cite{BBJ15, Fuj19a} (see also \cite{LX14}), to test uniform K-stability, one only needs to test on special test configurations. As a consequence,
 \begin{thm}[\cite{BBJ15, Fuj19a}]\label{thm-BBJ}
For a log Fano pair $(X, D)$, $(X, D)$ is uniformly K-stable if and only if $(X, D)$ is uniformly Ding-stable.
\end{thm}

\subsubsection{Stability via filtrations}

We here briefly recall the relevant definitions about filtrations and refer the details to \cite{BC11} (see also \cite{BHJ17}). %and \cite[Section 4.1]{Fuj15b}).
For any integer $\ell_0$ such that $-\ell_0 (K_Z+Q)=\ell_0 L$ is Cartier, we set:
\begin{equation}
R^{(\ell_0)}_m:=H^0(X, m\ell_0 L), \quad R^{(\ell_0)}:=\bigoplus_{m=0}^{+\infty} R^{(\ell_0)}_m, \quad N^{(\ell_0)}_m:=\dim_{\bC} R^{(\ell_0)}_{m}.
\end{equation}
If the integer $\ell_0$ is clear, we also denote the above data by $R_m, R, N_m$.
\begin{defn}\label{defn-gdfiltr}
A filtration $\mcF R_\bullet$ of the graded $\bC$-algebra $R=\bigoplus_{m=0}^{+\infty}R_m$ consists of a family of subspaces $\{\mcF^x R_m\}_x$ of $R_m$ for each $m\ge 0$ satisfying:
\begin{itemize}
\item (decreasing) $\mcF^x R_m\subseteq \mcF^{x'}R_m$, if $x\ge x'$;
\item (left-continuous) $\mcF^xR_m=\bigcap_{x'<x}\mcF^{x'}R_m$; 
\item (multiplicative) $\mcF^x R_m\cdot \mcF^{x'} R_{m'}\subseteq \mcF^{x+x'}R_{m+m'}$, for any $x, x'\in \bR$ and $m, m'\in \bZ_{\ge 0}$;
\item (linearly bounded) There exist $e_-, e_+\in \bZ$ such that $\mcF^{m e_-} R_m=R_m$ and $\mcF^{m e_+} R_m=0$ for all $m\in \bZ_{\ge 0}$. %We will call $e_+$ to be a shifting number.
\end{itemize}
We say that $\mcF$ is a $\bZ$-filtration if $\mcF^x R_m=\mcF^{\lceil x\rceil} R_m$ for each $x\in \bR$ and $m\in \bZ_{\ge 0}$. 

Given such a filtration $\mcF$, for any $\theta\in \bR$, the $\theta$-shifting of $\mcF$, denoted by $\mcF(\theta)$ is defined to be the filtration given by:
\begin{equation}\label{eq-Fshift}
\mcF(\theta)^x R_m:=\mcF^{x-m\ell_0 \theta} R_m.
\end{equation}
\end{defn}
Given any filtration $\{\mcF^{x}R_m\}_{x\in\bR}$ and $m\in \bZ_{\ge 0}$, the successive minima on $R_m$ is the decreasing sequence 
\[
\lambda^{(m)}_{\max}=\lambda^{(m)}_1\ge \cdots \ge \lambda^{(m)}_{N_m}=\lambda^{(m)}_{\min}
\]
defined by:
\[
\lambda^{(m)}_j=\max\left\{\lambda \in \bR; \dim_{\bC} \mcF^{\lambda} R_m \ge j \right\}.
\]
%In this paper, the
%filtrations we shall consider are always decreasing, multiplicative, linearly bounded $\bR$-filtrations of the graded algebra $R=\bigoplus_{m=0}^{+\infty}H^0(V, L^{\otimes m})=\bigoplus_{m=0}^{+\infty}R_m$. %, and we will call such filtrations to be {\it good}.
%For such filtrations, we define the following invariants after \cite{BC11}:
%\begin{equation}
%\def\arraystretch{1.5}
%\begin{array}{l}
%e_{\min}(R_m,\mcF)=\inf\{t\in\bR; \mcF^t R_m\neq R_m\}; \\
%e_{\max}(R_m,\mcF)=\sup\{t\in\bR; \mcF^t R_m\neq 0\};\\
%\displaystyle e_{\min}(\mcF)=e_{\min}(R_{\bull}, \mcF)=\liminf_{m\rightarrow +\infty} \frac{e_{\min}(R_m, \mcF)}{m}; \\
%\displaystyle e_{\max}(\mcF)=e_{\max}(R_{\bull}, \mcF)=\limsup_{i\rightarrow +\infty} \frac{e_{\max}(R_m, \mcF)}{m}. 
%\end{array}
%\end{equation}
If $\{\mcF^xR_m\}_{x}$ is a $\bZ$-filtration, then $\{\mcF^x R_m\}_x$ can be equivalently described as a $\bC^*$-equivariant degeneration of $R_m$. More precisely, there is a $\bC^*$-equivariant vector bundle $\mcR_m$ over $\bC$ such that 
\begin{equation}
\mcR_m\times_{\bC}\bC^*\cong R_m\times\bC^*, \quad (\mcR_m)_0=\bigoplus_{i=0}^{\Blue{N_m}}\mcF^{\lambda^{(m)}_{i+1}}R_m/\mcF^{\lambda^{(m)}_{i}}R_m.
\end{equation}
Denote $\mcF^{(t)}:=\mcF^{(t)}R=\bigoplus_{k=0}^{+\infty} \mcF^{kt}R_k$ and define 
\begin{equation}
\vol\left(\mcF^{(t)}\right)=\vol\left(\mcF^{(t)}R\right):=\limsup_{k\rightarrow+\infty}\frac{\dim_{\bC}\mcF^{mt}H^0(Z, m\ell_0 L)}{m^{n}/n!}.
\end{equation}
The following results are very useful.
\begin{prop}[{\cite{BC11}, \cite[Corollary 5.4]{BHJ17}}]\label{BHJvol}
\begin{enumerate}[(1)]
\item 
The probability measure 
\[
\frac{1}{N_m} \sum_{j}\delta_{m^{-1}\lambda^{(m)}_j}=-\frac{d}{d t} \frac{{\rm dim}_{\bC} \mcF^{mt}H^0(Z, m\ell_0 L)}{N_m}
\]
converges weakly as $m\rightarrow+\infty$ to the probability measure:
\[
\DHM(\mcF):=-\frac{1}{\ell_0^n L^{\cdot n}} d\; \vol\left(\mcF^{(t)}\right)=-\frac{1}{\ell_0^n L^{\cdot n}} \frac{d}{d t} \vol\left(\mcF^{(t)}\right) dt.
\]
\item The support of the measure $\DHM(\mcF)$  is given by ${\rm supp}(\DHM(\mcF))=[\lambda_{\min}, \lambda_{\max}]$ with 
\begin{align}
&
\displaystyle \lambda_{\min}:= \lambda_{\min}(\mcF):=\inf\left\{t\in \bR; \vol\left(\mcF^{(t)}\right)<\ell_0^n L^{\cdot n} \right\}; \label{lambdamin} \\
&
\lambda_{\max}:=\lambda_{\max}(\mcF):=\lim_{m\rightarrow+\infty}\frac{\lambda_{\max}^{(m)}}{m}=\sup_{m\ge 1}\frac{\lambda_{\max}^{(m)}}{m}.\label{lambdamax}
\end{align}
\end{enumerate}
\begin{rem}\label{rem-supadd}
The limit in the \eqref{lambdamax} exists because $\{\lambda^{(m)}_{\max}\}_{m\in \bZ_{>0}}$ is superadditive in the sense that $\lambda^{(m+m')}_{\max}\ge \lambda^{(m)}_{\max}+\lambda^{(m')}_{\max}$, by the multiplicative property of filtrations in Definition \ref{defn-gdfiltr}. 
\end{rem}
%Moreover, $\DHM(\mcF)$ is absolutely continuous with respect to the Lebesgue measure, except perhaps for a Dirac mass at $\lambda_{\max}$.
\end{prop}

For a filtration $\mcF R_\bullet$, choose $e_-$ and $e_+$ as in the definition \ref{defn-gdfiltr}. For convenience, we can choose $e_+=\lceil \lambda_{\rm max}(\mcF R)\rceil \in \bZ$. 
Set $e=e_+-e_-$ and define (fractional) ideals:
\begin{eqnarray}
I_{(m,x)}&:=&I^\mcF_{(m,x)}:={\rm Image}\left(\mcF^x R_m\otimes \mcO_Z(m \ell_0 L)\rightarrow \mcO_Z\right); \label{eq-Imx}\\
\tilde{\mcI}_m&:=&\tilde{\mcI}^{\mcF}_m:= I^{\mcF}_{(m, m e_+)}t^{-m e_+}+I^{\mcF}_{(m,me_+-1)}t^{1-m e_+}+\cdots\nonumber \\
&&\hskip 4cm \cdots+ I^{\mcF}_{(m, me_-+1)}t^{-m e_--1}+\mcO_Z\cdot t^{-me_-}; \label{eq-tcIm}\\
\mcI_m&:=&\mcI_m^{\mcF(e_+)}=\tilde{\mcI}^\mcF_m\cdot t^{m e_+}=I^\mcF_{(m, m e_+)}+I^{\mcF}_{(m, m e_+-1)} t^1+\cdots\nonumber\\
&&\hskip 4cm \cdots+I^{\mcF}_{(m, m e_-+1)} t^{me-1}+(t^{me})\subseteq \mcO_{Z_\bC}. \label{eq-cIm}
 %\mcI^{\mcF(e_+)}_m t^{-m e_+}
%\ccZ_m&:=&\ccZ^{\mcF}_m=\ccZ^{\mcF(e_+)}_m=\text{ normalization of } {\rm Bl}_{\mcI^\mcF_m }(X\times\bC); \label{eq-ckmcX} \\
%\ccL_m&:=&\ccL^{\mcF}_m=\ccL^{\mcF(e_+)}_m=\pi^*(-K_X\times\bC)-\frac{1}{m m_0} E_m. \label{eq-ckmcL}
\end{eqnarray}

\begin{lem}\label{lem-Isubadd}
$\{\tilde{\cI}_m\}_{m\in \bZ_{\ge 0}}$ defined above is a graded sequence of fractional ideals on $Z\times\bC$. 
\end{lem}
\begin{proof}
Note that by the multiplicative property of the filtrations, we have the inclusion $I_{(m,i)}\cdot I_{(m,j)}\subseteq I_{(m+m', i+j)}$
for any $m, m'\in \bZ_{\ge 0}$ and $i, j\in \bZ$. Here we set  $I_{(m,i)}=0$ for $i\gg 1$ and $I_{(m,i)}=\cO_{Z}$ for $i\ll 0$.
So we get: for any $m, m'\in \bZ_{\ge 0}$
\begin{eqnarray*}
\tilde{\cI}_m\cdot \tilde{\cI}_{m'}=\sum_{i, j}I_{(m,i)}\cdot I_{(m',j)} t^{-(i+j)}\subseteq \sum_{i,j} I_{(m+m', i+j)} t^{-(i+j)}\subseteq \sum_k I_{(m+m', k)}t^{-k}=\tilde{\cI}_{m+m'}.
\end{eqnarray*}
\end{proof}

\begin{defn-prop}[{\cite[Lemma 4.6]{Fuj18}}]\label{defn-ckTCm}
With the above notations, for $m$ sufficiently divisible, define the $m$-th approximating test configuration $\Blue{(\ccZ^{\mcF}_m,  \ccL^{\mcF}_m)}$ as:
\begin{enumerate}[(1)]
\item $\ccZ^{\mcF}_m$ is the normalization of blowup of $Z\times\bC$ along the ideal sheaf $\mcI^{\mcF(e_+)}_m$;
%\item $\ccQ^{\mcF}_m$ is the closure of $Q\times\bC^*$ under the $\bC^*$-equivariant inclusion $Q\times\bC^*\subset Z\times\bC^*\subset \check{\mcZ}$;
\item The semiample $\bQ$-divisor is given by:
\begin{equation} \label{eq-ckmcL}
\ccL^{\mcF}_m=\pi^*(L\times\bC)-\frac{1}{m \ell_0} E_m+\frac{e_+}{\ell_0}\Blue{\ccZ^{\mcF}_{m,0}},
\end{equation}
where $E_m$ is the exceptional divisor of the normalized blowup \Blue{and $\ccZ^{\mcF}_{m,0}$ is the central fibre of the flat family $\ccZ^{\mcF}_m\rightarrow \bC$}.
\end{enumerate}
For simplicity of notations, we also denote the data by $(\ccZ_m, \ccL_m)$ if the filtration is clear. Note that $m\ell_0 \ccL_m$ is Cartier over $\ccZ_m$.
\end{defn-prop}

We will be interested in the following invariants attached to filtrations:
\begin{eqnarray}
\bfE^\NA(\mcF)&=&\int_{\lambda_{\min}}^{+\infty} \frac{x}{\ell_0} \cdot \DHM(\mcF)=\lim_{m\rightarrow+\infty} \frac{1}{N_m}\sum_{j=1}^{N_m}\frac{\lambda^{(m)}_j}{m \ell_0}; \label{eq-ENAcF}\\
\Lam^\NA(\mcF)&=&\lim_{m\rightarrow+\infty}\frac{\lambda^{(m)}_{\max}(\mcF)}{m \ell_0}=\sup_{m\ge 1}\frac{\lambda^{(m)}_{\max}(\mcF)}{m \ell_0}\Blue{=\frac{\lambda_{\max}(\cF)}{\ell_0}}; \label{eq-lamaxF}\\
\bfJ^\NA(\mcF)&=&\Lam^\NA(\mcF)- \bfE^\NA(\mcF); \label{eq-JNAF}\\
&&\nonumber\\
\bL^\NA(\mcF)&:=&%\lim_{m\rightarrow+\infty} \bL^\NA(\ccZ_m, \ccL_m)-\frac{e_+}{\ell_0}=
\lct\left(Z\times\bC, Q\cdot \left(\mcI^{\mcF(e_+)}_\bullet\right)^{\frac{1}{\ell_0}}; (t)\right)+\frac{e_+}{\ell_0}-1; \label{eq-LNAcF} \\%=\lct\left(\left(\tilde{\mcI}^\mcF_\bullet\right)^{\frac{1}{\ell_0}}; (t)\right)-1;\\
&&\nonumber\\
\bfD^\NA(\mcF)&:=&- \bfE^\NA(\mcF)+\bL^\NA(\mcF).  \label{eq-bfDNAcF}%=-\bfE^\NA(\mcF)+\lct(\mcI^\mcF_\bullet; (t))+\frac{e_+}{\ell_0}-1.
\end{eqnarray}
In the above definition of $\bL^\NA$, we used the following notations (see \cite{JM12} for the definition of log canonical thresholds of graded sequence of ideals):
\begin{eqnarray*}
&&\lct\left(Z\times\bC, Q\cdot \left(\mcI^{\mcF(e_+)}_\bullet\right)^{\frac{1}{\ell_0}}; (t)\right)=\lim_{m\rightarrow+\infty} \lct\left(\left(Z\times\bC, Q\cdot\mcI^{\mcF(e_+)}_m\right)^{\frac{1}{m\ell_0}}; (t)\right);\\
&&\lct\left(Z\times\bC, Q\cdot \left(\mcI^{\mcF(e_+)}_m\right)^{\frac{1}{m\ell_0}}; (t)\right)\\
&&\hskip 10mm=\sup\left\{c; \left(Z\times\bC, Q\cdot \left(\mcI_m^{\mcF(e_+)}\right)^{\frac{1}{m\ell_0}}\cdot (t)^c \right) \text{ is sub-log-canonical } \right\}.
\end{eqnarray*}
See Lemma \ref{lem-LNAlim} for the fact that the limit in the definition of $\bfL^\NA(\cF)$ indeed exists. 

\begin{exmp}\label{ex-tc2fil}
Assume $(\mcZ,  \mcL)$ is a test configuration for $(Z, L)$. Choose $\ell_0>0$ such that $\ell_0 \mcL$ is Cartier. Then we have an associated $\bZ$-filtration $\mcF=\mcF_{(\mcZ, \ell_0\mcL)}$ on $R=R^{(\ell_0)}$ defined in the following way:
\vskip 1mm
$s\in \mcF^x R^{(\ell_0)}_m$ if and only if $t^{-\lceil x\rceil}\bar{s}$ extends to a holomorphic section of $m \ell_0 \mcL$, where $\bar{s}$ is the meromorphic section of $m\ell_0\mcL$ defined as the pull-back of $s$ via the projection $(\mcZ, \mcL)\times_\bC\bC^*\cong (Z, L)\times\bC^*  \rightarrow Z$. Assume the test configuration is dominating and write $\mcL=\rho^*L_{\bC}+D$ (see Definition \ref{defn-TC}) where $L_\bC=p_1^*L$. Then by \cite[Lemma 5.17]{BHJ17}, this filtration has the following more explicit description:
\begin{equation}\label{eq-filTCvan}
\cF^x R_m=\bigcap_{E}\{s\in H^0(Z, m\ell_0 L); r(\ord_E)(s)+m\ell_0\;\ord_E(D)\ge x b_E\},
\end{equation}
where $E$ runs over the irreducible components of the central fibre $Z_0$, $b_E=\ord_E(Z_0)=\ord_E(t)$ and $r(\ord_E)$ denotes the restriction of $\ord_E$ to $\bC(Z)$ under the inclusion $\bC(Z)\subset \bC(X\times\bC^*)=\bC(\mcX)$.
\vskip 1mm
For this filtration, we have $\bfF^\NA(\mcF)=\bfF^\NA(\mcZ, \mcL)$ for $\bfF$ being the functionals defined in \eqref{eq-ENAcF}-\eqref{eq-bfDNAcF}. For $m$ sufficiently divisible we have (see \cite[Theorem 5.18 and Lemma 7.7]{BHJ17})
\begin{equation}
\Lam^\NA(\mcZ, \mcL)=\frac{\lambda_{\max}(\mcF_{(\mcZ, \ell_0\mcL)})}{\ell_0}=\frac{\lambda_{\max}^{(m)}(\mcF_{(\mcZ, \ell_0\mcL)})}{\ell_0 m}
=\frac{1}{V}\rho^*(L\times\bP^1)^{\cdot n}\cdot \overline{\mcL}.
\end{equation} 
Moreover, because $\mcF_{(\mcZ, \ell_0\mcL)}$ is finitely generated (see \cite{WN12, Sze15, BHJ17}), for $m$ sufficiently divisible, the $m$-th approximating test configuration \Blue{$(\ccZ_m, \ccL_m)$ is equivalent to $(\mcZ, \mcL)$}.

\end{exmp}

\begin{exmp}\label{ex-va2fil}
\Blue{
Recall that $\Blue{\Zdiv}$ denotes the space of divisorial valuations. 
For any $v\in \Blue{\Zdiv}$, we have an associated filtration $\mcF=\mcF_v$:}
\begin{equation}\label{eq-filval}
\mcF_v^x R_m:=\{s\in R_m; v(s)\ge x\}.
\end{equation}
\Blue{Here we choose a holomorphic section $\mathfrak{e}$ of $L$ that does not vanish at the center of $v$, and define $v(s)=v(f)$ if $s=f \cdot \mathfrak{e}$ with $f\in \mcO_Z$. }
%\Blue{Let $W$ be the center of $v$ on $Z$, which is an irreducible subvariety of $Z$ defined by an ideal sheaf $\cI_W$. Let $E$ be an irreducible component of the exceptional divisor with respect to the normalized blowup of $\cI_W$. Then $\ord_E$ is a Rees valuation of $\cI_W$. By Izumi's inequality (see \cite{JM12, Li18, Ree89}), there exist $c_1,c_2>0$ such that $c_1\cdot \ord_{E}\le v\le c_2 A_{(Z,Q)}(v) \ord_E$. Because we assume $A_{(Z,Q)}(v)<+\infty$, from this inequality it is easy to see that $\mcF$ is indeed linearly bounded. } 

The following quantity plays an important role in recent studies of K-stability (see e.g. \cite{Fuj19a, Li17, BlJ20}):
\begin{equation}\label{eq-defSLv}
S_L(v)=\frac{1}{\ell_0^{n+1} L^{\cdot n}}\int_0^{+\infty} \vol(\mcF_v^{(x)}R)dx=:\frac{1}{L^{\cdot n}}\int_0^{+\infty} \vol(L-t v) dt,
\end{equation}
where we have denoted by $\vol(L-t v)$ the quantity $\vol(\mcF_v^{(t\ell_0)}R^{(\ell_0)})/\ell_0^{n}$.
Using integration by parts we get:
\begin{equation}\label{eq-ENAFv}
\bfE^\NA(\mcF_v)=-\frac{1}{\ell_0^n L^{\cdot n}}\int_0^{+\infty} \frac{x}{\ell_0} \cdot d\vol(\mcF^{(x)}R)=S_L(v).
\end{equation}
\Blue{According to the work of Blum-Jonsson \cite{BlJ20}, $S_L(v)$  computes the expected vanishing order of holomorphic sections of $L$ with respect to the valuation $v$.}

Moreover, by \cite[Proposition 2.1]{Fuj19b} (see also \cite[(5.3)]{BoJ18b}), we have a very useful inequality:
\begin{equation}\label{eq-SvsJ}
\frac{1}{n} S_L(v) \le \bfJ^\NA(\mcF_{v})=\Lam^\NA(\mcF_{v})-S_L(v)\le n S_L(v).
\end{equation}
\Blue{Note that $\Lam^\NA(\cF_v)$ is often denoted by $T(v)$ in the literature (for example in \cite{BlJ20}). }
\end{exmp}
\Blue{
\begin{rem}
More generally, for any valuation $v\in \Val(Z)$ of finite log discrepancy, \eqref{eq-filval} defines a filtration (see \cite[Lemma 3.1]{BlJ20}). In fact, there is an Izumi inequality that ensures that the linear boundedness condition (see \cite[Proposition 5.10]{JM12}, \cite[Theorem 3.1]{Li18}) in Definition \ref{defn-gdfiltr} is satisfied . 
\end{rem}}

\begin{exmp} 
Assume that \Blue{an algebraic torus} $\bT$ acts on $(Z, L)$. Then we have a weight decomposition:
\begin{equation}
R_m=\bigoplus_{\alpha\in M_\bZ} (R_m)_\alpha=(R_m)_{\alpha^{(m)}_1}\oplus \cdots\oplus  (R_m)_{\alpha^{(m)}_{N_m}}.
\end{equation}
For any $\xi\in N_\bR$, let $\kappa^{(m)}_j=\la \alpha^{(m)}_j, \xi\ra, j=1,\dots, N_m$ be the weights of $\xi$ on $R_m$. The Chow weight of $\xi$ on $L$ is then defined as:
\begin{equation}\label{eq-defCW}
\chw_L(\xi):=\lim_{m\rightarrow+\infty} \frac{1}{N_m}\sum_{j}\frac{\kappa^{(m)}_j}{m\ell_0}.
\end{equation}
\Blue{This invariant is so called because such weight under suitable normalization appeared in the study of Chow stability in Geometric Invariant Theory. 
In our set-up, we have $L=-K_Z-Q$ with the canonical $\bT$-action. Then there is an identity %(e.g. from \eqref{eq-CMstc})
\begin{equation}\label{eq-CWFut}
\Fut_{(Z, Q)}(\xi)=-\chw_L(\xi)
\end{equation}
where the left-hand-side is defined analytically by \eqref{eq-Futana}. The truth of this identity follows from an application of the equivariant Riemann-Roch theorem (see \cite[Proposition 3]{Don05}). }

On the other hand, $\xi$ determines a valuation $\wt_\xi$. Now let $W$ be the center of $\wt_\xi$ and $U$ be a $\bT$-invariant Zariski open set such that $U\cap W\neq \emptyset$. Let $\mathfrak{e}$ be an $\bT$-equivariant non-vanishing generator of $\mcO_Z(\ell_0 L)$(U). \Blue{Then for any $\xi\in N_\bR$, there exists $\mathbf{w}(\xi)\in \bR$ such that $\exp(s \xi)\circ \mathfrak{e}=\exp(s \mathbf{w}(\xi))\mathfrak{e}$. For convenience of notation, we set:
\begin{equation}\label{eq-Lieder}
\mathscr{L}_{\xi}\mathfrak{e}=\mathbf{w}(\xi)\mathfrak{e}=\left.\frac{d}{ds}\right|_{s=0}\exp(s\xi)\circ \mathfrak{e}.
\end{equation}
} 
Then we have:
\begin{equation*}
\bfE^\NA(\mcF_{\wt_\xi})=\frac{1}{N_m}\lim_{m\rightarrow+\infty}\sum_{j} \frac{\kappa^{(m)}_j}{m\ell_0}-\mathbf{w}(\xi)=\chw_L(\xi)-\mathbf{w}(\xi).
\end{equation*}
\end{exmp}

\begin{lem}[{see \cite[Lemma 5.17]{BHJ17}}]
The filtrations in the above examples are saturated. In other words, for $m$ sufficiently divisible, we have:
\begin{equation}
\mcF^xR^{(\ell_0)}_m=H^0\left(Z, \mcO_Z(-m K_Z\otimes I^\mcF_{m,x})\right).
\end{equation}
\end{lem}
To characterize Ding stability via filtrations, the following lemma is crucial.
\begin{prop}[{\cite[Lemma 4.3]{Fuj19a}, \cite[Lemma 4.7]{Fuj18}, \cite[Theorem 4.13]{BoJ18a}}]\label{prop-JccXconv}
Let $\mcF$ be any filtration. If we let $(\ccZ_m, \ccL_m)$ be the same as in Definition \ref{defn-ckTCm}, then for any $\bfF\in \{\Lam, \bfE, \bfJ, \bfL\}$, the following convergence is true:
\begin{eqnarray}
\lim_{m\rightarrow+\infty}\bfF^\NA(\ccZ_m, \ccL_m)=\bfF^\NA(\mcF). \label{eq-limEccZm} \label{eq-limJccZm}
%\lim_{m\rightarrow+\infty}\Lam^\NA(\ccZ_m, \ccL_m)&=&\Lam^\NA(\mcF); \\
%\lim_{m\rightarrow+\infty}\bfE^\NA(\ccZ_m, \ccL_m)&=&\bfE^\NA(\mcF); \\
%\lim_{m\rightarrow+\infty}\bfJ^\NA(\ccZ_m, \ccL_m)&=&\bfJ^\NA(\mcF).\\
%\lim_{m\rightarrow+\infty}\bfL^\NA(\ccZ_m, \ccL_m)&=&\bfL^\NA(\mcF).
\end{eqnarray}
\end{prop}

%\begin{proof}
%$\mcF^x_{(\ccZ_m, m \ell_0 \ccL_m)}R^{(m \ell_0)}_k$ is given by:
%\begin{equation}
%\left\{
%\begin{array}{ll}
%R^{} &
%\end{array}
%\right.
%\end{equation}
%\begin{equation}
%\mcI_m^k=: J_{(k; m, km e_+)}+J_{(k; m, kme_+-1)}t^1+\cdots+J_{(k;m,mr e_-+1)}t^{km e-1}+(t^{kme}).
%\end{equation}
%where
%\begin{equation}
%J_{(k;m,j)}:=\sum_{\stackrel{j_1+\cdots+j_k=j}{j_1,\dots,j_k\in [m e_-,me_+]\cap \bZ}}I_{(m,j_1)}\cdots I_{(m,j_k)}.
%\end{equation}
%\begin{eqnarray}
%\check{\lambda}^{(km)}_{\max}&:=&
%\sup\left\{x\in \bR; \mcF_{(\ccZ_m, \ccL_m)}^x R_{km}\neq 0\right\}\nonumber\\
%&=&k\left(-m e_++\max\{j\in [m e_-, me_+]\cap \bZ; \mcF^j R_m\neq 0\}\right)\nonumber \\
%&=&k (-me_++\lambda^{(m)}_{\max}(\mcF)).
%\end{eqnarray}
%\end{proof}
%We state and sketch a proof of a result of Fujita, which will be generalized to the equivariant case.
\Blue{The following result follows immediately from Definition \ref{defn-unistability} and the above result.
\begin{cor}[\cite{Fuj18}]
Assume that $(Z, Q)$ is uniformly Ding-stable. %The pair $(X\times\bC, \mcI_\bullet^{1/m_0}\cdot (t)^{\cdot d_\infty})$ is sub log canonical. 
%\begin{eqnarray*}
%d_\infty&:=&1-\frac{e}{\ell_0}+\frac{1}{\ell_0^{n+1}((-K_X)^{\cdot n})}\int_{e_-}^{e_+}\vol(\bar{\mcF}R^{(x)})dx\\
%&=&1-\frac{e_+}{\ell_0}+\bfE^\NA(\bar{\mcF}).
%\end{eqnarray*}
%Equivalently, we have:
Then there exists $\gamma>0$ such that for any filtration $\mcF$, 
\begin{equation}
\bfD^\NA(\mcF)\ge \gamma\cdot \bfJ^\NA(\mcF).
\end{equation}
\end{cor}}
%\begin{proof}
%By construction, we have the identity:
%\begin{equation}
%\bL^\NA(\ccZ_m, \ccQ_m, \ccL_m)=\lct\left(Z\times\bC, Q\cdot (\mcI^{\mcF}_m)^{\frac{1}{\ell_0 m}}; (t)\right)-1.
%\end{equation}
%As a consequence, 
%\begin{eqnarray*}
%\lim_{m\rightarrow+\infty} \bL^\NA(\ccZ_m, \ccQ_m, \ccL_m)&=&\lct\left(Z\times\bC, Q\cdot \left(\mcI^{\mcF}_\bullet\right)^{\frac{1}{\ell_0}}; %(t)\right)-1=\bL^\NA(\mcF).
%\end{eqnarray*}
%Combining this with \eqref{eq-limEccZm} and using $\bfD^\NA=-\bfE^\NA+\bL^\NA$, we get the limit:
%\begin{equation}
%\lim_{m\rightarrow+\infty} \bfD^\NA(\ccZ_m, \ccQ_m, \ccL_m)=\bfD^\NA(\mcF).
%\end{equation}
%If $Z$ is uniformly Ding-stable with a slope constant $\gamma$, then $\bfD^\NA(\ccZ_m, \ccQ_m, \ccL_m)\ge \gamma \bfJ^\NA(\ccZ_m, \ccL_m)$. The conclusion follows by letting $m\rightarrow +\infty$ and using Lemma \ref{prop-JccXconv}.
%\end{proof}

\subsubsection{Boucksom-Jonsson's non-Archimedean formulation}

Here we briefly recall the non-Archimedean formulation after Boucksom-Jonsson's work in \cite{BoJ18a, BoJ18b}. \Blue{Let $Z$ be a normal projective variety equipped with an ample $\bQ$-line bundle $L$. } We denote by $(Z^\NA, L^\NA)$ the Berkovich analytification of $(Z, L)$ with respect to the trivial absolute value on the ground field $\bC$. $Z^\NA$ is a topological space, whose points can be considered as semivaluations on $Z$, i.e. valuations $v: \bC(W)^*\rightarrow \bR$ on function field of subvarieties $W$ of $Z$, trivial on $\bC$. In particular, $\Blue{\Zdiv}\subset Z^\NA$. The topology of $Z^\NA$ is generated by functions
of the form $v\mapsto v(f)$ with $f$ a regular function on some Zariski open set $U\subset Z$. One can show that $Z^\NA$ is compact and Hausdorff, and $\Blue{\Zdiv}\subset Z^\NA$ is dense.

In this paper, we will only use \Blue{non-Archimedean potential}s on $L^\NA$ coming from test configurations and filtrations. Moreover we will always identify a \Blue{non-Archimedean potential}s with functions on $\Blue{\Zdiv}$. 

For any $w\in \Blue{\Zdiv}$, let $G(w)$ denote \Blue{ the unique $\bC^*$-invariant divisorial valuation on $Z\times\bC$ that satisfies 
$G(w)|_{\bC(Z)}=w$ and $G(w)(t)=1$ for the standard coordinate function $t$ on $\bC$. Alternatively it is determined by the following condition:  
for any $f=\sum_{i} f_i t^i\in \bC(Z)[t, t^{-1}]$ with $f_i\in \bC(Z)$, 
\begin{equation}\label{eq-G(w)}
G(w)\left(\sum_i f_i t^i\right)=\min_i \{w(f_i)+i\}
\end{equation}
%By \cite{}, we have the identity $A_{Z\times\bC}(G(w))=A_Z(w)+1$. 
}

\begin{defn}
Let $(\mcZ, \mcL)$ be a dominating test configuration for $(Z, L)$ with $\rho: \mcZ\rightarrow Z\times\bC$ being a $\bC^*$-equivariant morphism.  The non-Archimedean \Blue{potential associated to} $(\mcZ, \mcL)$ is defined as the following function on $\Blue{\Zdiv}$:
\begin{equation}\label{eq-phicZcL}
\phi_{(\mcZ, \mcL)}(w)=G(w)\left(\mcL-\rho^*(L\times\bC)\right).
\end{equation}
\Blue{Let $\cI$ be a $\bC^*$-invariant ideal on $Z\times\bC$ that is co-supported on $Z\times\{0\}$.
If $(\mcZ, \mcL)$ is obtained as the normalized blowup of $(Z, L)\times \bC$ along $\mcI$:}
\begin{equation}
\mcZ=\text{ normalization of } Bl_\mcI (Z\times\bC), \quad \mcL=\pi^*L\times\bC-c E
\end{equation}
for some $c\in \bQ>0$, where $\pi: \mcZ\rightarrow Z\times\bC$ is the natural projection and $E$ is the exceptional divisor of blowup, then:
\begin{equation}
\phi_{(\mcZ, \mcL)}(w)=-G(w)(c E)=-c\cdot G(w)(\mcI).
\end{equation}
\Blue{In particular the non-Archimedean potential associated to the trivial test configuration is identically 0.}

The set of \Blue{non-Archimedean potential}s obtained in such a way will be denoted as $\mcH^\NA(L)$.
\end{defn}

\begin{defn}\label{defn-phiF}
Let $\mcF=\mcF R_\bullet$ be a filtration. For any $w\in \Blue{\Zdiv}$, define the \Blue{non-Archimedean potential} associated to $\mcF$ as:
\begin{eqnarray}
\phi^{\mcF}_m(w)&=&-\frac{1}{m} G(w)\left(\left(\tilde{\mcI}^\mcF_m \right)^{\frac{1}{\ell_0}}\right)=-\frac{1}{m} G(w)\left(\left(\mcI^{\mcF(e_+)}_m t^{-m e_+} \right)^{\frac{1}{\ell_0}}\right)\nonumber\\
&=&
-\frac{1}{\ell_0}\frac{1}{m}G(w)\left(\mcI^{\mcF(e_+)}_m\right)+\frac{e_+}{\ell_0}; 
\label{eq-NAphimF}\\
\phi^{\mcF}(w)&=&
-G(w)\left(\left(\tilde{\mcI}^\mcF_\bullet \right)^{\frac{1}{\ell_0}}\right)=\lim_{m\rightarrow+\infty} \phi^\mcF_m(w). \label{eq-NAphiF} %-\frac{1}{\ell_0}\lim_{m\rightarrow+\infty}G(w)\left(\mcI^\mcF_m\right)+\frac{e_+}{\ell_0}. 
\end{eqnarray}
In particular, if $v\in \Blue{\Zdiv}$ and $\mcF=\mcF_v$, then we denote $\phi_v=\phi^{\mcF_v}$. 
\end{defn}
By Lemma \ref{lem-Isubadd}, we easily verify that $\phi^\cF_m$ satisfies property that $(m+m')\phi^\cF_{m+m'}\ge m \phi^\cF_m+m' \phi^\cF_{m'}$. This implies that the limit exists in \eqref{eq-NAphiF}. 
See \cite{BoJ18b} for an equivalent description using Fubini-Study operators on graded norms. 

Note that from the definitions \ref{defn-phiF} and \ref{defn-ckTCm} we have the identity:
\begin{equation}\label{eq-phiFTC}
\phi^\mcF_m=\phi_{(\ccZ^{\mcF}_m, \ccL^{\mcF}_m)}.
\end{equation}
\Blue{
Moreover, it is easy to see that we have the identity:
\begin{align}
\phi^{\cF}_m(w)&=-\frac{1}{m\ell_0}G(w)\left(\sum_{x}I^{\cF}_{m,x}t^{-x}\right)=-\frac{1}{m\ell_0}\min_{x} \left(w(I^\cF_{m,x})-x\right)\nonumber \\
&=-\frac{1}{m\ell_0} \min_{x,s} \{w(s)-x; s\in \cF^x R_m\}=\frac{1}{m\ell_0}\max_{x,s} \left\{x-w(s);  s\in \cF^x R_m\right\}.  \label{eq-phiFm2}
\end{align}
}
%\begin{rem}
%The pointwise convergence follows from the same argument as in \cite{JM12}. 
%\end{rem}
%\begin{lem}[{see \cite[Theorem 5.13]{BoJ18b}}]\label{lem-phivv=0}
%For any $v\in \Blue{\Zdiv}$, $\phi_v$ satisfies $\phi_v(v)=0$ and $(\omega_{\phi_v}^\NA)^n=\delta_v$.
%\end{lem}
%In this paper, we only need the fact that $\phi_v(v)=0$ which can be verified directly from the definition.

In \cite{BoJ18a, BoJ18b}, Boucksom-Jonsson defined and studied the non-Archimedean version of the class of finite energy psh metrics, and extended much of the pluripotential theory to the non-Archimedean setting. In particular, they defined non-Archimedean mixed Monge-Amp\`{e}re measures and non-Archimedean integrals. Using the non-Archimedean functionals are defined formally by the same formula as in the Archimedean case:
if $\phi$ is a finite energy \Blue{non-Archimedean potential}, one can write:
\begin{eqnarray}
\bfE^\NA(\phi)&:=&\bfE^\NA_{L}(\phi)=\frac{1}{(n+1) (2\pi)^{n} L^{\cdot n}}\sum_{j=0}^n \int_{Z^\NA}\phi \; \MA^\NA(\phi^{[j]}, \phi_\triv^{[n-j]}), \label{eq-ENAphi} \\
\Lam^\NA(\phi)&:=&\frac{1}{(2\pi)^{n} L^{\cdot n}}\int_{Z^\NA}\phi\cdot \MA^\NA(\phi), \\
%I^\NA(\phi)&=&\int_{X^\NA}(\phi-\phi_{\triv})(\MA(\phi_\triv)-\MA(\phi)),\\
\bfJ^\NA(\phi)&:=&\bfJ^\NA_L(\phi)=\Lam^\NA(\phi)-\bfE^\NA(\phi), \label{eq-JNAphi}\\
\bL^\NA(\phi)&:=&\bL^\NA_{(Z,Q)}(\phi)=\inf_{w\in Z^{\rm div}_\bQ} \left(A_{(Z, Q)}(w)+\phi(w)\right). \label{eq-LNAphi}
\end{eqnarray}
Since we don't need the details of non-Archimedean definitions, we just point out the important fact that these non-Archimedean functionals recover the corresponding \Blue{functionals} for test configurations. In other words, for any test configuration $(\mcZ, \mcL)$ of $(Z, L)$ and any functional ${\bf F}$ appearing in \eqref{eq-ENA}-\eqref{eq-bfDTC}, we have the following identity (see \cite{Berm15, BHJ17})
\begin{equation}\label{eq-FeqTC}
{\bf F}^\NA(\phi_{(\mcZ, \mcL)})={\bf F}^\NA(\mcZ, \mcL). %, \quad {\bf F}^\NA(\phi^{\mcF})={\bf F}^\NA(\mcF).
\end{equation}
We point out a consequence of this identity:
\begin{lem}\label{lem-LNAlim}
For any filtration $\cF$, the limit
$\lim_{m\rightarrow+\infty}\bfL^\NA(\phi^\cF_m)$ exists. This implies that the limit in defining $\bfL^\NA$ in \eqref{eq-LNAcF} is well defined. \end{lem}
\begin{proof}
We have pointed out that the multiplicative property of filtrations implies that $(m+m')\phi^{\cF}_{m+m'}\ge m\phi^{\cF}_m+m'\phi^{\cF}_{m'}$. Using the formula in \eqref{eq-LNAphi}, we get $(m+m') \bfL^\NA(\phi^{\cF}_{m+m'})\ge m\bfL^\NA(\phi^{\cF}_m)+m'\bfL^\NA(\phi^{\cF}_{m'})$. By Fekete's lemma, this implies the existence of the limit $\lim_{m\rightarrow+\infty}\bfL^\NA(\phi^{\cF}_{m})$. The last statement follows from the identity $\bfL^\NA(\phi^{\cF}_m)=\bfL^\NA(\check{\mcZ}_m, \check{\mcL}_m)$ by \eqref{eq-FeqTC} for $\bfF=\bfL$. 

\end{proof}

For filtrations, we need the following inequality:
\begin{lem}
For any filtration $\cF$, we always have the following inequality:
\begin{equation}\label{eq-weakL}
\inf_{w\in Z^{\rm div}_\bQ}(A_{(Z,Q)}(w)+\phi^\cF (w))\ge \bfL^\NA({\cF}).
\end{equation}
\end{lem}
\begin{proof}
Because $\phi^{\cF}_m$ is associated to test configurations, by \eqref{eq-FeqTC} and \eqref{eq-phiFTC} we have the identities:
\begin{equation}
\inf_{w\in Z^{\rm div}_\bQ}(A_{(Z,Q)}(w)+\phi^{\cF}_m(w))=\bfL^\NA(\phi^{\cF}_m)=\bfL^\NA(\check{\mcZ}^\cF_m, \check{\mcL}^\cF_m).
\end{equation}
Note that the functional $\phi\mapsto \inf_{w\in Z^{\rm div}_\bQ}(A_{(Z,Q)}(w)+\phi(w))$ is upper semi-continuous with respect to the pointwise convergence. Since $\phi^{\cF}_m\rightarrow \phi^\cF$ pointwise,  the inequality follows easily by letting $m\rightarrow+\infty$ and using the fact that $\bfL^\NA(\check{\mcZ}^\cF_m, \check{\mcL}^\cF_m)\rightarrow \bfL^\NA(\cF)$ (by Lemma \ref{lem-LNAlim}). 
\end{proof}
We will also use the following fact:
\begin{lem}\label{lem-phivv=0}
For any $v\in \Blue{\Zdiv}$, $\phi_v(v)=0$. 
\end{lem}
\begin{proof}
According to \eqref{eq-Imx}, for any $v\in \Blue{\Zdiv}$, $I^{\cF_v}_{(m,x)}={\rm Image}\left(\mcF^x_v R_m\otimes \mcO_Z(m \ell_0 L)\rightarrow \mcO_Z\right)$ where $\cF^x_v R_m=\{s\in H^0(Z, m\ell_0 L), v(s)\ge x\}$. So $v(I^{\cF_v}_{(m,x)})\ge x$. Recall that according to \eqref{eq-tcIm}, we have:
\begin{eqnarray*}
\tilde{\mcI}^{\mcF_v}_m&=& I^{\mcF_v}_{(m, m e_+)}t^{-m e_+}+I^{\mcF_v}_{(m,me_+-1)}t^{1-m e_+}+%\cdots\nonumber \\
%&&\hskip 4cm 
\cdots+ 
I^{\mcF_v}_{(m, me_-+1)}t^{-m e_--1}+\mcO_Z\cdot t^{-me_-}. %\label{eq-tcIm}
\end{eqnarray*}
Using Definition \ref{defn-phiF}, for $m\gg 1$, we then have:
\begin{eqnarray*}
\phi^{\cF_v}_m(v)&=&-\frac{1}{m\ell_0} G(v)(\tilde{\mcI}^{\cF_v}_m)=-\frac{1}{m\ell_0}\min_j G(v)(I^{\cF_v}_{(m, j)}t^{-j})\\
&=&-\frac{1}{m\ell_0}\min_j \left(v(I^{\cF_v}_{(m,j)})-j\right)\le 0.
\end{eqnarray*}
By letting $m\rightarrow+\infty$, we get $\phi_v(v)\le 0$. 

On the other hand, we claim that $\phi_v(v)\ge 0$, and hence $\phi_v(v)=0$.  
To see this, for any $m\gg 1$ we choose a section $s\in H^0(Z, m\ell_0 L)$ that vanishes at the center of $v$ over $Z$.  Set $x=v(s)>0$.  Then $G(v)(\mcI_{(m,\lfloor x\rfloor)}t^{-\lfloor x\rfloor})\le x-\lfloor x\rfloor<1$. It is easy to see that this implies $\phi^{\cF_v}_m(v)> -\frac{1}{m\ell_0}$.  Letting $m\rightarrow+\infty$, we indeed get $\phi_v(v)\ge 0$. 
\end{proof}

\addtocounter{thm}{-1}

\begin{rem}\label{rem-continuity}
\begin{enumerate}
\item
For any filtration $\cF$, it is expected that \Blue{the upper semicontinuous regularization $(\phi^{\cF})^*$ of $\phi^{\cF}$ is always a \Blue{non-Archimedean potential} on $L^\NA$ and there is also an identity ${\bf F}^\NA((\phi^{\mcF})^*)={\bf F}^\NA(\mcF)$}
where the right-hand-side were already defined in \eqref{eq-ENAcF}-\eqref{eq-LNAcF} and the left-hand-side can be well defined using the same formula as in \eqref{eq-ENAphi}-\eqref{eq-LNAphi} (see \cite{BoJ18a}). In fact, by the recent work \cite{BoJ18a}, the psh property of $(\phi^\cF)^*$ for any normal projective variety $(Z, L)$ would follow from a conjecture called {\it continuity of envelopes}. 
The truth of this latter conjecture is known when $Z$ is smooth (\cite[Theorem 8.5]{BFJ16}). Note that we do not need continuity of envelopes or even this identity for filtrations in this paper. Indeed, we only need the easier inequality \eqref{eq-weakL} in the inequalities \eqref{eq-weak1} and \eqref{eq-weak2}. 
\item
For $\phi_v$ in Lemma \ref{lem-phivv=0}, it further expected that $\phi_v$ is a continuous solution to the non-Archimedean Monge-Am\`{e}re equation $\MA^\NA(\phi_v)=\delta_v$. This again follows from the continuity of envelopes which is known in the smooth case (see \cite[Theorem 5.13]{BoJ18b}). 
\end{enumerate}

\end{rem}
%\begin{rem}
%We will not use the precise formula for ${\bf F}^\NA$ writing the previous functionals. 
%\end{rem}

%\begin{equation}
%(I-J)^\NA(\phi)=\bfE^\NA(\phi)-\int_{X^\NA}(\phi-\phi_\triv)\MA(\phi).
%\end{equation}

Later we will also use the fact that the multiplicative group $\bR^\times_+$ acts on the space of \Blue{non-Archimedean potential}s that come from filtrations. For any $b>0$ and a \Blue{non-Archimedean potential} that is represented by a function $\phi$ on $\Blue{\Zdiv}$, the action is given by the formula (see \cite[(2.1)]{BoJ18a}):
\begin{equation}\label{eq-NAscaling}
(b\circ \phi)(v)=b\cdot\phi(b^{-1}v).
\end{equation}
In the case that $\phi=\phi_{(\mcZ, \mcL)}$ and $b\in \bZ_{>0}$, the rescaling operation corresponds to the base change. To see this we denote 
\begin{equation}
(\mcZ, \mcQ, \mcL)^{(b)}:=\left({\text{ normalization of }} (\mcZ, \mcQ, \mcL)\times_{\bC, {\rm m}_b}\bC, b\cdot \eta\right)\overset{\pi_b}{\longrightarrow} (\mcZ, \mcQ, \mcL),
\end{equation}
where
${\rm m}_b: t'\rightarrow t'^b=t$, $b\cdot\eta:=b\cdot {\rm m}_b^*\eta$. Then it is easy to verify that
$(\pi_b)_*G(v)=b G(b^{-1}v)$ so that 
\begin{eqnarray}
\phi_{(\mcZ, \mcL)^{(b)}}(v)&=&G(v)(\pi_b^*(\mcL-\rho^*(L\times\bC)))=(\pi_b)_*G(v)(\mcL-\rho^*(L\times\bC))\nonumber \\
&=&b G(b^{-1}v)(\mcL-\rho^*(L\times\bC))=b\phi_{(\mcZ, \mcL)}(b^{-1}v)=(b\circ \phi_{(\mcZ, \mcL)})(v). \label{TCbasechange}
\end{eqnarray} 
%In other words, we have:
%\begin{equation}\label{eq-TCbasechange}
%b\circ \phi_{(\mcZ, \mcL)}=\phi_{(\mcZ,  \mcL)^{(b)}}.
%\end{equation}

\section{Twists of non-Archimedean potentials}

\subsection{Twists of test configurations}\label{sec-TCtwist}
\Blue{Let $Z$ be a normal projective variety and $Q$ be an effective $\bQ$-divisor. Let $\bG$ be a reductive complex Lie group that acts faithfully on $Z$ and preserves the $\bQ$-divisor divisor $Q$. 
Assume that $L:=-K_X-Q$ is $\bQ$-Cartier. Then there is an induced $\bG$-action on the $\bQ$-line bundle $L$.}

\begin{defn}\label{defn-GequiTC}
$(\mcZ, \mcQ, \mcL)$ is a $\bG$-equivariant test configuration for $(Z, Q, L)$ if 
\begin{itemize}
\item $(\mcZ, \mcQ, \mcL)$ is a test configuration for $(Z, Q, L)$. % (see Definition \ref{defn-TC}).
%\item $\eta$ is a holomorphic vector field on $\mcZ$ that generates a $\bC^*$-action $\sigma_\eta$ on the triple $(\mcZ, \mcQ, \mcL)$ and satisfies the property $\pi_*\eta=-t\partial_t$; 
\item $\bG$ acts on $(\mcZ, \mcQ, \mcL)$ such that the action of $\bG$ commutes with the $\bC^*$-action of the test configuration, and the action of $\bG$ on $(\mcZ, \mcQ, \mcL)\times_\bC {\bC^*}{\cong} (Z, Q, L)\times\bC^*$ coincides with the action of $\bG$ on (the first factor of) $(Z, Q, L)\times \bC^*$.

%\item The isomorphism $(\mcZ, \mcQ, \mcL)|_{\bC^*}\cong (Z, Q, L)\times \bC^*$ is $G\times\bC^*$-equivariant.

\end{itemize}
\end{defn}
\Blue{To continue, we note that in the definition of test configuration (Definition \ref{defn-TC}), the $\bC^*$-action on $(\mcZ, \mcL)$ can be equivalently encoded in the infinitesimal action of the generating holomorphic vector field. In other words, a test configuration is completely determined by the data $(\mcZ, \mcL, \eta)$ where $\eta$ is a holomorphic vector field on $\mcZ$ that lifts to a holomorphic vector field on the total space of a line bundle $m \mcL$ for some $m>1$ and satisfies $\pi_*\eta=t\partial_t$ where $t$ is the standard coordinate function on $\bC$. From this point of view, the next definition is a natural generalization. }

\Blue{Recall that $\bT$ always denotes the identity component of the center of $\bG$ and $N_\bR$ was defined in \eqref{eq-Nlattice}.}
\begin{defn}[\cite{His16b}]
Let $(\mcX, \mcL)$ be a test configuration for $(X, L)$ and $\eta$ be the holomorphic vector field generating the $\bC^*$-action. 
For any $\xi\in N_\bR$, the $\xi$-twist of $(\mcZ, \mcL, \eta)$ is the data $(\mcZ,  \mcL, \eta+\xi)$, which, for simplicity, will also be denoted by $(\mcZ_\xi, \mcL_\xi)$. If $\xi\in N_\bZ$, then $(\mcZ_\xi, \mcL_\xi)=(\mcZ, \mcL, \eta+\xi)$ is a test configuration. In general, we shall call $(\mcZ, \mcL, \eta+\xi)$ to be an $\bR$-test configuration.
\end{defn}
The twists of test configurations first appeared in the work of Hisamoto (\cite{His16a, His16b}).
The following result begins to study the twists of test configurations from \Blue{the} non-Archimedean point of view. 
\begin{prop}
Let $(\mcZ, \mcQ, \mcL)$ be a $\bG$-equivariant dominating test configuration for $(Z, Q, L)$. For any $\xi\in N_\bZ$, the \Blue{non-Archimedean potential} $\phi_{(\mcZ_\xi, \mcL_\xi)}$ defined by the twisted test configuration is related to $\phi_{(\mcZ, \mcL)}$ by the following identity:
for any $w\in \Blue{\Zdiv}$
\begin{equation}\label{eq-phitwist1}
\phi_{(\mcZ_\xi, \mcL_\xi)}(w)=\phi_{(\mcZ, \mcL)}(w_\xi)+\theta^L_\xi(w),
\end{equation}
where the function $\theta^L_\xi$, also denoted by $\theta_\xi$ if the $\bT$-equivariant $\bQ$-line bundle $L=-K_Z-Q$ is clear, is given by:
\begin{equation}\label{eq-thetaxi}
\theta_\xi(w)=A_{(Z, Q)}(w_\xi)-A_{(Z, Q)}(w).
\end{equation}
Moreover, the following identities hold true:
\begin{eqnarray}
\bfE^\NA(\mcZ_\xi, \mcL_\xi)&=&\bfE^\NA(\mcZ,  \mcL)-\Fut_{(Z,Q)}(\xi); \label{eq-EYBtwist}\\
\bL^\NA(\mcZ_\xi,  \mcL_\xi)&=&\bL^\NA(\mcZ,  \mcL); \label{eq-bLYBtwist}\\
\bfD^\NA(\mcZ_\xi,  \mcL_\xi)&=&\bfD^\NA(\mcZ, \mcL)+\Fut_{(Z,Q)}(\xi).\label{eq-DYBtwist}
\end{eqnarray}
\end{prop}

\begin{proof}
Since $\hat{\sigma}_{-\xi}(t)$ \Blue{is} the $\bC^*$-action generated by $-\xi$, we can let $\bar{\sigma}_{-\xi}: Z_\bC:=Z\times\bC\dashrightarrow Z_\bC$ be the birational map defined by:
$(x, t)\mapsto (\sigma_{-\xi}(t)\circ x, t)$ for any $(x, t)\in Z\times\bC^*$. 
In the following argument, we will also use $\la\eta\ra$ (resp. $\la \xi\ra$) to denote the $\bC^*$-action generated by $\eta$ (resp. $\xi$).
Consider the commutative diagram:
\begin{equation}\label{eq-3morph}
\xymatrix{% @R=1.5pc @C=0.5pc{
& \ar_{q_1}[ld] \mcU \ar^>>>>>>>{\rho_\mcW}[dd]  \ar^{q_2}[rd] & \\
\mcZ=\mcZ^{(1)} \ar_{\Blue{\rho_1}}[dd] \ar@{-->}[rr] &   & \mcZ=\mcZ^{(2)} \ar^{\Blue{\rho_2}}[dd]   \\
 & \ar_{\mu_1}[ld] \mcW \ar^{\mu_2}[rd] & \\
Z_\bC=Z^{(1)}_\bC \ar^{\bar{\sigma}_{-\xi}}@{-->}[rr]  &  & Z_\bC=Z^{(2)}_\bC 
}
\end{equation}
\Blue{where $\rho_i, i=1,2$ are $\bC^*$-equivariant dominant morphisms, and $\mathcal{U}$ (resp. $\mcW$) is any variety that $\la \eta\ra\times \la \xi\ra \cong (\bC^*)^2$-equivariantly dominates the fibre product of $\mcZ^{(1)}$ (resp. $Z^{(1)}_\bC$)  with $\mcZ^{(2)}$ (resp. $Z^{(2)}_\bC$)}. 
The map $\Blue{\rho_1}\circ q_1$ is $\la\eta\ra$-equivariant. Moreover, the test configuration $(\mcZ_\xi, \mcL_\xi)$ is equivalent to the test configuration $(\mcU, q_2^*\mcL, \eta)$.
We now decompose:
\begin{eqnarray}
q_2^*\mcL-q_1^*\Blue{\rho_1}^*L_\bC&=&q_2^*\mcL-q_2^*\Blue{\rho_2}^*L_\bC+q_2^*\Blue{\rho_2}^*L_\bC-q_1^*\Blue{\rho_1}^*L_\bC\nonumber\\
&=& q_2^*(\mcL-\Blue{\rho_2}^*L_\bC)+\Blue{\rho_\mcW}^*(\mu_2^*L_\bC-\mu_1^*L_\bC). \label{eq-twistdec}
\end{eqnarray}
\Blue{
Fix $w\in \Blue{\Zdiv}$ and $\alpha\in \bZ^r$. \Blue{For any $f\in \bC(Z)_\alpha$ (i.e. $f\circ \mathbf{t}^{-1}=\mathbf{t}^{\alpha}\cdot f$ for any $\mathbf{t}\in \bC^*)^r$, see \eqref{eq-falpha})}, let $\bar{f}=p_1^*f$ denote the function on $Z\times \bC^*$ via the projection $p_1$ to the first factor. Then $\bar{\sigma}_{-\xi}^* \bar{f}=\bar{f}\circ \bar{\sigma}_{-\xi}(t)=t^{\la \alpha, \xi\ra} \bar{f}$. \Blue{By the definition of $G(w)$ in \eqref{eq-G(w)}, we know that $G(w)(\bar{f}))=w(f)$}. So we can calculate:
\begin{eqnarray*}
(q_2)_*G(w)(\bar{f})&=&G(w)((q_2)^*\bar{f})=G(w)(t^{\la \alpha, \xi\ra} \bar{f})=\la \alpha, \xi\ra+G(w)(\bar{f})\\
&=&\la\alpha, \xi\ra+w(f)=w_\xi(f)=G(w_\xi)(\bar{f}).
\end{eqnarray*}}
So $(q_2)_*G(w)=G(w_\xi)$. For any $w\in \Blue{\Zdiv}$, by \eqref{eq-twistdec}, we have:
\begin{eqnarray*}
\phi_\xi(w)&=&\phi(w_\xi)+\theta_\xi(w),
\end{eqnarray*}
where $\theta_\xi(w)=G(w)(\mu_2^*L_\bC-\mu_1^*L_\bC)$. 
\Blue{Recall that $L_\bC=-(K_{\cZ_\bC}+Q_\bC)=-p_1^*(K_Z+Q)$ where $p_1: Z_\bC=Z\times\bC\rightarrow Z$ is the projection. 
To get identity \eqref{eq-thetaxi}, we first use the identity $A_{(Z,Q)}(w)=A_{(Z_\bC, Q_\bC)}(G(w))-1$ (see \cite[Proposition 4.11]{BHJ17}) and the expression of $\theta_\xi$ to get:
\begin{equation*}
A_{(Z, Q)}(w)+\theta_\xi(w)=A_{(Z^{(1)}_\bC, Q^{(1)}_\bC)}(G(w))-1-G(w)\left(\mu_2^*(K_{Z^{(2)}_\bC}+Q_\bC)-\mu_1^* (K_{Z^{(1)}_\bC}+Q_\bC)\right).
\end{equation*}
To continue, we observe that the right-hand-side of the above identity does not change if we replace $\mathcal{W}$ by 
higher birational models. So by possibly taking further blowing-ups we can assume that the center of the divisorial valuation $G(w)$ on $\mathcal{W}$ is a prime divisor. We then have the identity $A_{(Z_\bC, Q_\bC)}(G(w))=G(w)(K_{\mcW}-\mu_1^* (K_{Z_\bC}+Q_\bC))$ and can continue to compute the right-hand-side as:
\begin{eqnarray*}
&&G(w)\left(K_{\mcW}-\mu_1^* (K_{Z^{(1)}_\bC}+Q_\bC)\right)-1-G(w)\left(\mu_2^*(K_{Z^{(2)}_\bC}+Q_\bC)-\mu_1^* (K_{Z^{(1)}_\bC}+Q_\bC)\right)\\
&=&G(w)\left(K_{\mcW/(Z^{(2)}_\bC, Q^{(2)}_\bC)}\right)-1=(\mu_2)_*(\mu_1^{-1})_*G(w)\left(K_{\mcW/(Z^{(2)}_\bC, Q^{(2)}_\bC)}\right)\\
&=&(q_2)_*G(w)\left(K_{\mcW/(Z^{(2)}_\bC, Q^{(2)}_\bC)}\right)-1\\
&=&G(w_\xi)\left(K_{\mcW/(Z_\bC, Q_\bC)}\right)-1=A_{(Z, Q)}(w_\xi).
\end{eqnarray*}
}

By \eqref{eq-thetaxi} and \eqref{eq-phitwist1}, we have the identity:
\begin{eqnarray*}
A_{(Z, Q)}(w)+\phi_\xi(w)&=& A_{(Z, Q)}(w)+\phi(w_\xi)+\theta_\xi(w)=A_{(Z, Q)}(w_\xi)+\phi(w_\xi).
\end{eqnarray*}
Taking the infimum over $w$ on both sides and by the change of variable, we get the identity \eqref{eq-bLYBtwist}.

Let us prove \eqref{eq-EYBtwist}. Assume $\mcL=\pi^*(-K_Z-Q)+E$. Let $\mcL_b=\pi^*(-K_Z-Q)+b E$. Consider $$h(b):=\frac{1}{n+1}\overline{q_2^*\mcL_b}^{\cdot n+1}-\frac{1}{n+1}\overline{q_1^*\mcL_b}^{\cdot n+1}, $$
where the compactifications we use are using the isomorphism induced by $\eta$.
\begin{eqnarray*}
\frac{b}{db}h(b)&=&q_2^*\mcL_b^{\cdot n}\cdot q_2^*E- q_1^* \mcL_b^{\cdot n}\cdot q_1^* E=0.
\end{eqnarray*}
 So we get:
\begin{eqnarray*}
\bfE^\NA(\mcZ_\xi, \mcL_\xi)-\bfE^\NA(\mcZ,  \mcL)&=&\frac{1}{n+1}\overline{q_2^*\mcL}^{\cdot n+1}-\frac{1}{n+1}\overline{q_1^*\mcL}^{\cdot n+1}=h(1)=h(0)\\
&=&\frac{1}{n+1}\overline{q_2^*L}^{\cdot n+1}-\frac{1}{n+1}\overline{q_1^*L}^{\cdot n+1}\\
&=&\chw_L(\xi)=-\Fut_{(Z,Q)}(\xi). \quad (\text{see \eqref{eq-CWFut}})
\end{eqnarray*}

The identity \eqref{eq-DYBtwist} follows from \eqref{eq-bLYBtwist} and \eqref{eq-EYBtwist}.

\end{proof}

\begin{rem}
Note that the identities \eqref{eq-EYBtwist}-\eqref{eq-DYBtwist} \Blue{can also be obtained} by using Archimedean energy functionals. Let $\Phi=\{\vphi(t)\}$ be a smooth and psh Hermitian metric on $\mcL$. Then $\hat{\sigma}_\xi(t)^*\Phi:=\{\hat{\sigma}_\xi(t)^*\vphi(t)\}$ is a smooth and psh Hermitian metric on $(\mcZ_\xi, \mcL_\xi)$. On the other hand, because the action of $\bT\cong (\bC^*)^r$ on $-(K_Z+Q)$ is induced by the pull back of (logarithmic) n-forms, one can easily verify that:
\begin{eqnarray*}
\bfL(\hat{\sigma}_\xi(t)^*\vphi(t))=\bfL(\vphi(t)), \quad \bfE(\hat{\sigma}_\xi(t)^*\vphi(t))=\bfE(\vphi(t))-\log|t|^2 \cdot\Fut(\xi).
\end{eqnarray*}
The identities \eqref{eq-EYBtwist}-\eqref{eq-bLYBtwist} follow by taking the slope at infinity and using \eqref{eq-BHJslope}.
\end{rem}

If $\xi\in N_\bQ$ and $b\xi\in N_\bZ$ for some $b\in \bN$, then $(\mcZ_\xi, \mcQ_\xi, \mcL_\xi)$ induces a test configuration by base change: 
\begin{equation}\label{eq-TCxib}
(\mcZ_\xi, \mcQ_\xi, \mcL_\xi)^{(b)}:=\left( {\text{ normalization of }} (\mcZ, \mcQ, \mcL)\times_{\bC, {\rm m}_b}\bC, b \eta+b \xi \right),
\end{equation}
where
${\rm m}_b: t'\rightarrow t'^b=t$, $b\eta:=b\cdot {\rm m}_b^*\eta$ and $b \xi=b\cdot {\rm m}_b^*\xi$. Then with $\phi=\phi_{(\mcZ,\mcL)}$, we define the $\xi$-twist of $\phi$ to be the \Blue{non-Archimedean potential} represented by the following function on $\Blue{\Zdiv}$:
\begin{equation}\label{eq-phixirational}
\phi_\xi(v)=(b^{-1}\circ \phi_{(\mcZ_\xi, \mcL_\xi)^{(b)}})(v).
\end{equation}
For the non-Archimedean energies appearing in \eqref{eq-ENA}-\eqref{eq-bfDTC}, we also set:
\begin{equation}
\bfF^\NA(\mcZ_\xi, \mcQ_\xi, \mcL_\xi)=b^{-1} \bfF^\NA((\mcZ_\xi, \mcQ_\xi, \mcL_\xi)^{(b)}).
\end{equation}

\begin{lem}
For any $\xi\in N_\bQ$, the same identity as in \eqref{eq-phitwist1} holds true:
\begin{equation}\label{eq-phitwist2}
\phi_\xi(v)=\phi(v_\xi)+\theta_\xi(v).
\end{equation}
\end{lem}

\begin{proof}
For simplicity, we write $\phi_{(\mcZ, \mcL)^{(b)}}=b\circ \phi$.
From \eqref{eq-phixirational} and \eqref{eq-phitwist1}, we can calculate:
\begin{eqnarray*}
\phi_\xi(v)&=&(b^{-1}\circ (b\circ \phi)_{b\xi})(v)=b^{-1}\cdot (b\circ \phi)_{b\xi}(bv)\\
&=&b^{-1}\cdot \left((b\circ\phi)((bv)_{b\xi})+\theta_{b\xi}(bv)\right)\\
&=&b^{-1}\cdot \left(b\cdot \phi(b^{-1}(bv)_{b\xi})+\theta_{b\xi}(bv)\right)\\
&=&\phi(v_\xi)+b^{-1}\theta_{b\xi}(bv).
\end{eqnarray*}
Now we can note that:
\begin{eqnarray*}
b^{-1}\theta_{b\xi}(bv)&=&b^{-1} \left(A_{(Z, Q)}((bv)_{b\xi})-A_{(Z, Q)}(bv)\right)\\
&=&A_{(Z, Q)}(v_\xi)-A_{(Z, Q)}(v)=\theta_\xi(v).
\end{eqnarray*}

\end{proof}
For any $\xi\in N_\bR$, we can define $\phi_\xi$ using the formula \eqref{eq-phitwist2}. We will see in the following subsection that the twist $\phi_\xi$ can be understood as \Blue{the non-Archimedean potential} from a twisted filtration. Indeed, the identity \eqref{eq-phitwist2} is nothing but the non-Archimedean analogue of the well-known formula in the Archimedean case.

\subsection{Twists of filtrations}

Let $\mcF=\mcF R_\bullet $ be a filtration of $R=R^{(\ell_0)}=\bigoplus_{m=0}^{+\infty} H^0(Z, m\ell_0L)$. Assume that $\mcF$ is $\bT$-equivariant, which means that $\mcF^x R_m$ is a $\bT$-invariant subspace of $R_m$ for any $x\in \bR$.
For $\alpha\in M_\bZ=N_\bZ^\vee$, denote the weight space
\begin{equation}
(R_m)_\alpha=\{s\in R_m; \tau\circ s=\tau^\alpha s \text{ for all } \tau\in (\bC^*)^r \}.
\end{equation}
Then we have:
\begin{equation}
(\mcF^x R_m)_\alpha:=\{s\in \mcF^x R_m; \tau\circ s=\tau^\alpha s\}=\mcF^x R_m \cap (R_m)_\alpha,
\end{equation}
and the decomposition:
\begin{equation}
\mcF^x R_m=\bigoplus_{\alpha\in M_\bZ} (\mcF^x R_m)_\alpha.
\end{equation}

\begin{defn}
For any $\xi\in N_\bR$, the $\xi$-twist of $\mcF$ is the filtration $\mcF_\xi R_\bullet$ defined by:
\begin{equation}\label{eq-Ftwist}
\mcF_\xi^x R_m=\bigoplus_{\alpha\in M_\bZ} (\mcF_\xi^x R_m)_\alpha, \quad \text{where}\quad
(\mcF_\xi^x R_m)_\alpha:=(\mcF^{x-\la \alpha, \xi\ra} R_m)_\alpha.
\end{equation}
\end{defn}

\begin{exmp}
Let $(\mcZ, \mcQ, \mcL)$ be a test configuration for $(Z, Q, L)$, which determines a filtration $\mcF:=\mcF_{(\mcZ, \ell_0 \mcL)}$ of $R^{(\ell_0)}$ (see Example \ref{ex-tc2fil}). Recall that $s\in \mcF^x R_m$ if and only if $t^{-\lceil x\rceil}\bar{s}$ extends to a holomorphic section.
Let $\xi\in N_\bZ$. If $s\in (\mcF^x R_m)_\alpha$, then $\bar{\sigma}_\xi^*\bar{s}=t^{\la \alpha, \xi\ra}\bar{s}$ which implies $s\in \left(\mcF^{x-\la \alpha, \xi\ra}_{(\mcZ_\xi, \ell_0 \mcL_\xi)} R_m\right)_\alpha$.
So we get the identification:
$\mcF^x_{(\mcZ_\xi, \ell_0 \mcL_\xi)}R_m=\mcF^x_{(\mcX, \ell_0 \mcL),\xi} R_m$.
\end{exmp}
The following proposition deals with twists of filtrations associated to valuations.

\begin{prop}\label{prop-Fvxi}
Let $v\in (\Blue{\Zdiv})^\bT$ and $\mcF=\mcF_v$ be defined as in \eqref{eq-filval}.
We have the following identification of the filtration associated to the twisted valuation: for any $\xi\in N_\bR$
\begin{equation}\label{eq-Fvxi}
(\mcF_{v_\xi}^xR_m)_\alpha=\left(\mcF_v^{x-\la \alpha, \xi\ra-m\ell_0 \theta_{\xi}(v)}R_m\right)_\alpha,
\end{equation}
where $\theta_\xi(v)=\theta^L_\xi(v)$ is given by \eqref{eq-thetaxi}:
\begin{equation}\label{eq-thetaxi2}
\theta_\xi(v)=A_{(Z, Q)}(v_\xi)-A_{(Z, Q)}(v).
\end{equation}
%As a consequence, we have the identity:
%\begin{equation}
%\phi_{v_\xi}=\phi_{v,\xi}.
%\end{equation}
\end{prop}
%\begin{proof}
%\begin{eqnarray*}
%\phi_{v_\xi}=\phi^{\mcF(v_\xi)}=\phi^{\mcF(v)_\xi}-\theta_\xi(v)
%\end{eqnarray*}
%\end{proof}

\begin{proof}
 Let \Blue{$W$ (resp. $W'$) be the closure of the center of} $v$ (resp. $v_\xi$) on $Z$. Let $U$ (resp. $U'$) be a $\bT$-invariant Zariski open set such that $U\cap W\neq \emptyset$ (resp. $U'\cap W'\neq \emptyset$), and let $\mathfrak{e}$ (resp. $\mathfrak{e}'$) be an equivariant nonvanishing section of $-\ell_0 (K_Z+Q)|_U$ (resp. $-\ell_0 (K_Z+Q)|_{U'}$). 

Assume $s\in (\mcF_{v_\xi}^x R_m)_\alpha$. Write $s=f \mathfrak{e}^{ m}$ on $U$ and $s=f' \mathfrak{e}'^{ m}$ on $U'$ for $f\in \mcO_Z(U)$ and $f'\in \mcO_Z(U')$. We have the identity \Blue{(see \eqref{eq-Lieder} for the definition of $\mathscr{L}_\xi$)}:
\begin{eqnarray*}
\la \alpha, \xi\ra =\frac{\msc{L}_\xi s}{s}=\frac{\msc{L}_\xi(f)}{f}+m\frac{\msc{L}_\xi \mathfrak{e}}{\mathfrak{e}}.
\end{eqnarray*}
Then we have the following identities:
\begin{eqnarray}
v_\xi(s)&=&v_\xi(f')=v(f')+\frac{\msc{L}_\xi f'}{f'}\nonumber \\
&=&v(f)+v\left(\frac{\mathfrak{e}^{m}}{\mathfrak{e}'^{ m}}\right)+\la \alpha, \xi \ra-m\frac{\msc{L}_\xi \mathfrak{e}'}{\mathfrak{e}'}\nonumber \\
&=&v(s)+\la \alpha, \xi\ra+m\left(v\left(\frac{\mathfrak{e}}{\mathfrak{e}'}\right)-\frac{\msc{L}_\xi \mathfrak{e}'}{\mathfrak{e}'}\right)\nonumber \\
&=&v(s)+\la \alpha, \xi\ra+\ell_0 m\cdot \tilde{\theta}_\xi(v), \label{eq-vxis}
\end{eqnarray}
where
\begin{equation}
\tilde{\theta}_\xi(v)=\frac{1}{\ell_0}\left(v\left(\frac{\mathfrak{e}}{\mathfrak{e}'}\right)-\frac{\msc{L}_\xi \mathfrak{e}'}{\mathfrak{e}'}\right)=:\frac{1}{\ell_0}\left(v\left(\frac{\mathfrak{e}}{\mathfrak{e}'}\right)-{\bf c}\right).
\end{equation}
So $v_\xi(s)\ge x$ if and only if $v(s)\ge x-\la \alpha, \xi\ra-\tilde{\theta}_\xi(v)$. We need to verify $\tilde{\theta}_\xi=\theta_\xi$. 
To see this, we use the commutative diagram \Blue{in \eqref{eq-3morph}} and calculate. 
\begin{eqnarray*}
\theta_\xi(v)&=&-G(v)(\mu_2^*L_\bC-\mu_1^*L_\bC)=G(v)(\mu_2^*(K_{Z_\bC}+Q_\bC)-\mu_1^*(K_{Z_\bC}+Q_\bC))\\
&=&-\frac{1}{\ell_0}G(v)\left(\frac{\mu_2^*\bar{\mathfrak{e}'}}{\mu_1^*\bar{\mathfrak{e}}}\right)=-\frac{1}{\ell_0}G(v)\left(\frac{\mu_1^*\bar{\sigma}_\xi^*\bar{\mathfrak{e}'}}{\mu_1^*\bar{\mathfrak{e}}}\right)=\frac{1}{\ell_0}\left(-G(v)\left(\frac{\mu_1^*\bar{\sigma}_\xi^*\bar{\mathfrak{e}'}}{\mu_1^*\bar{\mathfrak{e}'}}\right)-
G(v)\left(\frac{p^*_1\bar{\mathfrak{e}'}}{p'^*_1\bar{\mathfrak{e}}}\right)\right)\\
&=&-\frac{1}{\ell_0}\left(G(v)\left(t^{\bf c}\right)-v\left(\frac{\mathfrak{e}'}{e}\right)\right)=\frac{1}{\ell_0}\left(v\left(\frac{\mathfrak{e}}{\mathfrak{e}'}\right)-{\bf c}\right)=\tilde{\theta}_\xi(v).
\end{eqnarray*}

\end{proof}
%Assume $\mcF^x R_m=(\sum_{j=1}^{} \mcF^x R_m)_{\alpha_j}$

\begin{prop}
Let $\mcF$ be a $\bT$-equivariant filtration and $\xi\in N_\bR$. 
%Let $(\ccZ^{\mcF_\xi}_m, \ccL^{\mcF_\xi}_m)$ (resp. $(\ccZ^\mcF_m, \ccL^\mcF_m)$) be the $m$-th approximation of $\mcF_\xi$ (resp. $\mcF$). Then we have the equivalence:
%\begin{equation}
%(\ccZ^{\mcF}_{m,\xi}, \ccL^{\mcF}_{m,\xi})\thickapprox (\ccZ^{\mcF_\xi}_m, \ccL^{\mcF_\xi}_m).
%\end{equation}
For any $w\in (\Blue{\Zdiv})^\bT$, we have the following identities:
\begin{eqnarray}
\phi^{\mcF_\xi}_m(w)&=&\phi^{\mcF}_m(w_\xi)+\theta_\xi(w)\label{eq-phimFtwist} \\
\phi^{\mcF_\xi}(w)&=&\phi^{\mcF}(w_\xi)+\theta_\xi(w). \label{eq-phiFtwist}
\end{eqnarray}

\end{prop}
\begin{proof}
Note that the second identity is obtained from the first one by letting $m\rightarrow+\infty$. So we just need to prove the first identity.
Set
\begin{equation}
(I^{\mcF_\xi}_{m,x})_\alpha={\rm Im}\left((\mcF^x R_m)_\alpha\otimes \mcO_Z(m\ell_0 L)\rightarrow \mcO_Z\right).
\end{equation}
By definitions in \eqref{eq-Imx} and \eqref{eq-Ftwist}, we have an identity of ideals:
\begin{equation}\label{eq-Iximx}
(I^{\mcF_\xi}_{m,x})_\alpha=(I^{\mcF}_{m,x-\la \alpha, \xi\ra})_\alpha
\end{equation}
So by \eqref{eq-tcIm} we have identities of fractional ideals:
\begin{equation}
\tilde{\mcI}^\mcF_m=\sum_x\sum_\alpha (I^\mcF_{m,x})_\alpha t^{-x}, \quad \tilde{\mcI}_{m}^{\mcF_\xi}=\sum_x \sum_\alpha (I^\mcF_{m,x-\la \alpha, \xi\ra})_\alpha t^{-x}
\end{equation}

Using the expression of \Blue{non-Archimedean potential} associated to filtrations in \eqref{eq-NAphiF} to $\phi^{\mcF_\xi}$ (see \eqref{eq-phiFm2}) and using the $(\bC^*\times\bT)$-invariance of the valuation of any $G(w)$, we indeed get \eqref{eq-phimFtwist}: 
\begin{eqnarray*}
-\phi_m^{\mcF_\xi}(w)&=&\frac{1}{m\ell_0}\min_{\alpha}\min_x\left(w((I^{\mcF_\xi}_{m,x})_\alpha)-x\right)\\
&=&\frac{1}{m\ell_0}\min_{\alpha}\min_x\left(w((I^{\mcF}_{m,x-\la \alpha, \xi\ra})_\alpha) -x\right)\\
&=&\frac{1}{m\ell_0}\min_\alpha \min_x\left(w((I^{\mcF}_{m,x})_\alpha)-x-\la \alpha, \xi\ra\right)\\
&=&-\theta_\xi(w)-\frac{1}{m\ell_0}\min_{\alpha}\min_x\left(w_\xi((I^\mcF_{m,x})_\alpha)-x\right) \quad (\text{by } \eqref{eq-vxis})\\
&=&-\theta_\xi(w)-\phi^{\mcF}_m(w_\xi).
\end{eqnarray*}
\end{proof}

\begin{lem}
For any $\xi\in N_\bR$, the following identities hold true:
\begin{eqnarray}
\bL^\NA(\mcF_\xi)&=&\bL^\NA(\mcF); \label{eq-bLFtwist}\\
\bfE^\NA(\mcF_\xi)&=&\bfE^\NA(\mcF)-\Fut_{(Z,Q)}(\xi); \label{eq-EFtwist}\\
\bfD^\NA(\mcF_\xi)&=&\bfD^\NA(\mcF)+\Fut_{(Z,Q)}(\xi).\label{eq-DFtwist}
\end{eqnarray}
In particular, if $\Fut_{(Z,Q)}\equiv 0$, then $\bfE^\NA(\mcF_\xi)=\bfE^\NA(\mcF)$ and $\bfD^\NA(\mcF_\xi)=\bfD^\NA(\mcF)$.
\end{lem}
\begin{proof}
By \eqref{eq-phiFtwist} and \eqref{eq-thetaxi2}, we get
\begin{eqnarray}\label{eq-A+phivxi}
A_{(Z,Q)}(v)+\phi_{m,\xi}(v)=A_{(Z,Q)}(v)+\phi_m(v_\xi)+\theta_\xi(v)=A_{(Z,Q)}(v_\xi)+\phi_m(v_\xi)
\end{eqnarray}
where $\phi_m=\phi^\mcF_m=\phi_{(\ccZ^{\mcF}_m, \ccL^{\mcF}_m)}$ (see \eqref{eq-phiFTC}). 
Taking infimum for $v$ ranging in $\rVal$ and using \eqref{eq-FeqTC}, we get the identity $\bfL^\NA(\phi_{m,\xi})=\bfL^\NA(\phi_m)$. The identity \eqref{eq-bLFtwist} follows by letting $m\rightarrow +\infty$ and using the definition \eqref{eq-LNAcF}.
 
Next choose a basis $\{s^{(m)}_1,\dots, s^{(m)}_{N_m}\}$ adapted to the filtration $\{\mcF^x R_m\}$, which means that
\begin{equation}
\mcF^x R_m={\rm span}\{s^{(m)}_1, \dots, s^{(m)}_{k_x}\}
\end{equation}
for some $k_x\in \{1,\dots, N_m\}$. Because $\mcF^x R_m$ is $(\bC^*)^r$-invariant, we can assume that $s^{(m)}_j$ are equivariant in the sense that:
\begin{equation}
\tau\circ s^{(m)}_j=\tau^{\alpha^{(m)}_j} \cdot s^{(m)}_j.
\end{equation}
Let $\lambda^{(m)}_1\ge \lambda^{(m)}_2\cdots\ge \lambda^{(m)}_{N_m}$ be the succesive minima. % Then we have:
%\begin{equation}
%{\rm span}\{s^{(m)}_1,\dots, s^{(m)}_j\}=\mcF^{\lambda^{(m)}_j}R_m.
%\end{equation}
Because of the $\bT$-equivariance, 
\begin{equation}
\lambda^{(m)}_j+\la \alpha^{(m)}_j, \xi\ra=:\lambda^{(m)}_j+\kappa^{(m)}_j, \quad j=1,\dots, N_m,
\end{equation}
are the set of successive minima for the twisted filtration. So we get:
\begin{eqnarray}
\bfE^\NA(\mcF_{\xi})&=&\frac{1}{N_m}\lim_{m\rightarrow+\infty}\sum_{j=1}^{N_m} \frac{\lambda^{(m)}_j+\kappa^{(m)}_j}{m\ell_0}\nonumber\\
&=&\bfE^\NA(\mcF)+\chw_L(\xi). \label{eq-EFxiCW}
\end{eqnarray}
Finally recall that in our set-up, $\chw_L(\xi)=-\Fut_{(Z,Q)}(\xi)$ (see \eqref{eq-CWFut}).

\end{proof}

\begin{defn}
For any $v\in \Blue{\Zdiv}$, define the invariant:
\begin{equation}
\beta(v):=\beta_{(Z, Q)}(v)=A_{(Z, Q)}(v)-S_{L}(v).
\end{equation}
\end{defn}
\begin{prop}
For any $v\in \Blue{\Zdiv}$ we have the inequality:
\begin{equation}\label{eq-betavsbfD}
\beta(v)\ge \bfD^\NA(\mcF_v).
\end{equation}
Moreover for any $\xi\in N_\bR$, we have the identity:
\begin{equation}\label{eq-betavxi}
\beta(v_\xi)=\beta(v)+\Fut_{(Z,Q)}(\xi).
\end{equation}
\end{prop}
\begin{proof}
Recall that 
\begin{equation}
\Blue{\bfD^\NA}(\mcF_v)=\Blue{\bfD^\NA}(\phi_v)=-\bfE^\NA(\phi_v)+\bL^\NA(\phi_v).
\end{equation}
By \eqref{eq-ENAFv}, we have 
\begin{eqnarray*}
S_{L}(v)&=&\bfE^\NA(\mcF_v)=\frac{1}{\ell_0^n L^{\cdot n}}\int_0^{+\infty} -\frac{x}{\ell_0} \cdot d\;\vol(\mcF^{(x)}R^{(\ell_0)})
\end{eqnarray*}
Moreover, using the inequality \eqref{eq-weakL} and $\phi_v(v)=0$ (by Lemma \ref{lem-phivv=0}), we always have:
\begin{equation}\label{eq-weak1}
\bfL^\NA(\cF_v)\le \inf_w \left(A(w)+\phi_v(w)\right)\le A(v).
\end{equation}
So we get \eqref{eq-betavsbfD}. %Assume $\lambda^{(m)}_{\max}=\lambda^{(m)}_1 \ge \lambda^{(m)}_2\ge\cdots \ge \lambda^{(m)}_{N_m}=\lambda^{(m)}_{\min}$ is the succesive minima of $R_m$. 
Because by \eqref{eq-Fvxi} $\mcF_{v_\xi}=\mcF_\xi(\theta_\xi(v))$ (see \eqref{eq-Fshift}), we use \eqref{eq-EFtwist} and \eqref{eq-thetaxi2} to 
%, the successive minima of $\mcF_{v_\xi}$ on $R_m$ consists of numbers $\{\lambda^{(m)}_j+\kappa^{(m)}_j+\theta_\xi(v)\}$ with $\kappa^{(m)}_j=\la \alpha^{(m)}_j, \xi\ra$. 
get the identity \eqref{eq-betavxi}:
\begin{eqnarray*}
S_L(v_\xi)&=&\bfE^\NA(\mcF_{v_\xi})=\bfE^\NA(\mcF_\xi(\theta_\xi(v))\\
&=&\bfE^\NA(\mcF_v)-\Fut_{(Z,Q)}(\xi)+\theta_\xi(v)\\
&=&S_{L}(v)-\Fut_{(Z,Q)}(\xi)+A(v_\xi)-A(v).
\end{eqnarray*}
\end{proof}

\subsection{$\bG$-Uniform Ding stability}
Let $(Z, Q)$, $L=-K_Z-Q$, $\bG$ and $\bT$ be as before. 
\begin{defn}
For any $\bT$-equivariant test configuration $(\mcZ, \mcQ, \mcL)$ of $(Z, Q, L)$,
the reduced $\bfJ$-norm of $(\mcZ, \mcL)$ is defined as:
\begin{equation}\label{eq-JNAT}
\bfJ^\NA_\bT(\mcZ, \mcL)=\inf_{\xi\in N_\bR} \bfJ^\NA(\mcZ_\xi, \mcL_\xi).
\end{equation}
For any graded filtration $\mcF$, its reduced $J$-norm is defined as:
\begin{equation}
\bfJ^\NA_\bT(\mcF)=\inf_{\xi\in N_\bR} \bfJ^\NA(\mcF_\xi).
\end{equation}
\end{defn}
The reason for defining $\bfJ^\NA_\bT$ comes from Hisamoto's slope formula:
\begin{thm}[{\cite[Theorem B]{His16b}}]\label{thm-JTslope}
Let $(\mcZ, \mcL)$ be a $\bT$-equivariant ample \Blue{test configuration} for $(Z, L)$. Let $\Phi=\{\vphi(s); s=-\Blue{\log|t|^2}\in [0, +\infty)\}$ be a bounded psh Hermitian metric on $\mcL$. Then we have the following limit formula:
\begin{equation}\label{eq-JTslope}
\lim_{s\rightarrow+\infty} \frac{\bfJ_\bT(\vphi(s))}{s}=\bfJ^\NA_\bT(\mcZ, \mcL).
\end{equation}
\end{thm}
For the convenience of the reader, we provide a refined proof of this result essentially following the argument in \cite{His16b} (which builds on some ideas of Berman). This will show that the arguments indeed work for any normal projective varieties. 
%Note that although Hisamoto works on smooth manifold, his argument uses only pluripotential theory (e.g. quasi-triangle inequality) which works for normal projective varieties.
\begin{proof}[Proof of \ref{thm-JTslope}]
%We can assume that $\Phi$ is a locally bounded weak geodesic ray. 
%Then one can easily check that $(s, \xi)\mapsto \bfJ(\sigma_\xi(s)^*\vphi(s))$ is a convex function on $\bR\times N_\bR$ (see Proposition \ref{prop-convex1}).
%So it is easy to see that $f(s):=\inf_{\xi\in N_\bR}\bfJ(\sigma_\xi(s)^*\vphi(s))$ is a convex function on $s$. So the limit in \eqref{eq-JTslope} exists. 
%Denote left and right sides of \eqref{eq-JTslope} by $a$ and $b$ respectively. 
First, by using the slope formula for $\bfJ$ and the definition of $\bfJ_\bT$ as an infimum, it is easy to verify:
\begin{equation}
\limsup_{s\rightarrow+\infty}\frac{\bfJ_\bT(\vphi(s))}{\Blue{s}}\le \bfJ^\NA_\bT(\mcZ, \mcL).
\end{equation}
% the inequality ``$a\le b$" in \eqref{eq-JTslope}. 
 For the other direction, by Lemma \ref{lem-infobt}, there exists $\xi_s\in N_\bR$ such that $\bfJ_\bT(\vphi(s))=\bfJ(\sigma_{\xi_s}(s)^*\vphi(s))$. By the quasi-triangle inequality for $\bfI$ (\cite[Theorem 1.8]{BBEGZ}) and hence for $\bfJ$, we have (for any fixed reference metric $\psi$):
\begin{eqnarray*}
\bfJ_\psi(\sigma_{\xi_s}(s)^*\psi)&\le& c_n \left(\bfJ_\psi(\sigma_{\xi_s}(s)^*\vphi(s))+\bfJ_{\sigma_{\xi_s}(s)^*\vphi(s)}(\sigma_{\xi_s}(s)^*\psi\right)\\
&\le& C \bfJ_\psi(\vphi(s))= C (\bfJ^\NA(\mcZ, \mcL) s+o(s))\le C's.
\end{eqnarray*}
By the properness of $\xi\mapsto \bfJ(\sigma_\xi(1)^*\psi)$ (Lemma \ref{lem-infobt}) and the identity $\sigma_{\xi_s}(s)=\sigma_{s\xi_s}(1)$,  this means that $\xi_s$ is uniformly bounded in $N_\bR$. Hence there exists $\Blue{\xi_\infty}\in N_\bR$ and a sequence $s_j\rightarrow+\infty$ such that $\xi_{s_j}\rightarrow \xi_\infty$.
We just need to show that
\begin{equation}\label{eq-JTslopebound}
\lim_{j\rightarrow+\infty}s_j^{-1}\left|\bfJ_\psi(\sigma_{\xi_{s_j}}(s_j)^*\vphi(s_j))-\bfJ_\psi(\sigma_{\xi_\infty}(s_j)^*\vphi(s_j))\right|=0,
\end{equation}
since it would imply the following inequality which concludes the proof:
\begin{eqnarray*}
\liminf_{j\rightarrow+\infty}\frac{\bfJ(\sigma_{\xi_{s_j}}(s_j)^*\vphi(s_j))}{\Blue{s_j}}&=&\lim_{s\rightarrow+\infty}\frac{\bfJ(\sigma_{\xi_\infty}(s_j)^*\vphi(s_j))}{\Blue{s_j}}\\
&=&\bfJ^\NA(\mcZ_{\xi_\infty}, \mcL_{\xi_\infty})\ge \bfJ^\NA_{\bT}(\mcZ, \mcL).
\end{eqnarray*}
To verify \eqref{eq-JTslopebound}, we use the easy fact $\left|\bfJ(\vphi_1)-\bfJ(\vphi_2)\right|\le 2\sup_X|\vphi_1-\vphi_2|$ to reduce to showing:
\begin{eqnarray}\label{eq-JTslopebound2}
\lim_{j\rightarrow+\infty}s_j^{-1}\sup_X |\sigma_{\xi_{s_j}}(s_j)^*\vphi(s_j)-\sigma_{\xi_\infty}(s_j)^*\vphi(s_j)|=0.
\end{eqnarray}
Now we fix a $\bC^*\times\bT$-equivariant embedding $\iota: \mcX\rightarrow \bP^{N_k-1}\times \bC$ with that $\iota^*\mathcal{O}_{\bP^{N_k-1}}(1)=\mcL^k$. The weight decomposition of $H^0(X, kL)$ allows us to choose homogeneous coordinates
$\{Z_1, \dots, Z_{N_k}\}$ on $\bP^{N_k-1}$ such that the $\bC^*\times\bT$-action is given by:
\begin{equation}
(\tau_0, \tau_1,\dots, \tau_r)\cdot Z_i=\tau_0^{\lambda_i}\prod_{p=1}^r \tau_k^{\alpha_i^p} \cdot Z_i.
\end{equation} 
Identify $X$ with the fibre at $t=1$: $X\cong \pi^{-1}(\{1\})\cap \mcX$, and set $e^{-\psi_{\FS}}=\left.\iota^*h_{\FS}^{1/k}\right|_X$ where $h_{\FS}$ is the standard Fubini-Study metric on $\bP^{N_k-1}$.
Then to verify \eqref{eq-JTslopebound}, we can replace %the weak geodesic ray 
$\{\vphi(s)\}$ by the $L^\infty$-comparable $\{\tilde{\vphi}(s)=\sigma_\eta(s)^*\psi_{\FS}\}$, which is given by the well-known explicit formula (recall that $s=-\log|t|$):
\begin{eqnarray*}
\tilde{\vphi}(s)-\psi_{\FS}&=&\frac{1}{k}\log\frac{ \sum_{i=1}^{N_k} |t|^{-2 \lambda_i}|Z_i|^2}{\sum_{i=1}^{N_k} |Z_i|^2}.
\end{eqnarray*}
More generally, for any $\xi\in N_\bR$, $\sigma_\xi(s)^*\tilde{\vphi}(s)$ is given by:
\begin{eqnarray*}
\sigma_\xi(s)^*\tilde{\vphi}(s)-\psi_{\FS}&=&\frac{1}{k}\log\frac{\sum_i |t|^{-2(\lambda_i+\la \alpha_i, \xi\ra)}|Z_i|^2}{\sum_i |Z_i|^2}.
\end{eqnarray*}
\Blue{Note that $N_\bR\cong \bR^r$ has a standard Euclidean norm. }
So we easily get for any $\xi, \xi'\in N_\bR$ (again with $s=-\log|t|$), 
\begin{equation}
\left|\sigma_\xi(s)^*\tilde{\vphi}(s)-\sigma_{\xi'}(s)^*\tilde{\vphi}(s)\right|=\frac{1}{k}\left|\log\frac{\sum_i |t|^{-2(\lambda_i+\la \alpha_i, \xi\ra)}|Z_i|^2}{\sum_i |t|^{-2(\lambda_i+\la \alpha_i, \xi'\ra} |Z_i|^2}\right|\le C (\log|t|^2) |\xi-\xi'|,
\end{equation}
where $C=C(k, \{\alpha_i\})$ does not depend on $\xi, \xi', s$. Substituting the variables $\xi$, $\xi', s$ by $\xi_{s_j}, \xi_\infty, s_j$ respectively into the above estimate, we easily get the limit \eqref{eq-JTslopebound2} by using the fact that $\xi_{s_j}\rightarrow \xi_\infty$.
\end{proof}

The next lemma generalizes \cite[Lemma 3.18]{His19}:
\begin{lem}\label{lem-JNAproper}
Assume $\chw_L\equiv 0$ on $\mathfrak{t}$. Then for any $\bT$-equivariant filtration $\mcF$ (satisfying the properties in Definition \ref{defn-gdfiltr}), $\xi\mapsto \bfJ^\NA(\mcF_\xi)$ is a convex and proper function. More precisely, there exists $C_1>0$ depending only on the $\bT$-action on $Z$, such that
\begin{equation}\label{eq-JNAproper}
\bfJ^\NA(\cF_\xi)\ge C_1 |\xi|-(e_-+\bfE^\NA(\cF)),
\end{equation}
where $e_-$ is any number satisfying $\cF^{m e_-}=0$ for $m\in \bN$ (see Definition \ref{defn-gdfiltr})).
As a consequence, it has a unique minimizer on $N_\bR$. Moreover if $\mcF=\mcF_{(\mcZ, \ell_0\mcL)}$ for some test configuration $(\mcZ, \mcL)$ of $(Z, L)$, then the minimizer is contained in $N_\bQ$.
\end{lem}
\begin{proof}
Assume that $m$ is sufficiently divisible such that $m\ell_0L$ is globally generated. Let 
\begin{equation}
\lambda^{(m)}_1\ge \lambda^{(m)}_2\ge \cdots\ge \lambda^{(m)}_{N_m}
\end{equation}
be the successive minima of $\mcF R_m$. Then we have %for $m$ sufficiently divisible such that $m \ell_0 \mcL$ is globally generated, we have:
\begin{eqnarray}
\bfJ^\NA(\mcF_\xi)&=&\Lam^\NA(\mcF_\xi)-\bfE^\NA(\mcF_\xi)  \quad \quad (\text{see } \eqref{eq-lamaxF}-\eqref{eq-JNAF}) \nonumber \\
&=&\sup_m \max_j \frac{\lambda^{(m)}_j+\la \alpha^{(m)}_j, \xi\ra }{m\ell_0}-\bfE^\NA(\mcF) \quad (\text{by \eqref{eq-EFtwist}}) \label{eq-JNAconvexF}\\
&\ge&\max_j \frac{\la\alpha^{(m)}_j, \xi\ra}{m\ell_0}-(e_-+\bfE^\NA(\cF)). \label{eq-JNAproperF}
\end{eqnarray}
The second identity used \eqref{eq-EFxiCW} and Proposition \ref{BHJvol}. The last inequality is because by definition \ref{defn-gdfiltr} $\mcF$ is linearly bounded from below: $\lambda^{(m)}_j\ge m\ell_0 e_-$.
From the expression \eqref{eq-JNAconvexF} it is clear that $\xi\mapsto \bfJ^\NA(\mcF_\xi)=:\bfj(\xi)$ is a convex function in $\xi\in N_\bR$. We will show it is a proper function.
Let ${\bf P}\subset M_\bR$ be closed convex hull of the set:
\begin{equation}
\left\{\frac{\alpha^{(m)}_j}{m \ell_0};\quad j=1,\dots, N_m, m\in \bZ_{\ge 0} \right\}.
\end{equation}
The following measure is supported on ${\bf P}$.
\begin{equation}
\DHM_\bT=\lim_{m\rightarrow+\infty} \frac{1}{N_m} \sum_{m}\delta_{\frac{\alpha^{(m)}_j}{m\ell_0}}.
\end{equation}
By \cite[Proposition 6.4]{BHJ17} (see also \cite[Proposition 2.1]{Bri87} and \cite{Oko96}), $P$ is a rational polytope and $\DHM_\bT$ is absolutely continuous with respect to the Lebesgue measure.
The Chow weight of $\xi$ is then given by:
\begin{equation}
\chw_{L}(\xi)=\lim_{m\rightarrow+\infty}\frac{1}{N_m}\sum_{m}\frac{\la \alpha^{(m)}_j, \xi\ra}{m\ell_0}=\int_{\bf P} \la y, \xi\ra \DHM_\bT=\vol({\bf P}) \cdot \la {\rm bc}_\bT, \xi\ra,
\end{equation}
where ${\rm bc}_\bT$ is the barycenter of $\DHM_\bT$.

If  $\chw \equiv 0$ on $\mathfrak{t}$, then ${\rm bc}_\bT=0$. This implies that $0$ is in the interior of ${\bf P}$. If $\Delta$ denotes the standard simplex, then there exists $\theta>0$ such that $\theta\Delta\subset {\bf P}$. So for any $\epsilon>0$ and $k=1,\dots, n$, there exist $m=m(\epsilon)\gg 1$ and $\alpha^{(m)}_{j^{\pm}_k}$, such that 
\begin{equation}
\left|\frac{\alpha^{(m)}_{j^+_k}}{m\ell_0}-\theta {\bf e}_k\right|\le \epsilon, \quad \left|\frac{\alpha^{(m)}_{j^-_k}}{m\ell_0}+\theta {\bf e}_k\right|\le \epsilon.
\end{equation}
So we get the inequality:
\begin{equation}\label{eq-JNAproper2}
\left\la \frac{\alpha^{(m)}_{j^{\pm}_k}}{m\ell_0}, \xi\right\ra\ge \theta |\xi_k|-\epsilon |\xi|, \text{ for all } k.
\end{equation}
Combining this with \eqref{eq-JNAproperF}, we indeed get the properness of $\bfj(\xi)$:
\begin{equation}
\bfj(\xi)\ge \left(\frac{\theta}{\sqrt{n}}-\epsilon\right)|\xi|
\end{equation}
Now assume $\mcF=\mcF_{(\mcZ, \ell_0 \mcL)}$. When $m$ is sufficiently divisible such that $m\ell_0 \mcL$ is globally generated, we have the identity:
\begin{eqnarray}
\bfJ^\NA(\mcZ_\xi,\mcL_\xi)&=&\Lam^\NA(\mcZ_\xi, \mcL_{\xi})-\bfE^\NA(\mcZ_\xi, \mcL_\xi)  \nonumber \\
&=&\max_j \frac{\lambda^{(m)}_j+\la \alpha^{(m)}_j, \xi\ra }{m\ell_0}-\bfE^\NA(\mcZ, \mcL) \label{eq-JNAconvexTC}.%\\
%&\ge&\max_j \frac{\la\alpha^{(m)}_j, \xi\ra}{m\ell_0}-C_1. \label{eq-JNAproperTC}
\end{eqnarray}
We see that in this case $\bfj$ is a rationally piecewisely linear, convex and proper function on $N_\bR$. So it obtains a minimum at some $\xi\in N_\bQ$.

\end{proof}

\begin{prop}\label{prop-limJT}
Assume $\chw_L(\xi)\equiv 0$ on $N_\bR$. Let $\cF$ be a filtration and $(\ccZ_m, \ccL_m)$ be the $m$-th approximating test configurations of $\mcF$ in Definition \ref{defn-ckTCm}. Then we have:
\begin{equation}
\limsup_{m\rightarrow+\infty}\bfJ^\NA_\bT(\ccZ_{m}, \ccL_{m})=\bfJ^\NA_\bT(\mcF).
\end{equation}

\end{prop}

\begin{proof}
By definition, we need to prove that:
\begin{equation}
{\bf I}:=\limsup_{m\rightarrow+\infty}\inf_{\xi\in N_\bR} \bfJ^\NA(\ccZ_{m,\xi}, \ccL_{m,\xi})=\inf_{\xi\in N_\bR} \bfJ^\NA(\mcF_\xi)=:{\bf II}.
\end{equation}
We first claim that for any $\xi\in N_\bR$:
\begin{equation}\label{eq-Jmxiconv}
\lim_{m\rightarrow+\infty} \bfJ^\NA(\ccZ_{m,\xi}, \ccL_{m,\xi})=\bfJ^\NA(\mcF_\xi).
\end{equation}
Indeed, by \eqref{eq-phimFtwist} we know $\phi^{\mcF_\xi}_m=\phi^{\mcF}_{m,\xi}$. 
On the other hand, by definition (see \eqref{eq-phiFTC})
$\phi^{\mcF}_{m,\xi}=\phi_{(\ccZ_m, \ccL_m),\xi}$. So we get:
\begin{equation}
\bfJ^\NA(\ccZ_{m,\xi}, \ccL_{m,\xi})=\bfJ^\NA\left(\phi_{(\ccZ_m, \ccL_m),\xi}\right)=\bfJ^\NA(\phi^{\mcF}_{m,\xi})=\bfJ^\NA(\phi^{\mcF_\xi}_m).
\end{equation}
So \eqref{eq-Jmxiconv} follows from \eqref{eq-limJccZm}. \eqref{eq-Jmxiconv} easily implies that ${\bf I}\le {\bf II}$, since for any $\xi\in N_\bR$, we then have:
\begin{equation}
\limsup_{m\rightarrow+\infty} \inf_{\xi'\in N_\bR}\bfJ^\NA(\ccZ_{m,\xi'}, \ccL_{m,\xi'})\le \lim_{m\rightarrow+\infty} \bfJ^\NA(\ccZ_{m,\xi}, \ccL_{m,\xi})=\bfJ^\NA(\mcF_\xi).
\end{equation}
We only need to prove ${\bf II}\le {\bf I}$. 

For simplicity of notations, set:
\begin{eqnarray*}
\bfj_m(\xi)&:=&\bfJ^\NA(\ccZ_{m,\xi}, \ccL_{m,\xi})=\Lam^\NA(\ccZ_{m,\xi}, \ccL_{m,\xi})-\bfE^\NA(\ccZ_{m,\xi}, \ccL_{m,\xi})\\
&=&\Lam^\NA(\ccZ_{m,\xi}, \ccL_{m,\xi})-\bfE^\NA(\ccZ_{m}, \ccL_{m})=:\mathfrak{f}_m(\xi)+\mathfrak{g}_m.\\
\bfj(\xi)&:=&\bfJ^\NA(\mcF_\xi)=\Lam^\NA(\cF_\xi)-\bfE^\NA(\cF_\xi)=\Lam^\NA(\cF_\xi)-\bfE^\NA(\cF)=:\mathfrak{f}(\xi)+\mathfrak{g}.
\end{eqnarray*}
Here we used \eqref{eq-EYBtwist}, \eqref{eq-EFtwist} and the assumption that $\Fut(\xi)=-\chw_L(\xi)=0$ \Blue{(see \eqref{eq-CWFut})} to see that $\mathfrak{g}_m$ and $\mathfrak{g}$ are constant functions on $N_\bR$. 

By \eqref{eq-limEccZm}, we know that $\lim_{m\rightarrow+\infty}\mathfrak{g}_m=\mathfrak{g}$.
By \eqref{eq-JNAproper} from Lemma \ref{lem-JNAproper}, we know that $\bfj_m(\xi)$ and $\bfj(\xi)$ satisfies the uniform properness estimates: there exist $C_1, C_2>0$ such that for any $\xi\in N_\bR$, we have
\begin{equation}
\bfj_m(\xi)\ge C_1|\xi|-C_2, \quad \bfj(\xi)\ge C_1|\xi|-C_2.
\end{equation}
So the infimum $\inf_{\xi\in N_\bR} \bfj_m(\xi)$ and $\inf_{\xi\in N_\bR}\bfj(\xi)$ are obtained on a uniformly bounded set of $\xi$, which we denote by $\Xi_{C_3}=\{\xi\in N_\bR; |\xi|\le C_3\}$. 

Moreover, by the proof of Lemma \ref{lem-JNAproper}, $\mathfrak{f}_m$ and $\mathfrak{f}$ are all convex functions on $\bR^r$. So $\mathfrak{f}_m$ are $\mathfrak{f}$ are continuous on $\bR^r$. Choose $m_p:=k^{p}, p\in \bN$ for some $k\in \bN$ sufficiently divisible. By Remark \ref{rem-supadd}, for any $\xi\in N_\bR$, $\mathfrak{f}_{m_p}(\xi)=\frac{\lambda^{(\max)}_{k^p}(\cF_{\xi})}{k^p}$ is increasing. So $\{\mathfrak{f}_{m_p}\}_{p\in \bN}$ is an increasing sequence of continuous functions converging pointwise to $\mathfrak{f}$ as $p\rightarrow+\infty$. By Dini's theorem, $\mathfrak{f}_{m_p}$ converges to $\mathfrak{f}$ uniformly on the compact set $\Xi_{C_3}$. 
As mentioned above, $\mathfrak{g}_m$ and $\mathfrak{g}$ do not depend on $\xi$ (because of the vanishing $\Fut\equiv 0$ on $N_\bR$). 
%On the other hand, because $\mathfrak{g}_m$ and $\mathfrak{g}$ are all linear, the convergence $\mathfrak{g}_m\rightarrow\mathfrak{g}$ is also uniform on $\Xi_{C_3}$. 
So we know that, as $p\rightarrow+\infty$, $\mathfrak{j}_{m_p}$ converges to $\mathfrak{j}$ uniformly over $\Xi_{C_3}$. So the convergence of infimum (over $\Xi_{C_3}$) also follows.
\end{proof}

%\begin{rem}
%By Proposition \ref{prop-Fvxi} we know that $\mcF_{v_\xi}=\mcF_\xi(\theta_\xi(v))$. 
%So we get:
%\begin{equation}
%(\ccZ^{\mcF}_{m,\xi}, \ccL^{\mcF}_{m,\xi})=(\ccZ^{\mcF_\xi}_m, \ccL^{\mcF_\xi}_m)=(\ccZ^{\mcF_{v_\xi}}_m, \ccZ^{\mcF_{v_\xi}}_m).
%\end{equation}
%
%There exists $C_4$ such that  $A(v_\xi)\le C_4$ for any $\xi\in \Xi_{C_3}$. 
%By \cite[Section 5]{BlJ20}, we know that for any $C_5>0$ there exists $m_1>0$ such that for any $m\ge m_1$ and $\xi\in \Xi_{C_3}$, we have:
%\begin{eqnarray*}
%&&|\lambda_{\max}(\ccZ_{m,\xi}, \ccL_{m,\xi})-\lambda_{\max}(\mcF_{v_\xi})|\le \frac{C_5}{m}. %,\quad |\bfE^\NA(\ccZ_{m,\xi}, \ccL_{m,\xi})-\bfE^\NA(\mcF_{v_\xi})|\le \epsilon.
%\end{eqnarray*}
%Because $\bfJ^\NA$ is translation invariant, $\bfJ^\NA(\mcF_\xi)=\bfJ^\NA(\mcF_{v_\xi})$.
%So we get that for any $\xi\in \Xi_{C_4}$,
%\begin{eqnarray}\label{eq-JNAuniconv}
%|\bfJ^\NA(\ccZ_{m,\xi}, \ccL_{m,\xi})-\bfJ^\NA(\mcF_{\xi})|&\le& |\lambda_{\max}(\ccZ_{m,\xi}, \ccL_{m,\xi})-\lambda_{\max}(\mcF_\xi)|\nonumber \\
%&&\hskip 2cm+|\bfE^\NA(\ccZ_{m,\xi}, \ccL_{m,\xi})-\bfE^\NA(\mcF_{v_\xi})|\nonumber \\
%&\le &\frac{C_5}{m}+|\bfE^\NA(\ccZ_{m}, \ccL_{m})-\bfE^\NA(\mcF_{v})|\stackrel{m\rightarrow+\infty}{\longrightarrow} 0.
%\end{eqnarray}
%This implies that $\bfj_m$ converges to $\bfj$ uniformly over $\Xi_{C_3}$.  So convergence of infimum (over $\Xi_{C_4}$) also follows.
%\end{rem}
\begin{rem}
One can also use the uniform estimates from \cite[section 5]{BlJ20} to get uniform convergence over $\Xi_{C_3}$ in the above proof.
\end{rem}

\begin{defn}[{see \cite{His16b, His19}}]\label{defn-Hisamoto}
$(Z, Q)$ is $\bG$-uniformly Ding-stable if there exists $\gamma>0$ such that for any $\bG$-equivariant test configuration $(\mcZ, \mcL)$ of $(Z, Q, L)$:
\begin{equation}
\bfD^\NA(\mcZ, \mcL)\ge \gamma\cdot \bfJ^\NA_\bT(\mcZ, \mcL).
\end{equation}
If one replaces $\bfD^\NA$ by $\Blue{\bfM^\NA}$, then one gets the definition of $\bG$-uniform K-stability.
\end{defn}
We should compare this notion with the following well-known definition:
\begin{defn}\label{defn-Gequivstability}
\begin{enumerate}
\item
$(Z, Q)$ is $\bG$-equivariantly uniformly Ding-stable if there exists $\gamma>0$ such that for any $\bG$-equivariant test configuration $(\mcZ, \mcL)$ of $(Z, Q, L)$:
\begin{equation}
\bfD^\NA(\mcZ, \mcL)\ge \gamma \cdot \bfJ(\mcZ, \mcL). 
\end{equation}
\item $(Z, Q)$ is $\bG$-equivariantly Ding-semistable if for any $\bG$-equivariant test configuration $(\mcZ,  \mcL)$ of $(Z, L)$:
\begin{equation}\label{eq-Dsemipoly}
\bfD^\NA(\mcZ,  \mcL)\ge 0. 
\end{equation}
$(Z, Q)$ is $\bG$-equivariantly Ding-polystable if $(Z, Q)$ is $\bG$-equivariantly Ding-semistable, and the identity in \eqref{eq-Dsemipoly} holds only when $(\mcZ, \mcQ, \mcL)$ is a product test configuration.
\end{enumerate}
If one replaces $\bfD^\NA$ by $\Blue{\bfM^\NA}$ in the above definition, one gets the definition of $\bG$-equivariantly uniform K-stability and so on. 
\end{defn}
\begin{rem}\label{rem-Gequivstability}
By running $\bC^*\times\bG$-equivariant MMP,
it is clear from the proof of \Blue{\cite[Theorem 1.4]{Fuj19a} (see also \cite{BBJ18})} (based on MMP process in \cite{LX14}) that $\bG$-equivariantly uniform Ding-stability is equivalent to $\bG$-equivariantly uniform K-stability. The same remark applies to $\bG$-equivariant semistability or polystability. \Blue{We refer to Appendix \ref{app-MMP} for more details. }
\end{rem}
Because $\bfJ^\NA_\bT\ge 0$, we see that $\bG$-uniform Ding-stability implies $\bG$-equivariant Ding-semistability, which in particular implies $\Fut_{(Z, Q)}\equiv 0$ on $\mathfrak{t}$.
In fact, $(Z, Q)$ \Blue{being $\bG$-uniformly Ding-stable} implies that $(Z, Q)$ is $\bG$-equivariantly Ding-polystable:
\begin{lem}[\cite{His16a, His16b}]\label{lem-uni2poly}
Assume $\chw_L\equiv 0$ on $\mathfrak{t}$. For any $\bT$-equivariant test configuration $(\mcZ, \mcQ, \mcL)$ of $(Z, Q, L)$,  $\bfJ_\bT(\mcZ, \mcL)=0$ if and only if $(\mcZ, \mcL)$ is a product test configuration generated by some $\eta\in N_\bZ$. As a consequence, 
if $(Z, Q)$ is $\bG$-uniformly Ding-stable, then for any $\bG$-equivariant test configuration $(\mcZ, \mcQ, \mcL)$ of $(Z, Q)$, $\bfD^\NA(\mcZ, \mcQ, \mcL)\ge 0$ and $=0$ if and only if $(\mcZ, \mcQ, \mcL)$ is a product test configuration generated by some $\eta\in N_\bZ$.
\end{lem}
\begin{proof}
By Lemma \ref{lem-JNAproper}, $\xi\mapsto J(\mcZ_\xi, \mcL_\xi)$ has a unique minimizer $\xi\in N_\bQ$. Assume $b\in \bN$ satisfies $b \xi\in N_\bZ$. Then we consider the test configuration $(\mcZ_\xi, \mcL_\xi)^{(b)}$ defined in \eqref{eq-TCxib}. Then 
\begin{equation}
\bfJ^\NA_\bT(\mcZ, \mcL)=\bfJ^\NA(\mcZ_\xi, \mcL_\xi)=b^{-1} \bfJ^\NA((\mcZ_\xi, \mcL_\xi)^{(b)})=0.
\end{equation} 
By \cite{BHJ17}, this implies $(\mcZ_\xi, \mcL_\xi)^{(b)}$ is a product test configuration which implies $(\mcZ, \mcL)$ itself is a product test configuration.
\end{proof}

\begin{prop}\label{prop-uDtestFv}
Assume that $(Z, Q)$ is $\bG$-uniformly Ding-stable. Then for any $v\in (\Blue{\Zdiv})^\bG$ with its associated filtration $\mcF_v$, we have:
\begin{equation}\label{eq-uDtestFv}
\bfD^\NA(\mcF_v)\ge \gamma \cdot \inf_{\xi\in N_\bR} \bfJ^\NA(\mcF_{v_\xi})=\gamma \cdot \bfJ^\NA_\bT(\mcF_v).
\end{equation}
\end{prop}
\begin{proof}
Let $(\ccZ_m, \ccQ_m, \ccL_m)$ be $m$-th approximating test configurations for $\mcF_v$ in Definition \ref{defn-ckTCm}. By $\bG$-uniform Ding-stability, we have:

\begin{equation}
\bfD^\NA(\ccZ_m, \ccL_m) \ge \gamma \cdot \inf_{\xi\in N_\bR} \bfJ^\NA(\ccZ_{m,\xi}, \ccL_{m,\xi}).
\end{equation}
Letting $m\rightarrow+\infty$ and using Proposition \ref{prop-JccXconv} and Proposition \ref{prop-limJT}, we get the conclusion.

\end{proof}

\begin{cor}\label{cor-Gunival}
If $(Z, Q)$ is $\bG$-uniformly Ding-stable, then there exists $\gamma'>0$ such that for any $v\in (\Blue{\Zdiv})^\bG$,
\begin{equation}
\sup_{\xi\in N_\bR}\left[A_{(Z, Q)}(v_\xi)-(1+\gamma')\cdot S_{L}(v_\xi)\right]\ge 0.
\end{equation}
\end{cor}

\begin{proof}
By the paragraph above Lemma \ref{lem-uni2poly}, we know that $\Fut_{(Z,Q)}\equiv 0$ on $\mathfrak{t}$. 
Because $\bfD^\NA(\mcF_\xi)=\bfD^\NA(\mcF)$, we see the inequality \eqref{eq-uDtestFv} in Proposition \ref{prop-uDtestFv} can be re-written as:
\begin{equation}\label{eq-uDsup1}
\sup_{\xi\in N_\bR}\left[-\bfE^\NA(\mcF_{v_\xi})+\bL^\NA(\mcF_{v_\xi})-\gamma\cdot  \bfJ^\NA(\mcF_{v_\xi})\right]\ge 0.
\end{equation}
On the other hand, recall that \eqref{eq-ENAFv}
\begin{equation}
\bfE^\NA(\mcF_{v_\xi})=S(v_\xi).
\end{equation}
Moreover by \eqref{eq-SvsJ} (see \cite[Proposition 2.1]{Fuj19b}), we know that:
\begin{equation}
\frac{1}{n} S(v_\xi) \le \bfJ^\NA(\mcF_{v_\xi})=\Lam^\NA(\mcF_{v_\xi})-S(v_\xi)\le n S(v_\xi).
\end{equation}
So, with $\Blue{\gamma'=\gamma n^{-1}}$, \eqref{eq-uDsup1} implies the inequality:
\begin{eqnarray*}
\sup_{\xi\in N_\bR}\left[\bfL^\NA(\cF_{v_\xi})-(1+\gamma') S_{L}(v_\xi)\right]\ge 0.
\end{eqnarray*}
Set $\phi_{v_\xi}=\phi^{\mcF_{v_\xi}}$ (see Definition \ref{defn-phiF}).
By \eqref{eq-weakL} and $\phi_{v_\xi}(v_\xi)=0$ (see Lemma \ref{lem-phivv=0}), we then have:
\begin{eqnarray}\label{eq-weak2}
\bL^\NA(\cF_{v_\xi})\le \inf_w (A(w)+\phi_{v_\xi}(w))\le A(v_\xi).
\end{eqnarray}
As a consequence, we get the inequality:
\begin{equation}
\sup_{\xi\in N_\bR}\left[ A(v_\xi)-(1+\gamma') S_{L}(v_\xi)\right]\ge 0.
\end{equation}
\end{proof}

\begin{cor}
If $(Z, Q)$ is $\bG$-uniformly Ding-stable, then for any $\bG$-invariant divisorial valuation $v\in \Blue{\Zdiv}$, we have $\beta(v)\ge 0$ and $\beta(v)=0$ if and only if $v=\wt_\xi$ for some $\xi\in N_\bR$.
\end{cor}
\begin{proof}
Fix any $v\in \Blue{\Zdiv}$, if $v=\wt_\xi$ for some $\xi\in N_\bR$, then $\beta(v)=\beta(\wt_\xi)=\Fut_{(Z,Q)}(\xi)=0$. Otherwise, there exists $\xi\in N_\bR$ such that
\begin{equation}
0\le A_{(Z,Q)}(v_\xi)-(1+\gamma') S_{(Z,Q)}(v_\xi)=\beta(v_\xi)-\gamma' S_{L}(v_\xi),
\end{equation}
which implies $\beta(v_\xi)\ge \gamma' S_{L}(v_\xi)>0$.
\end{proof}
\begin{rem}
We expect the converse to this result is also true.
\end{rem}

\section{Proof of Theorem \ref{thm-Gvalcriterion}}

\begin{proof}
Because $\bfM^\NA\ge \bfD^\NA$, so (2) implies  (1). 

We have pointed out in the paragraph below Remark \ref{rem-Gequivstability} that $\bG$-uniform Ding-stability implies that $\Fut_{(X,D)}\equiv 0$ on $\mathfrak{t}$. So (2) implying (3) follows from Corollary \ref{cor-Gunival}.

We prove (1) or (4) implies (2). Take any test configuration $(\mcX, \mcD, \mcL)$ for $(X, D, -(K_X+D))$. Because $\bG$ is \Blue{a} connected linear algebraic group, as explained in \Blue{Appendix} \ref{app-MMP} we can use $\bG$-equivariant MMP as in \cite{LX14} to get a special test configuration $(\mcX^s, \mcL^s)$. Moreover, there exists $d\in \bZ_{>0}$ such that, for any $\epsilon \in [0,1)$ and any $\xi\in N_\bR$, we have:
\begin{equation}\label{eq-D-Jxidecrease}
d(\bfD^\NA(\mcX, \mcL)-\epsilon \cdot \bfJ^\NA(\mcX_\xi, \mcL_\xi))\ge \bfD^\NA(\mcX^s,  \mcL^s)-\epsilon \cdot \bfJ^\NA(\mcX^s_{\xi}, \mcL^s_{\xi}).
\end{equation}
To verify the claim, first assume that $\xi\in N_\bZ$. \Blue{When $\xi=0$,  K. Fujita in \cite{Fuj19a} calculated  the variation of $\bfD^\NA-\epsilon\bfJ^\NA$ under the relative MMP process studied in \cite{LX14} by using intersection formulas on compactification of test configurations. We will explain this calculation in detail in Appendix \ref{app-MMP}. Recall that the compactification depends on the isomorphism between $(\mcX, \mcD, \mcL)\times_\bC\bC^*$ and $((X, D)\times\bC^*, p_1^*L)$ which is induced by the $\bC^*$-action. Assume that for the untwisted test configuration, the $\bC^*$-action is generated by $\eta$. 
Then to get the natural compactification of the $\xi$-twisted test configuration, we need to use the $\bC^*$-action generated by $\eta+\xi$ instead of the $\bC^*$-action generated by $\eta$. 
%Recall that (1) implies $\Fut_{(X, D)}\equiv 0$ on $\mathfrak{t}$. 
With this modification, \eqref{eq-D-Jxidecrease} follows directly from the same calculation as explained in Appendix \ref{app-MMP}. }

When $\xi\in N_\bQ$, choose $b\in \bN$ such that $b\xi\in N_\bZ$. Then by the discussion at the end of section \ref{sec-TCtwist} the $\xi$-twisted test configuration 
$(\mcX_\xi,  \mcL_\xi) $
is up to base change, or rescaling in terms of non-Archimedean \Blue{metrics}, equivalent to 
\begin{equation}
(\mcX, \mcL)^{(b)}:=\left(\text{normalization of } (\mcX, \mcL)\times_{\bC, {\rm m}_d}\bC, b \eta+b \xi\right)
\end{equation}
Then we can calculate the variation of intersection numbers on $(\mcX, \mcD, \mcL)^{(b)}$ to get inequality \eqref{eq-D-Jxidecrease}. For more details, we refer to section \ref{app-MMP}.

By continuity, \eqref{eq-D-Jxidecrease} holds for all $\xi\in N_\bR$.
Taking \Blue{the} supremum for $\xi$ ranging from $N_\bR$, we get:
\begin{equation}
\bfD^\NA(\mcX, \mcL)-\epsilon \bfJ^\NA_\bT(\mcX, \mcL)\ge \bfD^\NA(\mcX^s, \mcL^s)-\epsilon \bfJ^\NA_\bT(\mcX^s, \mcL^s).
\end{equation}
\Blue{So to check $\bG$-uniform Ding-stability, it suffices to check it on special test configurations. 
On the other hand, for a special test configuration, we have:
\begin{eqnarray*}
\bfM^\NA(\mcX^s_\xi,  \mcL^s_\xi)&=&
\bfD^\NA(\mcX^s_\xi, \mcL^s_\xi)=\bfD^\NA(\mcX^s, \mcL^s). 
\end{eqnarray*}
The second identity follows from \eqref{eq-DFtwist}.
So we get (1) implies (2), and also (4) implies (2). 
}

\Blue{
Now we prove that (3) implies (4). For a $\bG$-equivariant special test configuration $(\mcX^s, \mcL^s)$, if $v=\left.\ord(\mcX^s_0)\right|_{\bC(X)}$ denotes the $\bG$-invariant divisorial valuation obtained by restricting $\ord_{\mcX^s_0}$ to the functional sub-field $\bC(X)\subset \bC(X\times\bC)$ (see \cite[Lemma 4.1]{BHJ17}), then we claim to have the identities:
\begin{eqnarray}
\bfD^\NA(\mcX^s, \mcL^s)&=&A_{(X,D)}(v)-S_{L}(v)\nonumber \\
&=&A_{(X,D)}(v_\xi)-S_{L}(v_\xi)=\bfD^\NA(\mcX^s_\xi, \mcL^s_\xi).\label{eq-D2beta}
\end{eqnarray}
If this is true, then by using \eqref{eq-SvsJ}, we have:
\begin{align*}
A(v_\xi)-\delta_\bG \cdot S_L(v_\xi)&=\bfD^\NA(\mcX^s_\xi, \mcL^s_\xi)-(\delta_\bG-1)S_L(v_\xi)\\
&\le \bfD^\NA(\mcX^s, \mcL^s)-\frac{\delta_\bG-1}{n} \bfJ^\NA(\mcX^s_\xi, \mcL^s_\xi).
\end{align*}
So we get that (3) implies (4). To explain the claimed identities, note that, since the second identity in \eqref{eq-D2beta} follows from \eqref{eq-betavxi} and the last follows from \eqref{eq-DYBtwist}, 
we just need to explain the first identity using calculations in \cite{Fuj18, Li17} as follows.
First by \cite[Lemma 6.9]{Li17} we have
$\cF^t_{(\mcX^s, \mcL^s)}R_m=\cF_{v}^{t+m\ell_0 A(v)}R_m$ (with $A(v)=A_{(X,D)}(v))$ which implies 
\begin{equation}\label{eq-volshift}
\vol(\cF^{(t)}_{(\mcX^s, \mcL^s)})=\vol(\cF_{v}^{(t+\ell_0 A(v))}).
\end{equation} 
We can then calculate:
\begin{align*}
\bfD^\NA(\mcX^s, \mcL^s)&=-\bfE^\NA({\mcX^s, \mcL^s})=-\frac{1}{\ell_0^{n+1} L^{\cdot n}}\int_\bR (x-\ell_0 A(v))(-d\vol(\cF^{(x)}_{v}))\\
&=A(v)-\frac{1}{\ell_0^{n+1} L^{\cdot n}}\int_0^{+\infty}\vol(\cF^{(x)}_{v})dx=A(v)-S_L(v).
\end{align*}
The first identity follows from \eqref{eq-CMstc}. The second identity uses \eqref{eq-ENAcF}, \eqref{eq-volshift} and a change of variable. The third identity is obtained by using integration by parts. The last identify uses the expression of $S_L(v)$ in \eqref{eq-defSLv}. 
}
\end{proof}

\subsection{An alternative proof of the valuative criterion for $\bG$-uniform Ding stability}\label{sec-alternative}
Here we provide a proof of the valuative criterion for $\bG$-uniform Ding-stability without using the MMP program. In other words, we prove the equivalence of $(2)\Leftrightarrow (3)$ in Theorem \ref{thm-Gvalcriterion}. Since (2) implies (3) by Corollary \ref{cor-Gunival}, we just need to show the other direction. Our argument is motivated by Boucksom-Jonsson's work in \cite{BoJ18b} and will also be used in the proof of existence result in section \ref{sec-step4}. We first claim that it suffices to prove the following inequality: for any \Blue{non-Archimedean potential} $\phi=\phi_{(\mcX,\mcL)}$ coming from $\bG$-equivariant semi-ample test configuration,
\begin{equation}\label{eq-SL2bfE}
\inf_{v\in (\Xdiv)^{\bG}} (S_L(v)+\phi(v))\ge \inf_{v\in X^{\rm div}}(S_L(v)+\phi(v))\ge  \bfE^\NA(\phi).
\end{equation}
Assume that this is true. By the expression of $\bfL^\NA$ in \eqref{eq-LNAphi} (by using \eqref{eq-FeqTC}), we can find $v_k\in (\Xdiv)^\bG$ such that 
\begin{equation}\label{eq-bfLapprox}
\bfL^\NA(\phi)\le A_X(v_k)+\phi(v_k)\le \bfL^\NA(\phi)+\frac{1}{k}.
\end{equation}
Assuming the valuative condition, there exists $\xi_k\in N_\bR$ such that $A(v_{k,-\xi_k})\ge \delta S_L(v_{k,-\xi_k})$. By density, we can assume $\xi_k\in N_\bQ$ so that up to base change, $\phi_{\xi_k}$ is equivalent to a semi-ample test configuration.
So we can apply inequality \eqref{eq-SL2bfE} to $\phi_{\xi_k}$ to get:
\begin{eqnarray}\label{eq-Aphivk}
A(v_k)+\phi(v_k)&=&A(v_{k,-\xi_k})+\phi_{\xi_k}(v_{k,-\xi_k})\ge \delta S_L(v_{k,-\xi_k})+\phi_{\xi_k}(v_{k,-\xi_k})\nonumber \\
&\ge&\delta \bfE^\NA(\delta^{-1}\phi_{\xi_k}).
\end{eqnarray}
The first equality uses the identity \eqref{eq-A+phivxi}.
Combining \eqref{eq-bfLapprox}-\eqref{eq-Aphivk}, we get:
\begin{eqnarray*}
\bfD^\NA(\phi)&=&-\bfE^\NA(\phi)+\bfL^\NA(\phi)\ge-\bfE^\NA(\phi)+ A(v_k)+\phi(v_k)-\frac{1}{k}\\
&\ge& -\bfE^\NA(\phi_{\xi_k})+\delta \bfE^\NA(\delta^{-1}\phi_{\xi_k})-\frac{1}{k}=\delta \bfJ^\NA(\delta^{-1}\phi_{\xi_k})-\bfJ^\NA(\phi_{\xi_k})-\frac{1}{k}\\
&\ge&(1-\delta^{-1/n})\bfJ^\NA(\phi_{\xi_k})-\frac{1}{k}\ge (1-\delta^{-1/n})\bfJ^\NA_\bT(\phi)-\frac{1}{k},
\end{eqnarray*}
where we used the non-Archimedean version of Ding's inequality (\cite[Lemma 6.17]{BoJ18a}).

Coming back to the proof of the inequality \eqref{eq-SL2bfE}, we give a different proof \Blue{from that} in \cite{BoJ18b} (without using the Legendre duality by viewing $S_L(v)$ as $\bfE^*(\delta_v)$). 
To do this, we use the explicit description of the filtration $\cF=\cF_{(\mcX, \mcL)}$ associated to a normal semi-ample test configuration in \eqref{eq-filTCvan} and compare it with the filtration $\cF_v$ induced by any divisorial valuation $v$.  Using similar notation as there, we set
$\mcL=\rho^*L_\bC+D$ with $D=\sum_E a_E E$ where $E$ runs over irreducible components of the central fibre $X_0=\sum_E b_E E$.
By \eqref{eq-phicZcL}, we know that, for any fixed divisorial valuation $v$ over $X$:
\begin{equation}\label{eq-defphiva}
\phi(v)=\phi_{(\mcX, \mcL)}(v)=G(v)(D)=\sum_E a_E G(v)(E)=:a.
\end{equation}
Now for any $s\in \cF^x R_m$, $r(\ord_E)(s)+m\ell_0\cdot \ord_E(D)\ge x b_E$ by \eqref{eq-filTCvan}. This implies that:
\begin{eqnarray*}
v(s)&=&G(v)(\bar{s})=\sum_E G(v)(E) \ord_E(\bar{s})\ge \sum_E G(v)(E)(x b_E-m\ell_0 a_E)\\
&=&x G(v)(t)-m\ell_0\sum_E a_E G(v)(E)=x-m\ell_0 a.
\end{eqnarray*}
So we get $\cF^xR_m\subseteq \cF^{x-m\ell_0 a}_{v}R_m$. As a consequence, $\vol(\cF^{(t)})\le \vol(\cF_v^{(t-\ell_0a)})$. Because $\lambda_{\min}=\inf \{t\in \bR; \vol(\cF^{(t)})<\ell_0^n V\}$  by \cite[Corollary 5.4]{BHJ17} and $\vol(\cF_v^{(t)})<V\ell_0^n$ when $t>0$ (by Izumi's inequality, see \cite[5]{Li17}), we easily get the inequality $\lambda_{\min}\le \ell_0 a$.
We can then calculate as follows to get the wanted inequality:
\begin{eqnarray*}
\bfE^\NA(\phi)&=&-\frac{1}{V\ell_0^n}\int_\bR \frac{x}{\ell_0} d\vol(\cF^{(x)})=\frac{\lambda_{\min}}{\ell_0}+\frac{1}{V\ell_0^{n+1}}\int^{+\infty}_{\lambda_{\min}}\vol(\cF^{(x)})dx\\
&\le& a+\frac{1}{V \ell_0^{n+1}}\int_{\ell_0 a}^{+\infty}\vol(\cF^{(x)})dx\le a+\frac{1}{V\ell_0^{n+1}}\int_{\ell_0 a}^{+\infty}\vol(\cF_v^{(x-\ell_0a)})dx\\
&=&a+\frac{1}{V\ell_0^{n+1}}\int_0^{+\infty}\vol(\cF^{(t)}_v) dt=\phi(v)+S_L(v).
\end{eqnarray*}
The second identity is obtained by integration by parts (which holds even if $d\vol(\cF^{(x)})$ has a Dirac mass at $\lambda_{\max}(\cF)$). The second inequality is because the function $y\mapsto y+\frac{1}{V\ell_0^{n+1}}\int_y^{\infty}\vol(\cF^{(x)})dx$ is an increasing function of $y\in \bR$ (which is constant for $y\le \lambda_{\min}(\cF)$).
The last identity uses \eqref{eq-defphiva} and \eqref{eq-defSLv}.

\section{Proof of Theorem \ref{thm-YTD} and Theorem \ref{thm-AutYTD}}

The necessary part of Theorem \ref{thm-AutYTD} immediately follows from Theorem \ref{thm-torusproper} and Theorem \ref{thm-JTslope}. So the rest of this paper is devoted to proving Theorem \ref{thm-YTD}.  

By Theorem \ref{thm-analytic}, we just need to prove the Mabuchi functional is $\bG$-coercive. The general strategy is of course motivated by \cite{BBJ18} and our previous work \cite{LTW19}. However due to the various complications caused by twists, we need to re-work out the argument more carefully. One main point is that we only work with $\bK$-invariant (in particular $(S^1)^r$-invariant) metrics. The proof \Blue{proceeds by contradiction.} So we assume that the Mabuchi energy is not $\bG$-coercive.

\subsection{Step 1: Construct a destabilizing geodesic ray}\label{sec-step1}

In this step, assuming that the Mabuchi energy $\bfM=\bfM_{(X, D)}$ \Blue{is} not $\bG$-coercive,  we will find a destabilizing geodesic ray $\Phi=(\vphi(s))$ in $\mcE^1(X, L)^\bK$ such that
\begin{enumerate}[(1)]
\item
The Ding energy is decreasing along $\Phi=\{\vphi(s)\}$ for any $\xi\in N_\bR$:
\begin{equation*}
\bfD'^\infty(\Phi)=\lim_{s\rightarrow+\infty}\frac{\bfD(\vphi(s))}{\Blue{s}}\le 0.
\end{equation*} 
\item we have the normalization:
\begin{equation}
\sup(\vphi(s)-\psi)=0, \quad \bfE_{\psi}(\vphi(s))=- s.
\end{equation}
\item \Blue{For any $\xi\in N_\bR$, set $\Phi_\xi=\{\vphi_\xi(s)\}_{s\in [0, +\infty)}$ to be the geodesic ray defined by:
\begin{equation}\label{eq-Phixi}
\vphi_\xi(s):=\sigma_\xi(s)^*\vphi(s).
\end{equation}} 
Then $\Phi_\xi$ satisfies:
\begin{equation}
\bfJ'^\infty(\Phi_\xi)=\lim_{s\rightarrow+\infty}\frac{\bfJ_\psi(\sigma_\xi(s)^*\vphi(s))}{\Blue{s}}>0.
\end{equation}
\end{enumerate}

The argument for constructing such a destabilising geodesic ray is similar to the arguments in \cite{BBJ15, BBJ18}. All energy functionals in this step are on $X$ itself as defined in \eqref{eq-Ephi}-\eqref{eq-Mphi}.
 Assume the Mabuchi energy ${\bf M}={\bf M}_{\psi}$ (see \eqref{eq-Mphi}) is not $\bG$-coercive. Then choosing $\gamma_j\rightarrow 0$, we can pick a sequence $\{u_j\}_{j=1}^\infty\in (\mcE^1)^\bK=(\mcE^1(X, \omega))^\bK$ as in \cite[4.1]{LTW19} (in the $\bK$-invariant setting) such that $\vphi_j={\psi}+u_j$ satisfies:
\begin{equation}\label{eq-DMnotproper}
\bfD(\vphi_j)\le \bfM(\vphi_j)\le \gamma_j \bfJ_\bT(\vphi_j)-j\le \gamma_j \bfJ(\sigma^*\vphi_j)-j
\end{equation}
for any $\sigma\in \bT$. Because of Lemma \ref{lem-infobt}, we can assume that:
\begin{equation}\label{eq-vphijinf}
\Blue{\bfJ(\vphi_j)}=\inf_{\sigma\in \bT}\bfJ(\sigma^*\vphi_j).
\end{equation}
We normalize $\vphi_j$ such that $\sup (\vphi_j-{\psi})=0$. 
%Recall that the $\bG$-uniform K-stability implies that $\Fut\equiv 0$ on the Lie algebra of $\bT$ (see the paragraph before Lemma \ref{lem-uni2poly}). On the other hand, it is known that the Mabuchi functional is linear \Blue{along one parameter subgroups, with slope} given by the Futaki invariant of the generating holomorphic vector field (see \cite{Mab86} for the original smooth case and \cite[4.2]{HL20} for a similar calculation in a general singular setting). So $\bfM(\vphi)=\bfM(\sigma^*\vphi)$ for any $\sigma\in \bT$. 
\Blue{Now the inequality (see \eqref{eq-Mge-nJ}) $\bfM(\vphi_j) \ge C-n \bfJ(\vphi_j)$ together with \eqref{eq-DMnotproper} 
implies the estimate:
\begin{eqnarray*}
\bfJ(\vphi_j)\ge \frac{j+C}{n+\gamma_j}\rightarrow+\infty \quad \text{ as } j\rightarrow+\infty,
\end{eqnarray*} 
and hence $\bfE(\vphi_j)\le \sup(\vphi_j-\psi)-\bfJ(\vphi_j)\le  -\bfJ(\vphi_j)\rightarrow-\infty$.}

Denote $V=(2\pi)^n (-K_X-D)^{\cdot n}$. By the work \cite{Dar17, DNG18}, we can connect ${\psi}$ and $\vphi_j$ by a unit speed geodesic segment $\{\vphi_j(s)\}\in \PSH_\bd(X, L)^{\bK}$ 
parametrized so that $S_j:=-\bfE(\vphi_j)\rightarrow +\infty$ with $s\in [0, S_j]$. In particular, $\bfE(\vphi_j(s))=-s$.
Then $\psi$ and $\vphi_{j,\xi}:=\sigma_\xi(S_j)^*\vphi_j$ is connected by the geodesic segment $\sigma_\xi(s)^*\vphi_j$, $s\in [0, S_j]$.

By \cite[4.1.2]{LTW19} (see also \cite{BB17, BDL17}), $\bfM$ satisfies the following convexity property along our geodesic segments:
\begin{eqnarray}\label{eq-Dvphisub}
\bfD(\vphi_j(s))&\le&\bfM(\vphi_j(s))\le \frac{S_j-s}{S_j}\bfM({\psi})+\frac{s}{S_j}\bfM(\vphi_j)\nonumber\\
&\le& C+\frac{s}{S_j}(\gamma_j \bfJ(\vphi_j)-j)\le C+\frac{s}{S_j}\gamma_j \bfJ(\vphi_j).
\end{eqnarray}

Using $\bfM\ge \bfH-n \bfJ$ \Blue{(see \eqref{eq-Mge-nJ})}, we get $\bfH(\vphi_j(s))\le (\gamma_j+n)s+C$. So for any fixed $S>0$ and $s\le S$, the metrics $e^{-\vphi_j(s)}$ lie in the set:
\[
\mathcal{K}_S:=\{\Blue{\vphi\in \cE^1}; \sup(\vphi-{\psi})= 0 \text{ and } \bfH(\vphi)\le (\gamma_j+n)S+C \}.
\]
This is a compact subset of the metric space $(\mcE^1, d_1)$ by Theorem \ref{thm-BBEGZ} from \cite{BBEGZ}. So, by arguing as in \cite{BBJ15}, after passing to a subsequence, $\{\vphi_j(s)\}$ converges to a geodesic ray
$\Phi:=\{\vphi(s)\}_{s\ge 0}$ in $(\mcE^1)^\bK$, uniformly for each compact time interval. Moreover $\{\vphi(s)\}_{s\in \bR}$ satisfies 
\begin{equation}\label{eq-Dvphidec}
\lim_{s\rightarrow+\infty} \frac{\bfD(\vphi(s))}{\Blue{s}}\le 0, \quad \sup(\vphi(s)-\psi)=0, \quad \bfE(\vphi(s))=-s.
\end{equation}
For any $\xi\in N_\bR$, by \eqref{eq-vphijinf} we have
\begin{equation}
\bfJ(\sigma_\xi(S_j)^* \vphi_{j})\ge \bfJ(\vphi_j)=-\bfE(\vphi_j)+O(1)=S_j+O(1) \rightarrow+\infty.
\end{equation} 
The second identity uses Lemma \ref{lem-Hartogs}.
Moreover $\{\sigma_\xi(s)^*\vphi_j(s)\}_{s\in [0, S_j]}$ converges strongly to the geodesic ray $\Phi_\xi:=\{\sigma_\xi(s)^*\vphi(s)\}_{s\ge 0}$. So we get, for any $\xi\in N_\bR$, 
\begin{equation}
\lim_{s\rightarrow+\infty} \bfJ_\psi(\sigma_\xi(s)^*\vphi(s))=+\infty
\end{equation}
This implies that  $\{\sigma_\xi(s)^*\vphi(s)\}$ is a nontrivial geodesic, because (for $\bfE$-normalized potentials) $\bfJ$-energy is comparable to $d_1$-distance which is linear along geodesics (see \cite[(31)]{Dar17}, \cite[Theorem 3.6]{DNG18}). In particular, for any $\xi\in N_\bR$
\begin{equation}\label{eq-J'infpositive}
\bfJ'^\infty(\Phi_\xi):=\lim_{s\rightarrow+\infty} \frac{\bfJ_\psi(\sigma_\xi(s)^*\vphi(s))}{\Blue{s}}>0.
\end{equation}
\begin{prop}[{see \cite[Proposition 1.6]{His16b}}]\label{prop-convex1}
Let $\Phi=\{\vphi(s)\}_{s\in [0, +\infty)}\subset \mcE^1(L)^{(S^1)^r}$ be a geodesic ray. The function 
Let $(s, \xi)\rightarrow \bfJ(\sigma_\xi(s)^*\vphi(s))$ is convex in $(s, \xi)\in [0, +\infty)\times N_\bR$.
\end{prop}
\begin{proof}
Choose any $\xi_0, \xi'\in N_\bR$. 
Consider the holomorphic map (see \eqref{eq-sigmaxi}):
\begin{equation}
F: X\times \bC\times \bC \rightarrow X\times\bC , \quad (x, \mathfrak{z}=s+iu, \mathfrak{c}=c+id)\mapsto (\sigma_{\xi_0}(\mathfrak{z}) \sigma_{\xi'} (\mathfrak{c}\mathfrak{z})\cdot x, \mathfrak{z}).
\end{equation}
Then $F^*\Phi$ is a finite energy psh Hermitian metric on $p_1^*L$ where $p_1: X\times\bC\times \bC \rightarrow X$ is the projection. For any $c\in \bR$, denote $\xi_c:=\xi_0+c\xi'$.

Note that, because $\exp(J\xi), \exp(J\xi')\in(S^1)^r$ and $\vphi(s)\in \mcE^1(L)^{(S^1)^r}$, we have:
\begin{eqnarray*}
F^*\Phi&=&(\exp(s \xi_0)\exp(uJ\xi_0))^*\exp((sc-ud) \xi')^*\exp((sd+uc)J\xi')^*\vphi(s)\\
&=&\exp(s\xi_0)^*\exp((sc-ud)\xi')^*\vphi(s).
\end{eqnarray*} 
In particular, $F^*\Phi|_{u=0}$ is the twisted geodesic ray $\sigma_{\xi_0+c\xi'}(s)^*\vphi(s)$. Because $F$ is holomorphic we know that $\sddb F^*\Phi\ge 0$. \Blue{
For simplicity of notation, set $\mathfrak{J}(\mathfrak{z}, \mathfrak{c}, s):=\bfJ(\sigma_{\xi_0+\mathfrak{c}\xi'}(\mathfrak{z})^*\vphi(s))$. Then $\mathfrak{J}$ is a continuous function because $(\mathfrak{z}, \mathfrak{c}, s)\mapsto \sigma_{\xi_0+\mathfrak{c}\xi'}(\mathfrak{z})^*\vphi(s)\in \cE^1$ is continuous (see \cite[Theorem 1.7]{BBJ18}) and $\bfJ=\Lam-\bfE$ is continuous with respect to the strong topology (see Definition \ref{defn-strong}). 
To prove the convexity, we can calculate using integration along the fibre formula:
\begin{align*}
\hskip 2mm &(2\pi)^n L^{\cdot n}\cdot \sddb_{(\mathfrak{z}, \mathfrak{c})} \mathfrak{J}(\mathfrak{z}, \mathfrak{c}, s) \\ %\bfJ(\sigma_{\xi_0+\mathfrak{c}\xi'}(\mathfrak{z})^*\vphi(s))\\
=&\sddb \left(\int_{X} (F^*\Phi-\psi) (\sddb \psi)^n-\frac{1}{n+1}\int_X (\Phi-\psi)\sum_{k=0}^n (\sddb\Phi)^k\wedge (\sddb\psi)^{n-k}\right)\\
=&\int_X \sddb(F^*\Phi)\wedge (\sddb\psi)^n-(\sddb\psi)^{n+1}\\
&\hskip 4cm -\frac{1}{n+1}\int_X F^*(\sddb\Phi)^{n+1}-(\sddb\psi)^{n+1}\\
=&\int_X F^*(\sddb\Phi)\wedge (\sddb\psi)^n\ge 0
\end{align*}
where we identified $\psi$ with $p_1^*\psi$ and used the vanishing $(\sddb\Phi)^{n+1}=0$ and $(\sddb\psi)^{n+1}=0$ on $X\times\bC\times\bC$.  
As a consequence we easily get $\mathfrak{J}(s,c):=\bfJ(\sigma_{\xi_0+c\xi'}(s)^*\vphi(s))$ is convex. See also \cite[Theorem 3.1]{ACKPZ} for a similar result in a local setting.}

\end{proof}

\begin{prop}
The function $\xi\mapsto \bfJ'^\infty(\Phi_\xi)$ is convex in $\xi\in N_\bR$.
\end{prop}
\begin{proof}
Using the notations in the proof of the Proposition \ref{prop-convex1}, we consider the convex function $f(s, c):=\bfJ(\sigma_{\xi_0+c\xi'}(s)^*\vphi(s))$. Then for any $0<c_1<c_2$, by convexity we have
\begin{equation}
f(s, c_1)\le (1-\frac{c_1}{c_2}) f(s, 0)+\frac{c_1}{c_2} f(s, c_2).
\end{equation}
Dividing both sides by $s$ and letting $s\rightarrow+\infty$, we get the wanted convexity:
\begin{equation}
\bfJ'^\infty(\Phi_{\xi_0+c_1\xi'})\le (1-\frac{c_1}{c_2})\bfJ'^\infty(\Phi_{\xi_0})+\frac{c_1}{c_2}\bfJ'^\infty(\Phi_{\xi+c_2\xi'}).
\end{equation}
\end{proof}
Because a convex function on $N_\bR\cong \bR^r$ is continuous, it obtains a minimum on compact set. Combing this with \eqref{eq-J'infpositive} we get:
\begin{cor}\label{cor-minchi}
For any $C>0$ there exists $\chi=\chi(C,\Phi)>0$ such that
for any $\xi$ satisfying $|\xi|<C$, $\bfJ'^\infty(\Phi_\xi)\ge \chi>0$.
\end{cor}
\begin{rem}
Recently, the above corollary has been strengthened in \cite[Proof of Proposition 6.2]{Li20}. For the destabilising geodesic ray obtained above, we indeed have the inequality:
\begin{equation}
\inf_{\xi\in N_\bR}\bfJ'^\infty(\Phi_\xi)\ge 1.
\end{equation}
\end{rem}

\subsection{Step 2: Perturbed and twisted test configurations}\label{sec-pertTC}

Fix a $\bG$-equivariant resolution of singularities $\mu: Y\rightarrow X$ such that $\mu$ is an isomorphism over $X^\reg$, $\mu^{-1}(X^\sing)=\sum_{k=1}^{g}E_k$ is a $\bG$-invariant simple normal crossing divisor and that there exist $\theta_k\in \bQ_{>0}$ for $k=1,\dots, g$ such that $E_\theta=\sum_{k=1}^g \theta_k E_k$ satisfies $P:=P_\theta=\mu^*L-E_\theta$ is an ample $\bQ$-divisor over $Y$. We can then choose and fix a smooth $\bK$-invariant Hermitian metric $\vphi_P$ on $P$ such that $\sddb \vphi_P>0$. 

For any $\epsilon\in \bQ_{>0}$, define \Blue{$\bQ$-line} bundles on $Y$ by 
\begin{equation}\label{eq-Lepsilon}
\hat{L}_\epsilon:=(1+\epsilon)\mu^*L-\epsilon E_\theta=\mu^*L+\epsilon P, \quad L_\epsilon=\frac{1}{1+\epsilon}\hat{L}_\epsilon.
\end{equation}
Then $\hat{L}_\epsilon$ is a positive $\bQ$-line bundle on $Y$. Define a smooth reference potential on $\hat{L}_\epsilon$ by $\hat{\psi}_\epsilon=\psi+\epsilon \vphi_P\in (\mcE^1(X, L_\epsilon))^\bK$. Let $\Phi=\{\vphi(s)\}$ be \Blue{the} geodesic ray in $(\mcE^1(X, L))^\bK$ constructed in the above subsection, which satisfies:
\begin{equation}\label{eq-geocon}
\sup_X(\vphi(s)-\psi)=0, \quad \bfE_{\psi}(\vphi(s))=-s.
\end{equation}.

In this section we will first construct a sequence of test configurations \Blue{for} $(Y, \hat{L}_\epsilon)$ using the method from \cite{BBJ15}. Denote by $p'_i, i=1,2$ the projection of  $Y\times\bC$ to the two factors. Define a singular and a smooth $\bK$-invariant Hermitian metric on $p'^*_1 \hat{L}_\epsilon$ by 
\begin{equation}\label{eq-hatPhiep}
\hat{\Phi}_\epsilon:=\mu^*(\Phi)+\epsilon\; p'^*_1(\vphi_{P}), \quad \hat{\Psi}_\epsilon:=p'^*_1(\mu^*\psi+\epsilon\; \vphi_P).
\end{equation}
Here for simplicity of notation, we still use $\mu$ to denote the morphism $\mu\times \mathrm{id}: Y\times\bC\rightarrow X\times\bC$.  
Then $\sddb\hat{\Phi}_\epsilon\ge 0$, $\sddb\hat{\Psi}_\epsilon\ge 0$. Fix a very ample line bundle $H'$ over $Y$. Consider the following coherent sheaf:
\begin{eqnarray*}
\mcF_{\epsilon,m}&:=&\mcO_{Y}(p'^*_1(m\hat{L}_\epsilon))\otimes_{\mcO_Y} \mcJ(Y, m \hat{\Phi}_\epsilon)\\
&=&\mcO_Y\left(K_{Y}+m \mu^*L+(m\epsilon P-K_{Y}-(n+1)H')+(n+1)H'\right)\otimes_{\mcO_Y} \mcJ(Y, m \mu^*\Phi)
\end{eqnarray*}
\Blue{where $\mcJ(Y, m\hat{\Phi}_\epsilon)$ (resp. $\mcJ(Y, m\mu^*\Phi)$) denotes the multiplier ideal sheaf of the psh potential $m\hat{\Phi}_\epsilon$ (resp. $m\mu^*\Phi$). For the second identity, we substituted the expression of $\hat{L}_\epsilon$ in \eqref{eq-Lepsilon} and used the 
fact that $\mcJ(Y, m\hat{\Phi})=\mcJ(Y, m\mu^*\Phi)$ by the smoothness of the psh potential $\vphi_P$. }
Because $P$ is positive, for $m\gg \epsilon^{-1}$ and sufficiently divisible, 
$m \epsilon P-K_{Y}-(n+1)H' $ is an ample line bundle on $Y$.
In this case, by Nadel vanishing theorem, for any $j\ge 1$, 
\begin{eqnarray*}
R^j(p'_2)_*(\mcF_{\epsilon,m} \otimes p_1^*H'^{-j})=0. 
\end{eqnarray*}
By the relative Castelnuovo-Mumford criterion, $\mcF_{\epsilon,m}$ is $p'_2$-globally generated.
Because $\bD$ is Stein, $\mcO(p'^*_1(m\hat{L}_\epsilon)\otimes \mcJ(m\hat{\Phi}_\epsilon))$ is generated by global sections on $Y\times \bD$ if $m\gg \epsilon^{-1}$ and $m$ is sufficiently divisible.

Let $\pi'_m: \mcY_{\epsilon,m}\rightarrow Y_\bC$ denote the normalized blow-up of $Y\times \bC$ along $\mcJ(m \hat{\Phi}_\epsilon)$, with exceptional divisor $E_{\epsilon,m}$ and set 
\begin{equation}
\hat{\mcL}_{\epsilon,m}:=\pi'^*_m p_1'^*\hat{L}_\epsilon-\frac{1}{m}E_{\epsilon,m}, \quad \mcL_{\epsilon,m}=\frac{1}{1+\epsilon}\hat{\mcL}_{\epsilon,m}.
\end{equation}
Then $(\mcY_{\epsilon,m}, \hat{\mcL}_{\epsilon,m})$ is a $\bG$-equivariant normal semi-ample test configuration for $(Y, \hat{L}_\epsilon)$. To see $\bG$-equivariance, note that by the $\bK$-invariance of $\hat{\Phi}_\epsilon$, $\mcJ(Y, m\hat{\Phi})$ is invariant under the action of $\bK$ on $\mcO_Y$. Because $\bG=\bK^\bC$, the invariance under $\bG$ action follows from the biholomorphicity of the $\bG$-action.

The associated \Blue{non-Archimedean potential} $\hat{\phi}_{\epsilon,m}\in \mcH^\NA(\hat{L}_\epsilon)$ is given by:
\begin{equation}
\hat{\phi}_{\epsilon,m}(w)=-\frac{1}{m}G(w)(\Blue{\mcJ(m\mu^*\Phi)}),
\end{equation}
for each $w\in \Blue{Y^{\mathrm{div}}_\bQ}$. Note again that since $\vphi_P$ is a smooth Hermitian metric, 
\begin{equation}
\mcJ(m\hat{\Phi}_\epsilon)=\mcJ(m\mu^*\Phi)=:\Blue{\mcJ(m\tilde{\Phi})}.
\end{equation}

We will denote by $\hat{\Phi}_{\epsilon, m}=\{\hat{\vphi}_{\epsilon,m}(s)\}$ the geodesic ray associated to $(\mcY_{\epsilon,m}, \hat{\mcL}_{\epsilon, m})$. 
By Demailly's regularization result (\cite[Proposition 3.1]{Dem92}), $\hat{\Phi}_{\epsilon,m}$ is less singular then $\hat{\Phi}_{\epsilon}$. As a consequence, $\hat{\Phi}_{\epsilon,m,\xi}:=\{\sigma_\xi(s)^*\vphi_{\epsilon, m}(s)\}_{s\in [0, +\infty)}$ is less singular than $\hat{\Phi}_{\epsilon,\xi}=\{\sigma_\xi(s)^*\vphi_\epsilon(s)\}_{s\in [0,+\infty)}$. By the monotonicity of $\bfE$ and ${\bf \Lambda}$ energy (see \eqref{eq-monotone}), we get:
\begin{eqnarray}
\bfE^\NA_{\hat{L}_\epsilon}(\hat{\phi}_{\epsilon,m,\xi})&=&\lim_{s\rightarrow+\infty} \frac{\bfE_{\hat{\psi}_\epsilon}(\sigma_\xi(s)^*\hat{\vphi}_{\epsilon,m}(s))}{\Blue{s}}\nonumber \\
&&\ge \lim_{s\rightarrow+\infty} \frac{\bfE_{\hat{\psi}_\epsilon}(\sigma_\xi(s)^*\hat{\vphi}_{\epsilon}(s))}{\Blue{s}}=:\bfE'^\infty_{\hat{\psi}_\epsilon}(\hat{\Phi}_{\epsilon,\xi}).\label{eq-ENAepmlb}
\\
{\bf \Lambda}^\NA_{\hat{L}_\epsilon}(\hat{\phi}_{\epsilon,m,\xi})&=&\lim_{s\rightarrow+\infty} \frac{{\bf \Lambda}_{\hat{\psi}_\epsilon}(\sigma_\xi(s)^*\hat{\vphi}_{\epsilon,m})}{\Blue{s}}\nonumber\\
&&\ge \lim_{s\rightarrow+\infty} \frac{{\bf \Lambda}_{\hat{\psi}_\epsilon}(\sigma_\xi(s)^*\hat{\vphi}_{\epsilon}(s))}{\Blue{s}}
=:{\bf \Lambda}'^\infty_{\hat{\psi}_\epsilon}(\hat{\Phi}_{\epsilon,\xi}). \label{eq-KNAepmlb}
\end{eqnarray}
%Note that the first identity holds because $\hat{\Phi}_{\epsilon, m}$ is a subgeodesic ray with algebraic singularities (see \cite[Lemma 5.3]{BBJ18}). 

The following convergence will be important for us. 
\begin{lem}
With the above notations and assuming that $\Phi=\{\vphi(s)\}_{s\in [0, +\infty)}$ satisfies \eqref{eq-geocon}, for any $\xi\in N_\bR$ the following identities hold true:
\begin{equation}\label{eq-limENAep}
\lim_{\epsilon \rightarrow 0} \bfE'^\infty_{\hat{\psi}_\epsilon}(\hat{\Phi}_{\epsilon,\xi})=\lim_{s\rightarrow+\infty} \frac{\bfE_\psi(\vphi_\xi(s))}{\Blue{s}}=:\bfE'^\infty(\Phi_\xi).
\end{equation}
\begin{equation}\label{eq-limKNAep}
\lim_{\epsilon\rightarrow 0} {\bf \Lambda}'^\infty_{\hat{\psi}_\epsilon}(\hat{\Phi}_{\epsilon,\xi})=\lim_{s\rightarrow+\infty} \frac{{\bf \Lambda}_\psi(\vphi_\xi(s))}{\Blue{s}}=:{\bf \Lambda}'^\infty(\Phi_\xi).
\end{equation}
\end{lem}
\begin{proof}
\Blue{
Set $\hat{\psi}_\epsilon=\mu^*\psi+\epsilon \vphi_P$ to be the smooth psh potential on the $\bQ$-line bundle $\hat{L}_\epsilon=\mu^*L+\epsilon P$ on $Y$. We also identify the smooth psh potential $\psi$ on $L$ with its pull back on $\mu^*L$. 
Because $\bfE$ satisfies \Blue{the} cocycle condition and is affine along geodesics, it is easy to verify that, for any $\Blue{\vphi\in \cE^1}(\hat{L}_\epsilon)$, 
\begin{eqnarray*}
\bfE_{\hat{\psi}_\epsilon}(\sigma_\xi(s)^*\hat{\vphi}_\epsilon)&=&\bfE_{\sigma_\xi(s)^*\hat{\psi}_\epsilon}(\sigma_\xi(s)^*\vphi)+E_{\hat{\psi}_\epsilon}(\sigma_\xi(s)^*\hat{\psi}_\epsilon)\\
&=&\bfE_{\hat{\psi}_\epsilon}(\hat{\vphi}_\epsilon)+\chw_{\hat{L}_\epsilon}(\xi)\cdot s,
\end{eqnarray*}
where $\chw_{\hat{L}_\epsilon}=\chw_L+O(\epsilon)$ is the Chow weight of $\xi$ (see \eqref{eq-defCW}). It was proved in \Blue{\cite[Proposition 4.11]{LTW19}} that:
\begin{equation}
\lim_{\epsilon\rightarrow 0}\Blue{\bfE'^\infty_{\hat{\psi}_\epsilon}({\hat{\Phi}}_\epsilon)=\bfE'^\infty_{\psi}(\Phi)}.
\end{equation} 
These combine to give \eqref{eq-limENAep}. Next we prove \eqref{eq-limKNAep}. By the definition of ${\bf \Lambda}$-energy (see \eqref{eq-Kphi})
\begin{eqnarray*}
(2\pi)^n \hat{L}_\epsilon^{\cdot n}\cdot {\bf \Lambda}_{\hat{\psi}_\epsilon}(\hat{\vphi}_{\epsilon,\xi}(s))&=&\int_X (\sigma_\xi(s)^*\vphi(s)+\epsilon\sigma_\xi(s)^*\vphi_P-(\psi+\epsilon \vphi_P)) (\sddb (\psi+\epsilon \vphi_P))^n\\
&=&\int_X (\sigma_\xi(s)^*\vphi(s)-\psi)(\sddb\psi)^n\\
&&\hskip 0.5cm+\int_X(\sigma_\xi(s)^*\vphi(s)-\psi)[(\sddb(\psi+\epsilon \vphi_P))^n-(\sddb\psi)^n]\\
&&\hskip 1cm+\epsilon\int_X (\sigma_\xi(s)^*\vphi_P-\vphi_P)(\sddb(\psi+\epsilon \vphi_P))^n\\
&=&(2\pi)^n L^{\cdot n} \cdot {\bf \Lambda}_\psi(\vphi_\xi(s))+{\bf I}_{\epsilon}(s)+{\bf II}_\epsilon(s)
\end{eqnarray*}
where we denoted:
\begin{eqnarray*}
{\bf I}_\epsilon&=&\epsilon\int_X(\sigma_\xi(s)^*\vphi(s)-\psi)\frac{(\sddb(\psi+\epsilon \vphi_P))^n-(\sddb \psi)^n}{\epsilon}\\
&=&\epsilon \int_X(\sigma_\xi(s)^*(\vphi(s)-\psi)+\sigma_\xi(s)^*\psi-\psi) \Omega_\epsilon\\
&=&\epsilon ({\bf A}_\epsilon(s)+{\bf B}_\epsilon(s))
\end{eqnarray*}
with
\begin{align*}
&\Omega_\epsilon:=\frac{1}{\epsilon}\left((\sddb(\psi+\epsilon \vphi_P))^n-(\sddb \psi)^n\right)\ge 0\\
&\mathbf{A}_\epsilon(s)=\int_X (\sigma_\xi(s)^*(\vphi(s)-\psi)\Omega_\epsilon, \quad \mathbf{B}_\epsilon(s)=\int_X (\sigma_\xi(s)^*\psi-\psi)\Omega_\epsilon;
\end{align*} 
and also:
%\begin{eqnarray*}
%{\bf I}_\epsilon&=&\epsilon\int_X(\sigma_\xi(s)^*\vphi(s)-\psi)\Omega_\epsilon\\
%&=&\epsilon \int_X(\sigma_\xi(s)^*(\vphi(s)-\psi)+\sigma_\xi(s)^*\psi-\psi) \Omega=\epsilon ({\bf A}_\epsilon(s)+{\bf B}_\epsilon(s)).
%\end{eqnarray*}
\begin{align*}
\mathbf{II}_\epsilon(s)=\epsilon \int_X (\sigma_\xi(s)^*\vphi_P-\vphi_P)(\sddb (\psi+\epsilon \vphi_P)^n=\epsilon\cdot \mathbf{C}_\epsilon
\end{align*}
with
\begin{align*}
\mathbf{C}_\epsilon(s)=\int_X(\sigma_\xi(s)^*\vphi_P-\vphi_P)(\sddb(\psi+\epsilon\vphi_P)^n.
\end{align*}
%Write ${\bf II}_\epsilon=\epsilon {\bf C}_\epsilon$. 
Then we get the identity:
\begin{align}\label{eq-Lamep}
(2\pi)^n L_\epsilon^{\cdot n}\cdot {\bf \Lambda}'^\infty_{\hat{\psi}_\epsilon}(\Phi_{\epsilon,\xi})&=(2\pi)^nL^{\cdot n}\cdot {\bf \Lambda}'^\infty_\psi(\Phi_\xi)\nonumber \\
&+ \lim_{s\rightarrow+\infty}\frac{\epsilon {\bf A}_\epsilon(s)}{\Blue{s}}+\epsilon \lim_{s\rightarrow+\infty}\frac{{\bf B}_\epsilon(s)}{\Blue{s}}+\epsilon \lim_{s\rightarrow+\infty}\frac{{\bf C}_\epsilon(s)}{\Blue{s}}.
\end{align}
Note that all of ${\bf \Lambda}_{\psi}(\vphi_\xi(s))$, ${\bf A}_\epsilon$, ${\bf B}_\epsilon$ and ${\bf C}_\epsilon$ are convex in $s\in [0, +\infty)$.
We deal with the limits on the right-hand-side.
\begin{enumerate}
\item
Since $\sup(\vphi(s)-\psi)=0$, $\epsilon \mathbf{A}_\epsilon\le 0$. So $\epsilon \mathbf{A}'^\infty_\epsilon=\epsilon\cdot \lim_{s\rightarrow+\infty}\frac{\mathbf{A}_\epsilon(s)}{s}\le 0$. On the other hand, for any $s\in [0, +\infty)$, $\lim_{\epsilon\rightarrow 0}\epsilon \mathbf{A}_\epsilon(s)=0$. This together with the convexity of $s\mapsto \epsilon \mathbf{A}(s)$ gives, for any $s\in [0, +\infty)$,
\begin{equation}
\lim_{\epsilon\rightarrow0}\epsilon \mathbf{A}'^\infty_\epsilon\ge \lim_{\epsilon\rightarrow0}\epsilon \frac{\mathbf{A}_\epsilon(s)}{s}=0.
\end{equation}
So we get $\lim_{\epsilon\rightarrow 0}\epsilon \mathbf{A}'^\infty_\epsilon=0$. 
%Because $\epsilon {\bf A}_\epsilon$ is convex, $\epsilon {\bf A}_\epsilon \le 0$ and $\lim_{\epsilon\rightarrow 0}\epsilon {\bf A}_\epsilon=0$, it is easy to verify that (see \cite[Proof of Lemma 4.2]{LTW19}) 
%\begin{equation}
%\lim_{\epsilon\rightarrow 0}\lim_{s\rightarrow+\infty} \frac{\epsilon {\bf A}_\epsilon(s)}{\Blue{s}}=0.
%\end{equation}
\item
It is straightforward to verify that there exists $C>0$ independent of $\epsilon$ such that $$\max\left\{\left|\frac{d}{ds}\sigma_\xi(s)^*\psi\right|, \left|\frac{d}{ds}\sigma_\xi(s)^*\vphi_P\right|\right\}\le C$$ which implies $|\sigma_\xi(s)^*\psi-\psi|\le Cs$, $|\sigma_\xi(s)^*\vphi_P-\vphi_P|\le Cs$. 
From the above expressions defining $\mathbf{B}_\epsilon$, $\mathbf{C}_\epsilon$ and the continuous dependence of the smooth volume $\Omega_\epsilon$ on $\epsilon$, we know that there exists $C>0$ independent of $\epsilon$ such that
$|\mathbf{B}_\epsilon(s)|\le Cs, |\mathbf{C}_\epsilon(s)|\le Cs$ which implies:  
$
\left|{\bf B}'^\infty_\epsilon\right|\le C, \left|{\bf C}'^\infty_\epsilon \right|\le C.
$
\end{enumerate}
Combining the above estimates with the convergence $\lim_{\epsilon\rightarrow 0}L_\epsilon^{\cdot n}=L^{\cdot n}$, we use \eqref{eq-Lamep} to get:
\begin{equation}
\lim_{\epsilon\rightarrow+\infty}{\bf \Lambda}'^\infty_{\hat{\psi}_\epsilon}(\Phi_{\epsilon,\xi})={\bf \Lambda}'^\infty_\psi(\Phi).
\end{equation}
}
\end{proof}

\subsection{Step 3: Uniform convergence of $\bL^\NA$ functions}
\Blue{Let $\mu: Y\rightarrow (X, D)$ be the same $\bG$-equivariant log resolution of singularities as explained at the beginning of section \ref{sec-pertTC}. We have the following identity:
\begin{eqnarray*}
K_Y=\mu^*(K_X+D)+\sum_{k=1}^g a_k E_k %=\mu^*K_X-\sum_{i=1}^{g_1} b_i E''_i+\sum_{j=g_1+1}^g a_j E'_j,
\end{eqnarray*}
where $E_k$ are exceptional divisors and $a_k>-1$ for any $k=1,\dots, g$. 
Recall that $E_\theta:=\sum_k \theta_k E_k$ is chosen such that $P:=\mu^*(-K_X-D)-E_\theta$ is ample over $Y$. Then it is easy to check that we have a decomposition (see \eqref{eq-Lepsilon}):
\begin{align}\label{eq-KBL}
-K_Y-B_\epsilon=\frac{1}{1+\epsilon}(\mu^*(-K_X-D)+\epsilon P)=L_\epsilon.
\end{align}
where, for simplicity of notation, we set:
\begin{equation}\label{eq-Bepsilon}
B_\epsilon=\sum_k (-a_k)E_k+\frac{\epsilon}{1+\epsilon}E_\theta.
\end{equation}
In general, $B_\epsilon$ is not effective and can be decomposed as:
\begin{equation*}
B_\epsilon=\Delta_\epsilon-F=\Delta_0+\frac{\epsilon}{1+\epsilon}E_\theta-F
\end{equation*}
where $\Delta_\epsilon=\Delta_0+\frac{\epsilon}{1+\epsilon}E_\theta$, $\Delta_0$ and $F$ are now effective divisors with
\begin{align}\label{eq-Delta0F}
\Delta_0=\sum_{-1<a_k< 0}(-a_k)E_k+\sum_{a_k\ge 0}(\lceil a_k\rceil-a_k)E_k, \quad F=\sum_{a_k\ge 0} \lceil a_k\rceil E_k.
\end{align}}
%\begin{equation}\label{eq-KYG}
%-K_Y+F=\frac{1}{1+\epsilon}( \mu^*(-K_X-D)+\epsilon P)+\Delta_\epsilon=\frac{1}{1+\epsilon}\hat{L}_\epsilon+\Delta_\epsilon=L_\epsilon+\Delta_\epsilon.
%\end{equation}
%From now on, we denote:
%\begin{equation}
%B_\epsilon:=\Delta_\epsilon-F=\sum_k \left(-a_k+\frac{\epsilon}{1+\epsilon}\theta_k\right)E_k.
%\end{equation}
%Then we have the identity $-(K_Y+B_\epsilon)=L_\epsilon$. 
Note that the test configuration $(\mcY_{\epsilon,m}, \mcL_{\epsilon,m})$ constructed in the above section induces a test configuration $(\mcY_{\epsilon,m}, \mcB_{\epsilon,m}, \mcL_{\epsilon,m})$ of the pair $(Y, B_\epsilon)$.

We consider the Ding energy \eqref{eq-DB} to the pair $(Y, B_\epsilon)$. Set 
\begin{equation}
\vphi_\epsilon=\frac{\hat{\vphi}_\epsilon}{1+\epsilon}=\frac{\vphi+\epsilon\vphi_P}{1+\epsilon}\in (\mcE^1(L_\epsilon))^\bK
\end{equation}
and 
\begin{eqnarray*}
\bfD_{\psi_\epsilon}(\vphi_\epsilon)&=&-\bfE_{\psi_\epsilon}(\vphi_\epsilon)+\bL_{(Y, B_\epsilon)}({\vphi}_\epsilon)
\end{eqnarray*}
where 
\Blue{\begin{equation}\label{eq-psiep}
\psi_\epsilon=\frac{\hat{\psi}_\epsilon}{1+\epsilon}=\frac{\psi+\epsilon \vphi_P}{1+\epsilon}
\end{equation}}
and (with $B=B_\epsilon=\Delta_\epsilon-F$ in \eqref{eq-DB}), 
\begin{equation}
\bL_{(Y, B_\epsilon)}(\vphi_\epsilon)=- \log\left(\frac{1}{(2\pi)^nL_\epsilon^{\cdot n}} \int_Y e^{-\vphi_\epsilon} \frac{|s_F|^2}{|s_{\Delta_\epsilon}|^2}\right)=:\bL_\epsilon(\vphi_\epsilon).
\end{equation}

\Blue{For simplicity of notation in the following discussion, for any divisor $D$ on $Y$, we will denote by $D_\bC$ the divisor $D\times\bC$ on $Y\times\bC$.}
The following two results were proved in \cite[4.3]{LTW19}. The first one is based on \cite{Bern15, BBEGZ} and the second one based on \cite{BBJ18, BFJ08}.
\begin{prop}
\begin{enumerate}[(1)]
\item

With the above notations, let $\epsilon$ be sufficiently small such that $\lfloor \Delta_\epsilon\rfloor=0$.
Assume that $\Phi_\epsilon=\{\vphi_\epsilon(s)\}$ is a psh ray in $\mcE^1(Y, L_\epsilon)$. Then 
$\bL_{(Y, B_\epsilon)}(\vphi_\epsilon(s))$ is convex in $s=\log|t|^{-1}$.

\item

Fix $0\le \epsilon\ll 1$.
Let ${\Phi}_\epsilon=\{{\vphi}_\epsilon(s)\}$ be a psh ray in $\mcE^1(Y, {L}_\epsilon)$ normalized such that $\sup({\vphi}_\epsilon(s)-{\psi}_\epsilon)=0$. We consider $\Phi_\epsilon$ as an $S^1$-invariant psh metric on $p'^*_1{L}_\epsilon \rightarrow Y_\bC$. Then we have the identity:
\begin{equation}\label{eq-LBexpan}
\lim_{s\rightarrow+\infty} \frac{\bL_{(Y, B_\epsilon)} ({\vphi}_\epsilon(s))}{\Blue{s}}=
\inf_{w\in \mathfrak{W}} \left(A_{Y_\bC}(w)-w(\Phi_\epsilon)-w((\Delta_\epsilon)_\bC)+ w(F_\bC)\right)-1,
%\inf_{w} \left(A_{X_\bC}(w)-1-\frac{1}{1+\epsilon}w(\Phi)-\frac{\epsilon}{1+\epsilon} w((E_\theta)_\bC)\right)
\end{equation}
where $\mathfrak{W}$ is the set of $\bC^*$-invariant divisorial valuations $w$ on $Y_\bC=Y\times \bC$ with $w(t)=1$.
\end{enumerate}

\end{prop}
Now let $\hat{\Phi}_\epsilon$ be the same as in \eqref{eq-hatPhiep} and set $\Phi_\epsilon=\frac{1}{1+\epsilon}\hat{\Phi}_\epsilon$. 
To state the next result, we define functions on the set of valuations on $Y_\bC$:
\begin{eqnarray}\label{eq-defhepm}
h_{\epsilon, m}(w)&:=&A_{Y_\bC}(w)-\frac{1}{1+\epsilon}w(\hat{\Phi}_{\epsilon,m})-w((B_\epsilon)_\bC) \nonumber \\
&=&A_{Y_\bC}(w)-\frac{1}{1+\epsilon}\frac{1}{m}w(\Blue{\mcJ(m\tilde{\Phi})})-w((\Delta_\epsilon)_\bC)+w(F_\bC)\label{eq-hepsm}\\
h_{\epsilon}(w)&:=&A_{Y_\bC}(w)-\frac{1}{1+\epsilon}\Blue{w(\tilde{\Phi})}-w((B_\epsilon)_\bC)\nonumber\\
&=&A_{Y_\bC}(w)-w((\Delta_0)_\bC)+w(F_\bC)-\frac{1}{1+\epsilon}\Blue{w(\tilde{\Phi})}-\frac{\epsilon}{1+\epsilon} w((E_\theta)_\bC) \label{eq-heps}
%&\le&A_{Y_\bC}(w)-w((\Delta_0)_\bC)+w(F_\bC).\nonumber
\end{eqnarray}
\Blue{where $\tilde{\Phi}=\mu^*\Phi$. }
Then by \eqref{eq-LBexpan} we have the identity:
\begin{equation}\label{eq-Lviah}
%{\bf I}_{\epsilon,m}:=
\Blue{\bL'^\infty(\Phi_{\epsilon,m})=\inf_{w\in \mathfrak{W}}h_{\epsilon,m}(w)-1=:\bL^\NA(\phi_{\epsilon,m}), \quad %{\bf I}_\epsilon:=
\bL'^\infty(\Phi_\epsilon)=\inf_{w\in \mathfrak{W}}h_\epsilon(w)-1.}
\end{equation}

\begin{prop}\label{prop-uniformLest}
There exists $K>0$ such that if we set
\begin{equation}\label{eq-fWK}
\mathfrak{W}_K:=\{w\in \mathfrak{W}; A_{Y_\bC}(w)<K\},
\end{equation} 
then the following statements are true:
\begin{enumerate}[(1)]
\item The following identities hold true:
\begin{eqnarray}
%\bL'^\infty(\Phi)&=&\inf_{w\in W'}h_0(w)\\
\bL'^\infty(\Phi_\epsilon)=\inf_{w\in \mathfrak{W}_K}h_\epsilon(w)-1, \quad \bL^\NA(\phi_{\epsilon,m})=\inf_{w\in \mathfrak{W}_K}h_{\epsilon, m}(w)-1.
\end{eqnarray}

\item
There exists a constant $C'>0$ independent of $\epsilon$ and $m$ such that for any $\epsilon\ge 0$, $m\in \bN$ and $w\in \mathfrak{W}_K$, we have:
\begin{equation}\label{eq-unidiffh}
|h_{\epsilon, m}(w)-h_\epsilon(w)|\le C' \frac{1}{m}, \quad |h_{\epsilon}(w)-h_0(w)|\le C'\epsilon.
\end{equation}

\item
\Blue{For any $k\in \bN$, there exist $m_0=m_0(k)$ and $\epsilon_0=\epsilon_0(k)$ such that 
\begin{equation}\label{eq-Lslopeuniconv}
|\bfL^\NA(\phi_{\epsilon,m})-\bfL'^\infty(\phi_\epsilon)|\le k^{-1}
\end{equation}
for any $0<\epsilon\le \epsilon_0$ and any $m\ge m_0$. 
In particular we have the convergence:
\begin{eqnarray}\label{eq-Lepmconv}
\lim_{m\rightarrow+\infty} \bL^\NA(\phi_{\epsilon,m})&=&\lim_{s\rightarrow+\infty} \frac{\bL_{(Y, B_\epsilon)}(\vphi_\epsilon(s))}{\Blue{s}}=:\bL'^\infty(\Phi_\epsilon).
\label{eq-limLNAm}
\end{eqnarray}}
Moreover we have the convergence:
\begin{eqnarray}\label{eq-Lepconv}
\lim_{\epsilon\rightarrow 0}\bL'^\infty(\Phi_\epsilon)&=&\bL'^\infty(\Phi). \label{eq-limLNAep}
\end{eqnarray}

\end{enumerate}
\end{prop}
\begin{proof}
\Blue{First, by the valuative description of multiplier ideal sheaves, 
we have the following inequalities proved in \cite[Lemma B.4]{BBJ18}. 
%By the definition of multiplier ideals, 
For any $w\in \Blue{(Y\times\bC)^{\mathrm{div}}_\bQ}$, 
\begin{equation}\label{eq-wJPhim}
 w(\Blue{\mcJ(m\tilde{\Phi})})\le  m\; w(\tilde{\Phi})\le w(\Blue{\mcJ(m\tilde{\Phi})})+ A_{Y_\bC}(w)
\end{equation}
where $\tilde{\Phi}=\mu^*\Phi$. }
So we get the following inequality for functions defined in \eqref{eq-hepsm} and \eqref{eq-heps}:
\begin{eqnarray*}
h_\epsilon(w)&\le& h_{\epsilon, m}(w)\le h_\epsilon(w)+\frac{1}{m}A_{Y_\bC}(w)\le 2 A_{\bC}(w)-w((\Delta_0)_\bC)+w(F_\bC).
\end{eqnarray*}

So there exists $C_1>0$ such that
\begin{equation}\label{eq-hepless}
\inf_{w\in \mathfrak{W}}h_\epsilon(w)\le \inf_{w\in \Blue{\mathfrak{W}}}h_{\epsilon, m}(w)\le C_1.
\end{equation}
Let $W_{\epsilon,m}:=\{w\in \mathfrak{W}; h_{\epsilon, m}\le C_1+1\}$. Then 
\Blue{\begin{equation}\label{eq-bIepm}
\inf_{w\in \mathfrak{W}}h_{\epsilon}(w)=\inf_{w\in W_{\epsilon,m}}h_\epsilon(w)=:\textbf{I}_\epsilon, \quad \inf_{w\in \mathfrak{W}}h_{\epsilon,m}(w)=\inf_{w\in W_{\epsilon,m}}h_{\epsilon, m}(w)=:\textbf{I}_{\epsilon,m}.
\end{equation}}
For any $w\in W_{\epsilon, m}$, we have:
\begin{eqnarray*}
A_{Y_\bC}(w)&\le&C_1+1+w((\Delta_0)_\bC)-w(F_\bC)+\frac{1}{1+\epsilon}\frac{1}{m}w(\Blue{\mcJ(m\tilde{\Phi})})+\frac{\epsilon}{1+\epsilon}w((E_\theta)_\bC)\\
&\le& C_1+1+w((\Delta_0)_\bC)-w(F_\bC)+\frac{1}{1+\epsilon}\Blue{w(\tilde{\Phi})}+\frac{\epsilon}{1+\epsilon}w((E_\theta)_\bC)\\
&\le& C_1+1+w((\Delta_0)_\bC)+\Blue{w(\tilde{\Phi})}+w((E_\theta)_\bC)\\
&\le&C_1+1+C_2+(1-\tau)A_{Y_\bC}(w).
\end{eqnarray*}
The last inequality is by \cite[Lemma 5.5]{BBJ18}. So if we let $K=\frac{C_1+1+C_2}{\tau}$, then $W_{\epsilon,m}\subseteq \mathfrak{W}_K$ (see \eqref{eq-fWK}) for any $\epsilon, m$ \Blue{which implies $\inf_{w\in \mathfrak{W}} h_{\epsilon,m}\le \inf_{w\in \mathfrak{W}_K} h_{\epsilon,m}\le \inf_{w\in W_{\epsilon,m}}h_{\epsilon,m}$. So by \eqref{eq-bIepm} and \eqref{eq-hepless} we get:}
\begin{equation}
{\bf I}_\epsilon=\inf_{w\in \mathfrak{W}_K}h_\epsilon(w), \quad {\bf I}_{\epsilon,m}=\inf_{w\in \mathfrak{W}_K}h_{\epsilon, m}(w).
\end{equation}
This proves the statement in (1).

Moreover, for any $w\in \mathfrak{W}_K$ we then have:
\begin{equation}
h_\epsilon(w)\le h_{\epsilon,m}(w)\le h_\epsilon(w)+\frac{K}{m}.
\end{equation}
This proves the first estimate in \eqref{eq-unidiffh}. The second inequality was proved in \cite[Proposition 4.6]{LTW19}.
Finally the estimate \eqref{eq-Lslopeuniconv} follows from the first two statements, and the limit \eqref{eq-limLNAep} also follows formally from the second estimate in \eqref{eq-unidiffh}.

\end{proof}
 The following proposition says that the infimum in \eqref{eq-LBexpan} can be taken among $\bG$-invariant valuations.
\begin{prop}\label{prop-infGinv}
For any $0\le \epsilon \ll 1$, let $\Phi_\epsilon=\{\vphi_\epsilon(s)\} \subset (\mcE^1(Y, L_\epsilon))^\bK\times\bR$ be as before.
If we let $\mathfrak{W}^\bG$ denote the set of $\bC^*\times\bG$ invariant divisorial valuations $w$ on $Y\times\bC$ with $w(t)=1$. Then we have:
\begin{equation}\label{eq-LslopeGinv}
\bL'^\infty(\Phi_{\epsilon})=\inf_{w\in \mathfrak{W}^\bG} h_{\epsilon}(w)-1.
\end{equation}
\end{prop}
\begin{proof}

Note that $\Phi_{\epsilon,m}$ is associated to $\bC^*\times\bG$-equivariant test configuration $(\mcY_{\epsilon,m}, \mcB_{\epsilon,m}, \mcL_{\epsilon,m})$.
By choosing a $\bC^*\times\bG$-equivariant log resolutions in Remark \ref{rem-calLNA} and arguing as in the proof of the above proposition, we see that the following infimum calculating $\bL'^\infty(\Phi_{\epsilon,m})$ can be taken over $\mathfrak{W}^\bG\cap \mathfrak{W}_K$:
\begin{equation*}
\bL'^\infty(\Phi_{\epsilon,m})=\inf_{w\in \mathfrak{W}} h_{\epsilon,m}(w)=\inf_{w\in \mathfrak{W}^\bG\cap \mathfrak{W}_K}h_{\epsilon,m}(w)
\end{equation*} 
For $\bL'^\infty(\Phi_\epsilon)$, we can use \eqref{eq-unidiffh} to estimate:
\begin{equation}
\left|\inf_{w\in \mathfrak{W}^\bG\cap \mathfrak{W}_K}h_\epsilon-\inf_{w\in \mathfrak{W}^\bG\cap \mathfrak{W}_K}h_{\epsilon,m}\right|\le C'\frac{1}{m}
\end{equation}
So we can let $m\rightarrow+\infty$ and use \eqref{eq-limLNAm} to conclude.

\end{proof}

\subsection{Step 4: Completion of the proof}\label{sec-step4}

With the above preparations, we can complete the proof of our main result. On the one hand,
by \eqref{eq-Dvphidec}, 
\begin{eqnarray}\label{eq-Dslopeneg}
\bL'^\infty(\Phi)&=&\lim_{s\rightarrow+\infty} \frac{\bL(\vphi(s))}{\Blue{s}}
=
\lim_{s\rightarrow+\infty} \frac{\bfD(\vphi(s))}{s}+\lim_{s\rightarrow+\infty}\frac{\bfE(\vphi(s))}{s}\nonumber\\
&\le& \Blue{0+\bfE'^\infty(\Phi)}=-1.
\end{eqnarray}
\Blue{
For any $k\in \bN$, by using Proposition \ref{prop-infGinv} for $\epsilon=0$, we can choose a sequence of $\bG$-invariant divisorial valuations $v_k\in \Blue{\Xdiv}$ such that
\begin{equation}\label{eq-vkappr}
\bL'^\infty(\Phi)\le A_{(Y,B_0)}(v_k)-G(v_k)(\mu^*\Phi)<\bL'^\infty(\Phi)+\frac{1}{k},
\end{equation}
and $A_{(X,D)}(v_k)\le K-1$ where the constant $K$ is from Proposition \ref{prop-uniformLest}. 
Here we used the identity $h_0(G(v_k))-1=A_{(Y,B_0)}(v_k)-G(v_k)(\mu^*\Phi)$ obtained from the expression \eqref{eq-heps}.}
Note that $\bfL'^\infty(\Phi)$ is indeed finite by \cite[Theorem 5.4]{BBJ18}.

By Corollary \ref{cor-Gunival}, there exist $\delta=\delta_\bG(X,D)>1$ and $\xi_k\in N_\bR$ such that 
\begin{equation}\label{eq-Avkxi}
A_{(X,D)}(v_{k,\xi_k})\ge \delta S_{L}(v_{k,\xi_k})
\end{equation} 
where $L=-K_X-D$. 
We claim that $|\xi_k|$ is uniformly bounded. To see this first recall that $\Fut_{(Z,D)}\equiv 0$ on $\mathfrak{t}$ under the assumption of $\bG$-uniform Ding-stability. By using \eqref{eq-betavxi}, we then have
\begin{eqnarray}
0\le A_{(X,D)}(v_{k,\xi_k})-\delta S_{L}(v_{k,\xi_k})&=&\delta (A_{X,D}(v_{k,\xi_k})-S_{L}(v_{k,\xi_k}))-(\delta-1)A_{(X,D)}(v_{k,\xi_k})\nonumber \\
&=&\delta (A_{(X,D)}(v_k)-S_{L}(v_k))-(\delta-1)A_{(X,D)}(v_{k, \xi_k}).
\end{eqnarray}
So we get the estimate:
\begin{eqnarray*}
A_{(X,D)}(v_{k,\xi_k})\le \frac{\delta}{\delta-1}A_{(X,D)}(v_k)\le \frac{\delta}{\delta-1}(K-1)=C_1.
\end{eqnarray*}
This implies $|\xi_k|\le C_2$ for some $C_2$ independent of $k$. Indeed, we have $S_{L}(v_{k,\xi_k})\le \delta^{-1} C_1$, which implies $\Lam^\NA(\mcF_{v_{k, \xi_k}})\le (n+1)\delta^{-1} C_1$ (see \eqref{eq-SvsJ}). By the proof of Lemma \ref{lem-JNAproper}, we get $|\xi_k|\le C_2$ for some $C_2>0$ independent of $k$.

If $S_{L}(v_{k,\xi_k})=0$ then  $v_{k,\xi_k}$ is trivial and $S_{{L}_\epsilon}(v_{k,\xi_k})=0$ for $\epsilon\ge 0$. Otherwise, $S_{{L}_\epsilon}(v_{k,\xi_k})\neq 0$ for $0\le \epsilon\ll 1$. \Blue{Note that we can naturally identify the $\bQ$-line bundle $L=-(K_X+D)$ on $X$ with its pull back $\mu^*L=L_0=-K_Y-B_0$ on $Y$. Moreover recall that for any $0\le \epsilon \ll 1$ with $\epsilon\in \bQ$, $L_\epsilon$ is a $\bQ$-line bundle on $Y$ satisfying $-K_Y-B_\epsilon=L_\epsilon$ (see \eqref{eq-KBL}). %Consider the quantity: 
We have the following crucial  estimate similar to \cite[Proposition 4.16]{LTW19}.}
\Blue{\begin{lem}
Let $v$ be any non-trivial divisorial valuation on $X$. For any $0\le \epsilon\ll 1$, set
\begin{equation}
\Theta(\epsilon):=\frac{A_{(Y, B_\epsilon)}(v)(-K_Y-B_\epsilon)^{\cdot n}}{\int_0^\infty \vol_Y(-K_Y-B_\epsilon-x\cdot v)dx}.
\end{equation}
Then there exists $C'>0$ independent of $\epsilon$ and $v$ such that
\begin{equation}\label{eq-Thetaest}
\Theta(\epsilon)\ge (1-C'\epsilon)\Theta(0).
\end{equation}
\end{lem}
\begin{proof}
Assume $v:=q\cdot\ord_F$ with $q\in \bQ_{>0}$. By rescaling invariance of $\Theta$ with respect to $v$, we can assume $q=1$.%Set $f_\epsilon(v)=\frac{A_{(Y,B_\epsilon)}(v)}{S_{L_\epsilon}(v)}$ and $f_0=\frac{A_{(X,D)}(v)}{S_L(v)}$. 
We need to estimate:
\begin{equation}\label{eq-3ratios}
\frac{\Theta(\epsilon)}{\Theta(0)}=\frac{A_{(Y,B_\epsilon)}(v)}{A_{(Y,B_0)}(v)}\frac{L_{\epsilon}^{\cdot n}}{L^{\cdot n}}\frac{\int_0^{+\infty}\vol(L-tF)dt}{\int_0^{+\infty}\vol(L_\epsilon-tF)dt}.
\end{equation}
The first factor can be estimated as follows (recall that $B_\epsilon=B_0+\frac{\epsilon}{1+\epsilon}E_\theta$ in \eqref{eq-Bepsilon} and $B_0=\Delta_0-F$ with $\Delta_0$ and $F$ given in \eqref{eq-Delta0F}):
\begin{align*}
\frac{A_{(Y,B_\epsilon)}(v)}{A_{(X,D)}(v)}&=\frac{A_Y(v)-v(B_\epsilon)}{A_Y(v)-v(B_0)}=\frac{A_Y(v)-v(B_0)-\frac{\epsilon}{1+\epsilon}v(E_\theta)}
{A_Y(v)-v(B_0)}\\
&=1-\frac{\epsilon}{1+\epsilon}\left(\frac{A_Y(v)-v(B_0)}{v(E_\theta)}\right)^{-1}\ge 1-\frac{\epsilon}{1+\epsilon}\left(\frac{A_Y(v)-v(\Delta_0)}{v(E_\theta)}\right)^{-1}\\
&\ge1-\frac{\epsilon}{1+\epsilon}\lct(Y,\Delta_0; E_\theta)^{-1}.
\end{align*}
Note that the last quantity on the right-hand-side does not depend on $v$ and approaches 1 as $\epsilon\rightarrow 0$. The second ratio $\frac{L_\epsilon^{\cdot n}}{L^{\cdot n}}=\frac{L_\epsilon^{\cdot n}}{L_0^{\cdot n}}$ does not depend on $v$ and approaches 1 as $\epsilon\rightarrow 0$. We estimate the third ratio in \eqref{eq-3ratios} by simply estimating the integrand:
\begin{equation*}
\vol(L_\epsilon-tF)=\vol(\mu^*L-\frac{\epsilon}{1+\epsilon}E_\theta-tF)\le \vol(\mu^*L-tF)
\end{equation*}
because $E_\theta$ is effective. So the third ratio is always greater than 1. Combining these estimates of three ratios in \eqref{eq-3ratios}, the estimate \eqref{eq-Thetaest} follows easily. 
\end{proof}}
With \eqref{eq-Thetaest} proved, we can set $\delta':=1+\frac{\delta-1}{2}>1$. Then when $\epsilon$ is sufficiently small, we have
\begin{equation}\label{eq-AYBSep}
A_{(Y,B_\epsilon)}(v_{k,\xi_k})=\Theta(\epsilon)\delta S_{{L}_\epsilon}(v_{k,\xi_k})\ge (1-C'\epsilon)\delta S_{{L}_\epsilon}(v_{k,\xi_k})\ge \delta' S_{{L}_\epsilon}(v_{k,\xi_k}).
\end{equation}
\Blue{Now we have all the estimates available to complete the proof. 
First, by \eqref{eq-Lepconv}, there exists $\epsilon_0=\epsilon_0(k)>0$ such that for any $0\le \epsilon\le \epsilon_0$
\begin{equation}\label{eq-final1}
|\bfL'^\infty(\Phi_\epsilon)-\bfL'^\infty(\Phi)|\le k^{-1}.
\end{equation}
By \eqref{eq-Lslopeuniconv}, we can also assume that there exists $m_0=m_0(k)$ such that for $m\ge m_0$ and for any $0<\epsilon<\epsilon_0$:
\begin{equation*}
|\bfL^\NA(\phi_{\epsilon,m})-\bfL'^\infty(\Phi_\epsilon)|\le k^{-1}
\end{equation*}
which together with \eqref{eq-final1} implies:
\begin{equation}\label{eq-final2}
|\bfL^\NA(\phi_{\epsilon,m})-\bfL'^\infty(\Phi)|\le 2k^{-1}.
\end{equation}
Moreover by \eqref{eq-unidiffh}, by possibly replacing $\epsilon_0$ with $\min\{\epsilon_0, C'^{-1}k^{-1}\}$ and replacing $m_0$ with 
$\max\{m_0, C' k\}$, we can assume that $m_0$ and $\epsilon_0$ are chosen such that for any $m\ge m_0$ and any $0<\epsilon<\epsilon_0$, we have
\begin{align*}
|h_{\epsilon,m}(G(v_k))-h_0(G(v_k))|&\le |h_{\epsilon,m}(G(v_k))-h_\epsilon(G(v_k))|+|h_\epsilon(G(v_k))-h_0(G(v_k))|\\
&\le  C' (m_0^{-1}+\epsilon_0)<2 k^{-1}. 
\end{align*}
On the other hand, from the definition \eqref{eq-hepsm}-\eqref{eq-heps}, we know that:
\begin{align}\label{eq-final3}
&h_{\epsilon,m}(G(v_k))-1=A_{(Y,B_\epsilon)}(v_k)+\phi_{\epsilon,m}(v_k), \quad h_0(G(v_k))-1=A_{(X,D)}(v_k)-G(v_k)(\Phi).
\end{align}
So if we combine \eqref{eq-vkappr} with \eqref{eq-final2} and \eqref{eq-final3} , we get if $m\ge m_0$ and $0<\epsilon\le \epsilon_0$, then
\begin{align*}
&\hskip 5mm |A_{(Y,B_\epsilon)}(v_k)+\phi_{\epsilon,m}(v_k)-\bfL^\NA(\phi_{\epsilon,m})|=|h_{\epsilon,m}(G(v_k))-1-\bfL^\NA(\phi_{\epsilon,m})|\\
&\le |h_{\epsilon,m}(G(v_k))-h_0(G(v_k))|+|\bfL^\NA(\phi_{\epsilon,m})-\bfL'^\infty(\Phi)|+|h_0(G(v_k))-1-\bfL'^\infty(\Phi)|\\
&\le 5k^{-1}. 
\end{align*} 
Roughly speaking this means that $v_k$ approximately computes the infimum in the definition of $\bfL^\NA(\phi_{\epsilon,m})$ (see \eqref{eq-Lviah}).
\Blue{In the following estimates, $\psi_\epsilon$ is the smooth reference potential on $L_\epsilon=-(K_{Y}+B_\epsilon)$ defined in \eqref{eq-psiep}. 
$\Phi_\epsilon=\{\vphi_\epsilon(s)=\frac{1}{1+\epsilon}(\vphi(s)+\epsilon\psi_P)\}_{s\in [0, +\infty)}$ and $\Phi_{\epsilon,m,-\xi_k}=\{\sigma_{-\xi_k}(s)^*\vphi_\epsilon(s)\}_{s\in [0, +\infty)}$. }
We can continue to estimate:
for any $m\ge m_0$ and $0<\epsilon\le \epsilon_0$, }
\begin{align*}
&\quad \bL_{(Y, B_\epsilon)}^\NA(\phi_{\epsilon,m})+5k^{-1}\ge A_{(Y, B_\epsilon)}(v_k)+\phi_{\epsilon,m}(v_k) & \\
&=A_{(Y, B_\epsilon)}(v_{k,\xi_k})+\phi_{\epsilon,m,-\xi_k}(v_{k, \xi_k}) & (\text{by } \eqref{eq-A+phivxi}) \\
&\ge \delta' S_{{L}_\epsilon}(v_{k, \xi_k})+\phi_{\epsilon,m,-\xi_k}(v_{k, \xi_k})& (\text{by } \eqref{eq-AYBSep})\\
&=\delta' (S_{{L}_\epsilon}(v_{k, \xi_k})+\delta'^{-1}\phi_{\epsilon,m,-\xi_k}(v_{k, \xi_k})) & (\text{note } \delta'>1)\\
&\ge \delta' \bfE_{{L}_\epsilon}^\NA(\delta'^{-1}\phi_{\epsilon,m,-\xi_k}) & (\text{by \eqref{eq-SL2bfE}}) \\ %or \cite[\text{Proposition 7.5}]{BoJ18a}})\\
&=\left(-\delta' \bfJ^\NA_{{L}_\epsilon}(\delta'^{-1}\phi_{\epsilon,m,-\xi_k})+\bfJ^\NA_{{L}_\epsilon}(\phi_{\epsilon,m,-\xi_k})\right)+\bfE^\NA_{{L}_\epsilon}(\phi_{\epsilon,m,-\xi_k}) & \text{ (by \eqref{eq-JNAphi}  and \eqref{eq-FeqTC}) }
\\
&\ge(1-\delta'^{-1/n})\bfJ^\NA_{{L}_\epsilon}(\phi_{\epsilon,m,-\xi_k})+\bfE^\NA_{{L}_\epsilon}(\phi_{\epsilon,m,-\xi_k}) & (\text{by } \eqref{eq-NADingineq})\\ %\cite[\text{Lemma 6.17}]{BoJ18a})\\
&=(1-\delta'^{-1/n})({\bf \Lambda}'^\infty_{\psi_\epsilon}(\Phi_{\epsilon,m,-\xi_k})-\bfE'^\infty_{\psi_\epsilon}(\Phi_{\epsilon,m,-\xi_k}))+\bfE'^\infty_{\psi_\epsilon}(\Phi_{\epsilon,m,-\xi_k}) 
& (\text{by Proposition \ref{prop-BHJslope}})
\\
&=(1-\delta'^{-1/n}){\bf \Lambda}'^\infty_{\psi_\epsilon}(\Phi_{\epsilon,m,-\xi_k})+\delta'^{-1/n} \bfE'^\infty_{\psi_\epsilon}(\Phi_{\epsilon,m,-\xi_k}) & \text{(re-arrange)}\\
&\ge(1-\delta'^{-1/n}){\bf \Lambda}'^\infty_{\psi_\epsilon}(\Phi_{\epsilon,-\xi_k})+\delta'^{-1/n}\bfE'^\infty_{{\psi}_\epsilon}(\Phi_{\epsilon,-\xi_k})& (\text{by } \eqref{eq-ENAepmlb}-\eqref{eq-KNAepmlb})\\
&=(1-\delta'^{-1/n})\bfJ'^\infty_{{\psi}_\epsilon}(\Phi_{\epsilon,-\xi_k})+\bfE'^\infty_{{\psi}_\epsilon}(\Phi_{\epsilon,-\xi_k}). & \text{(re-arrange)}
\end{align*}

Letting $m\rightarrow+\infty$ and using \eqref{eq-limLNAm}, we get the following inequality:
\begin{eqnarray*}
\bL'^\infty_{(Y, B_\epsilon)}({\Phi}_\epsilon)+5 k^{-1}\ge (1-\delta'^{-1/n}) \bfJ'^\infty_{{\psi}_\epsilon}(\Phi_{\epsilon,-\xi_k})+ \bfE'^\infty_{{\psi}_\epsilon}(\Phi_{\epsilon, -\xi_k}).
\end{eqnarray*}
\Blue{Observe that $\bfE'^\infty_{\psi_\epsilon}(\Phi_{\epsilon, -\xi_k})=(1+\epsilon)^{-1}\bfE'^\infty_{\hat{\psi}_\epsilon}(\hat{\Phi}_{\epsilon,-\xi_k})$}. 
So we can let $\epsilon\rightarrow 0$, and use \eqref{eq-limLNAep} and \eqref{eq-limENAep}-\eqref{eq-limKNAep} to get:
\begin{align}
\bL'^\infty(\Phi)+5 k^{-1}&\ge (1-\delta'^{-1/n})\bfJ'^\infty(\Phi_{-\xi_k})+\bfE'^\infty(\Phi_{-\xi_k}) & \nonumber\\
&=(1-\delta'^{-1/n})\bfJ'^\infty(\Phi_{-\xi_k})+\bfE'^\infty(\Phi) & (\text{by } \Fut(\xi_k)=0)\nonumber \\
&\ge(1-\delta'^{-1/n})\chi-1. & (\text{ by Corollary \ref{cor-minchi} })  \nonumber
\end{align}
But when $k\gg 1$, this contradicts \eqref{eq-Dslopeneg} because $\chi>0$.
%\begin{rem}
%Corresponding to Remark \ref{rem-Jxiproper}, the above contradiction chain can be simplified if we know the expected inequality $\inf_{\xi\in N_\bR}\bfJ'^\infty(\Phi_\xi)>0$ is true.
%\end{rem}
\Blue{\begin{rem}\label{rem-Hisamoto}
In the special case when $X$ is smooth and $D=0$, Hisamoto claimed to prove Theorem \ref{thm-YTD} involving only Ding-stability in the first version of \cite{His19}, with a different argument which also depends on Berman-Boucksom-Jonsson's variational approach. 
%a different argument for the statement involving only Ding-stability, which depends on Berman-Boucksom-Jonsson's variational approach, was claimed by Hisamoto in \cite{His19} (however see Remark \ref{rem-Hisamoto}). 
%In the above proof, if $X$ is already smooth, then we can set $(Y, B)=(X, \emptyset)$ to give a different proof of Hisamoto's claimed %result. 
Hisamoto's original argument was however not complete and corrections have been made in a recent revision based on the uniform boundedness of $\xi_k$ and the monotonicity of $\Lam$ from the above proof.
\end{rem}
}

\appendix

\section{$\bG$-equivariant versions of results from \cite{Fuj19a, LX14}
%On the proof of inequality \eqref{eq-D-Jxidecrease}
}\label{app-MMP}
\Blue{Let $X$ be a normal projective variety and $D$ be a $\bQ$-divisor. Assume that $(X, D)$ is a log Fano pair, which means that $L:=-(K_X+D)$ is an ample $\bQ$-Cartier divisor and $(X, D)$ has at worst klt singularities.}
In this section we explain that the minimal model program (MMP) techniques in \cite{LX14} can be applied in our $\bG$-equivariant setting to simplify $\bG$-equivariant test configurations. This allows us to prove the following $\bG$-equivariant version of the result from \cite{Fuj19a}.
\begin{thm}
Assume that $\bG$ is a connected reductive group acting algebraically on $(X, D, L)$. 
Let $(\mcX, \mcD, \mcL)$ be a $\bG$-equivariant \Blue{test configuration} of $(X, D, L)$. There exist $d\in \bZ_{>0}$ and a $\bG$-equivariant special test configuration $(\mcX^s, \mcD^s, \mcL^s)$ such that for any $\epsilon \in [0,1]$ and any $\xi\in N_\bQ$, we have:
\begin{eqnarray*}
d\left(\bfD^\NA(\mcX_\xi, \mcL_\xi)-\epsilon\cdot \bfJ^\NA(\mcX_\xi, \mcL_\xi)\right)\ge \bfD^\NA(\mcX^s_\xi,  \mcL_\xi)-\epsilon\cdot \bfJ^\NA(\mcX^s_\xi, \mcL^s_\xi).
\end{eqnarray*}
\end{thm}
We will just explain the key points of the original proof that need to be modified to get this result. 
For simplicity of notations, we assume that $D=\emptyset$ like in \cite{LX14}. The logarithmic case can be obtained by running a log MMP and using the same argument (see \cite[section 6]{Fuj19a}).  
\begin{proof}[Sketch of the proof]
There are three main steps of using MMP process in \cite{LX14} to obtain a special test configuration from any given test configuration. Step 1 is to use semistable reduction and run a relative MMP to get the log canonical modification $(\mcX^\lc, \mcL^\lc)$. Step 2 is to run an MMP with rescaling to get $(\mcX^\ac, \mcL^\ac)$ with $\mcL^\ac=-K_{\mcX^\ac}$. Step 3 is to do a Fano extension to get a special test configuration $(\mcX^s, -K_{\mcX^s})$. 
There are essentially two key facts that \Blue{make this process work in a $\bG$-equivariant fashion}. 
The first is the well-known fact that resolution of singularities can be carried out in the $\bG$-equivariant fashion. This follows from the existence of functorial resolution of singularities (see \cite{Kol07}).  The second fact is that, under the assumption that $\bG$ is connected, the outputs of MMP are automatically $\bG$-equivariant. 
Indeed, it is enough to see that the extremal contractions are $\bG$-equivariant, since then the flips are also $\bG$-equivariant and the result from \cite{BCHM} including termination applies directly. A quick way to get this $\bG$-equivariance is by using a result of Blanchard in the following general form \footnote{The author learned this application of Blanchard's result to the equivariant MMP from \cite{Pas17}.}:
\begin{thm}[{\cite[Proposition 4.2.1]{BSU13}, see also \cite[\S 2.4]{Akh95}}]
Let $f: X\rightarrow Y$ be a proper morphism of varieties (or even general schemes) such that $f_*(\mcO_X)=\mcO_Y$. Let $\bG$ be a connected group scheme acting on $X$. Then there exists a unique $\bG$-action on $Y$ such that $f$ is $\bG$-equivariant.
\end{thm}
Roughly speaking, this says that an algebraic action by a connected group $\bG$ moves points in the same fibre to points in the same fibre. Note that this theorem applies directly to any extremal contraction $f$ in MMP, which by definition satisfies $f_*\mcO_X=\mcO_Y$ (see \cite[Definition 1.25]{KM98}). 
Alternatively, as pointed out in \cite[1.5]{And01} and \cite[pg. 228]{LX14}, the $\bG$-equivariance of the MMP comes from the facts the connected group $\bG$ carries any curve to a numerically equivalent curve, and that an extremal contraction contracts all and only the set of numerically equivalent curves in an extremal ray.

Moreover, because the intersection numbers are functorial under base change and birational morphisms, we can verify the inequality in the theorem by adapting the calculation in \cite{Fuj18} twisted by base change and by birational map $\bar{\sigma}_{b\xi}$ away from the central fiber. We will now write down the details of calculations for each step. 

\begin{enumerate}
\item Step 1: Using the same argument as in \cite[Proof of Lemma 5]{LX14} by replacing $\bC^*$ by $\bC^*\times \bG$ (which is based on the existence of functorial resolution of singularities), we know that there exist a base change $z^d: \bC\rightarrow \bC$ and a semistable family $\mcY$ over $\bC$ with a $(\bC^*\times \bG)$-equivariant morphism $\pi: \mcY\rightarrow \tilde{\mcX}$ that is a log resolution of $(\tilde{\mcX}, \tilde{\mcX}_0)$, where $\tilde{\mcX}$ is the normalization of $(\mcX\times_{\bC, z^d}\bC)$ with a natural morphism ${\rm m}_d: \tilde{\mcX}\rightarrow \mcX$. Set 
\begin{equation}
\rho: \mcX^{\rm lc}={\rm Proj}\; R(\mcY/\tilde{\mcX}, K_{\mcY})\rightarrow \tilde{\mcX}.
\end{equation}
\Blue{Then the projective morphism $\rho$ is $(\bC^*\times\bG)$-equivariant and is the log canonical modification of $(\tilde{\mcX}, \tilde{\mcX}_0)$, which means that $(\mcX^\lc, \mcX^\lc_0)$ has log canonical singularities and $K_{\mcX^\lc}+\mcX^\lc_0$ is ample over $\tilde{\mcX}$. See \cite[Proposition 2]{LX14}. }
%Then the log canonical modification is also $(\bC^*\times\bG)$-equivariant.
%Let $(X, D)$ be a log Fano pair and $(\mcX, \mcD, \mcL)/\bC$ be a normal, ample test configuration for $(X, D, -(K_X+D))$. Then there exist $d\in \bZ_{>0}$, a projective
%birationally $\bC^*$-equivariant morphism $\pi: \mcX^\lc\rightarrow \mcX^{(d)}$ and a normal, ample test configuration $(\mcX^\lc, \mcD^\lc, \mcL^\lc)$ for $(X, -(K_X+D))$ such that
%\begin{enumerate}[(1)]
%\item $(\mcX^\lc, \mcD^\lc+\mcX^\lc_0)$ is log canonical.
%\item For any $\epsilon\in [0,1]$ and any $\xi\in N_\bQ$, we have:
%\begin{equation}
%d\left(\bfD^\NA(\mcX_\xi, \mcD_\xi, \mcL_\xi)-\epsilon\cdot \bfJ^\NA(\mcX_\xi, \mcL_\xi)\right)\ge \bfD^\NA(\mcX^\lc_\xi, \mcD^\lc_\xi, \mcL^\lc_\xi)-\epsilon\cdot \bfJ^\NA(\mcX^\lc_\xi, \mcL^\lc_\xi).
%\end{equation}
%\end{enumerate}
%As in \cite{LX14}, there exist $d\in \bZ_{>0}$ and the log canonical modification $\pi: \mcX^\lc\rightarrow (\mcX^{(d)}, \mcX^{(d)}_0)$. 

Set $\mcL^\lc_0=\pi^*{\rm m}_d^*\mcL$ and let $E$ be the $\bQ$-divisor on $\mcX^\lc$ defined by
$$
{\rm Supp}(E)\subset \mcX^\lc_0, \quad E \sim_\bQ K_{\mcX^\lc/\bC}+\mcL^\lc_0.
$$
Set $\mcL^\lc_t=\mcL^\lc_0+tE$. Because $E$ is relatively ample over $\tilde{\mcX}$, $(\mcX^\lc, \mcL^\lc_t)/\bC$ is a normal, ample test configuration for $(X, -K_X)$ for $0<t\ll 1$ (see \cite[Theorem 2]{LX14}), which is $(\bC^*\times\bG)$-equivariant. % satisfying $\CM(\mcX^{\lc}, \mcD^\lc, \mcL^\lc_t)\le d\cdot \CM(\mcX, \mcD, \mcL)$. 
Let $\mcX^\lc_0=\sum_{i=1}^p E_i$ be the irreducible decomposition and set $E:=\sum_{i=1}^p e_i E_i$. Assume $e_1\le \cdots\le e_p$. Set $\Delta_t:=-K_{\mcX^\lc}-\mcL^\lc_t=-(1+t)E$. Because $(\mcX^\lc, \mcX^\lc_0)$ is log canonical,  we can calculate:
\begin{equation}
\bL^\NA(\mcX^\lc, \mcL^\lc_t)=\lct(\mcX^\lc, \Delta_t; \mcX^\lc_0)=1+(1+t)e_1.
\end{equation}

Choose $b\in \bZ_{>0}$ such that $b\xi\in N_\bZ$. We consider the following commutative diagrams, where $\mcZ$ is the normalization of the graph $\bar{\sigma}_{b\xi}\circ \mathfrak{i}_{b\eta}$. 
\begin{equation}\label{eq-diagtwist}
\xymatrix{% @R=1.5pc @C=0.5pc{
& \ar_{\Pi}[ld] \mcZ   \ar^{\Theta}[rd] & \\
(X\times\bP^1)^{(b)}  \ar^{{\rm m}_b}[d] \ar^{\mathfrak{i}_{b\eta}}@{-->}[r] & (\mcX^\lc)^{(b)} \ar^{{\rm m}_b}[d]  \ar^{\bar{\sigma}_{b\xi}}@{-->}[r] & (\mcX^\lc)^{(b)} \ar^{{\rm m}_b}[d]   \\
X\times\bP^1\ar^{\mathfrak{i}_\eta}@{-->}[r]  & \mcX^\lc & \mcX^\lc 
}
\end{equation}
For simplicity of notations, set $\tilde{\phi}_{t, b\xi}:=\Theta^*{\rm m}_b^*\bar{\mcL}^\lc_t$ and $\tilde{\psi}:=\Pi^*{\rm m}_b^* p_1^*(-K_X)$. Note that $\bfD^\NA$ and $\bfL^\NA$ are multiplicative under base change (see \cite[Proposition 2.5.(3)]{Fuj18}). Moreover, $\bL^\NA$ is invariant under twisting: $\bL^\NA(\mcX^\lc_\xi, \mcL^\lc_{t, \xi})=\bL(\mcX^\lc,  \mcL^\lc_{t})$ (by \eqref{eq-bLYBtwist}). Then we can calculate:
\begin{eqnarray*}
&&b V\cdot  (\bfD^\NA(\mcX^\lc_\xi, (\mcL^\lc_t)_\xi)-\epsilon \bfJ^\NA(\mcX^\lc_\xi, (\mcL^\lc_t)_\xi))\\
&=&V\cdot \left((1-\epsilon)\bfE^\NA+\bfL^\NA-\epsilon\Lam^\NA\right)(\mcX^\lc_{b\xi}, (\mcL^\lc_t)_{b\xi})\\
&=&
-\frac{1-\epsilon}{n+1}\tilde{\phi}_{t,b\xi}^{\cdot n+1}+1+(1+t) e_1 V-\epsilon \tilde{\psi}^{\cdot n}\cdot \tilde{\phi}_{t,b\xi}.
\end{eqnarray*}
Taking derivative with respect to $t$, we get, for $0\le t\ll 1$:
\begin{eqnarray*}
&&b V\cdot \frac{d}{dt}\left[\bfD^\NA(\mcX^\lc_\xi, (\mcL^\lc_t)_\xi)-\epsilon \bfJ^\NA(\mcX^\lc_\xi, (\mcL^\lc_t)_\xi)\right]\\
&=&-(1-\epsilon)\tilde{\phi}_{t,b\xi}^{\cdot n}\cdot \Theta^*{\rm m}_b^*E+e_1 V-\epsilon \tilde{\psi}^{\cdot n}\cdot \Theta^*{\rm m}_b^*E\\
&=&-(1-\epsilon)\tilde{\phi}^{\cdot n}_{t, b\xi}\cdot \Theta^*{\rm m}_b^*\sum_{j=1}^p(e_j-e_1)E_j-\epsilon \tilde{\psi}^{n-1}\cdot \Theta^*{\rm m}_b^*\sum_{j=1}^p (e_j-e_1)E_j\le 0.
\end{eqnarray*}
The last inequality uses the relative nefness of $\tilde{\phi}_{t,b\xi}$ and $\tilde{\psi}$. 
%\begin{eqnarray*}
%&&(n+1)V\left[d \left(\bfD^\NA(\mcX_\xi, \mcL_\xi)-\epsilon \bfJ^\NA(\mcX_\xi, \mcL_\xi)\right)\right.\\
%&&\hskip 2cm \left.-\left(\bfD^\NA(\mcX^\lc_\xi, \mcL^\lc_\xi)-\epsilon \bfJ^\NA(\mcX^\lc_\xi, \mcL^\lc_\xi)\right)\right]\\
%&=&(n+1)V b^{-1} \left[\left(\bfD^\NA\left((\mcX_\xi, \mcL_{0,\xi})^{(b)}\right)-\epsilon \bfJ^\NA\left((\mcX_\xi, \mcL_{0,\xi})^{(b)}\right)\right)\right.\\
%&&\hskip 2cm \left.-\left(\bfD^\NA\left((\mcX^\lc_\xi,  \mcL^\lc_\xi)^{(b)}\right)-\epsilon \bfJ^\NA\left((\mcX^\lc_\xi, \mcL^\lc_\xi)^{(b)}\right)\right)\right]\\
%&=&b^{-1}(1-\epsilon)\left(\tilde{\phi}_{t,b\xi}^{\cdot n+1}-\tilde{\phi}_{0,b\xi}^{\cdot n+1}\right)+b^{-1} \epsilon(n+1)t (\tilde{\psi}^{\cdot n}\cdot \Theta^*{\rm m}_b^* E)-(n+1)t e_1 V\\
%&=&(1-\epsilon)t\left(b^{-1}(\tilde{\phi}_{t,b\xi}^{\cdot n+1}-\tilde{\phi}_{0,b\xi}^{\cdot n+1})-(n+1) e_1 V\right)+\epsilon (n+1)t \left(b^{-1}\tilde{\psi}^{\cdot n} \Theta^*{\rm m}_b^*E- e_1 V\right)\\
%&=&(1-\epsilon)tb^{-1} \sum_{i=0}^n \left(\tilde{\phi}_{t,b\xi}^{\cdot i}\cdot \tilde{\phi}_{0,b\xi}^{n-i}\cdot \Theta^*{\rm %m}_b^*\sum_{j=1}^p (e_j-e_1)E_j\right)\\
%&&\hskip 3cm+b^{-1}\epsilon(n+1)t \left(\tilde{\psi}^{\cdot n}\cdot \Theta^*{\rm m}_b^* \sum_{j=1}^p (e_j-e_1) E_j\right)\ge 0.
%\end{eqnarray*}
After integration, we get, for any $0\le t\ll 1$:
\begin{eqnarray}
d\cdot \left(\bfD^\NA(\mcX_\xi, \mcL_\xi)-\epsilon \bfJ^\NA(\mcX_\xi, \mcL_\xi)\right)&=&\bfD^\NA(\mcX^\lc_\xi, (\mcL^\lc_0)_\xi)-\epsilon \bfJ^\NA(\mcX^\lc_\xi, (\mcL^\lc_0)_\xi)\nonumber \\
&\ge& \bfD^\NA(\mcX^\lc_\xi, (\mcL^\lc_t)_\xi)-\epsilon \bfJ^\NA(\mcX^\lc_\xi, (\mcL^\lc_t)_\xi). \label{eq-step1}
\end{eqnarray}
We set $\mcL^\lc=\mcL^\lc_t$ for some fixed $t\in \bQ_{>0}$ sufficiently small.
\item Step 2: With the $(\mcX^\lc, \mcL^\lc)$ from the first step, we run a relative MMP with scaling to get a normal, ample test configuration $(\mcX^\ac, \mcL^\ac)$  for $(X, -K_X)$ such that $(\mcX^\ac, \mcX^\ac_0)$ log canonical and $-K_{\mcX^\ac}\sim_{\bQ}\mcL^\ac$ \Blue{(the superscript ``ac" stands for ``anti-canonical")}. More concretely, take $\ell \gg 1$ such that $\mcH^\lc=\mcL^\lc-(\ell+1)^{-1}(\mcL^\lc+K_{\mcX^\lc})$ is relatively ample. Set $\mcX^0=\mcX^\lc$, $\mcL^0=\mcL^\lc$, $\mcL^0=\mcH^\lc$ and $\lambda_0=\ell_0+1$. Then $K_{\mcX^0}+\lambda_0\mcH^0=\ell \mcL^0$. We run a $K_{\mcX^0}$-MMP over $\bC$ with scaling $\mcH^0$. Then we obtain a sequence of models:
\begin{equation*}
\mcX^0\dasharrow \mcX^1\dasharrow \cdots \dasharrow \mcX^k
\end{equation*}
and a sequence of critical values 
\begin{equation*}
\lambda_{i+1}=\min\{\lambda; K_{\mcX^i}+\lambda \mcH^i \text{ is nef over } \bC\}
\end{equation*}
with $\ell+1=\lambda_0\ge \lambda_1\ge \cdots \ge \lambda_k>\lambda_{k+1}=1$. For any $\lambda_i\ge \lambda\ge \lambda_{i+1}$, let $\mcH^i$ be the pushforward of $\mcH$ to $\mcX^i$ and set:
\begin{equation*}
\mcL^i_\lambda=\frac{1}{\lambda-1}(K_{\mcX^i}+\lambda \mcH^i)=\frac{1}{\lambda-1}(K_{\mcX^i}+\mcH^i)+\mcH^i=:\frac{1}{\lambda-1}E+\mcH^i. 
\end{equation*}
By the earlier discussion, $(\mcX^i, \mcL^i)$ is indeed automatically $(\bC^*\times\bG)$-equivariant and $E$ is a $\bG$-invariant divisor supported on the central fibre $\mcX^i_0$. 

Write $E=\sum^k_{j=1}e_j\mcX^i_{0,j}$ with $e_1\le e_2\le \cdots\le e_k$. Using similar diagram and notations as in Step 1, we can calculate
\begin{equation*}
Vb\cdot \left(\bfD^\NA(\mcX^i_\xi, \mcL^i_\xi)-\epsilon \bfJ^\NA(\mcX^i_\xi, \mcL^i_\xi)\right)=-(1-\epsilon)\frac{\tilde{\phi}_{\lambda,b\xi}^{\cdot n+1}}{n+1}-\epsilon \tilde{\psi}^{\cdot n} \cdot \tilde{\phi}_{\lambda,b\xi}+\frac{\lambda}{\lambda-1}e_1 V
\end{equation*}
whose derivative with respect to $\lambda$ is given by:
\begin{equation}\label{eq-der2}
\frac{1}{(\lambda-1)^2}\sum_i \left((1-\epsilon)\tilde{\phi}_{\lambda,b\xi}^{\cdot n}+\epsilon \tilde{\psi}^{\cdot n}\right)\cdot (e_i-e_1) \Theta^*{\rm m}_b^*E_i\ge 0.
\end{equation}
As in \cite{Fuj19a, LX14}, we verify easily that 
$\bfF^\NA(\mcX^i_\xi, (\mcL^i_{\lambda_{i+1}})_\xi)=\bfF^\NA(\mcX^{i+1}_{\xi}, (\mcL^{i+1}_{\lambda_{i+1}})_\xi)$ for $\bfF\in \{\bfD, \bfJ\}$. 
Moreover by \cite[Lemma 2]{LX14}, we know that $K_{\mcX^k}+\mcL^k_{\lambda_k}\sim_{\bQ} 0$. Set $\mcX^\ac={\rm Proj}\; R(\mcX^k/\bC, \mcL^k_{\lambda_k})$ and $\mcL^\ac=-K_{\mcX^\ac}$. 

After integrating \eqref{eq-der2} over each subinterval $[\lambda_{i+1}, \lambda_{i}]$ and summing up, we then get, for any $\xi\in N_\bQ$:
\begin{equation}\label{eq-step2}
\bfD^\NA(\mcX^\lc_\xi, \mcL^\lc_\xi)-\epsilon \bfJ^\NA(\mcX^\lc_\xi, \mcL^\lc_\xi)\ge \bfD^\NA(\mcX^\ac_\xi, \mcL^\ac_\xi)-\epsilon \bfJ^\NA(\mcX^\ac_\xi, \mcL^\ac_\xi). 
\end{equation}
\item Step 3: Combining the \cite[Theorem 6]{LX14} with the previous discussion on the equivariant MMP, we know that by a base change $\tilde{\mcX}^\ac=\mcX^\ac\times_{\bC, z^d}\bC$ and running an appropriate $(\bC^*\times\bG)$-equivariant MMP, we can get a $(\bC^*\times \bG)$-equivariant diagram:
\begin{equation}\label{eq-Fanoext}
\xymatrix{% @R=1.5pc @C=0.5pc{
& & \ar_{p}[ld] \hat{\mcX}   \ar^{q}[rd] & \\
\tilde{\mcX}^\ac  & \ar_{\pi'}[l] \mcX' \ar@{-->}[rr] &  & \mcX^s
}
\end{equation}
which satisfies $A(\mcX^s_0; \tilde{\mcX}^\ac, \tilde{\mcX}^\ac_0)=0$ and $\pi'$ exactly extracts the divisor $\mcX^s_0$ \Blue{(here the superscript ``s" stands for ``special")}. Then we have $\pi'^*K_{\tilde{\mcX}^\ac}=K_{\mcX'}$ and, with $\mcL^\ac=-K_{\mcX^\ac}$ (resp. $\tilde{\mcL}^\ac=-K_{\tilde{\mcX}^\ac}$) and $\mcL'=-K_{\mcX'}$, 
\begin{eqnarray*}
d \cdot (\bfD^\NA(\mcX^\ac_\xi, \mcL^\ac_\xi)-\epsilon\bfJ^\NA(\mcX^\ac_\xi, \mcL^\ac_\xi))&=&\bfD^\NA(\tilde{\mcX}^\ac_\xi, \tilde{\mcL}^\ac_\xi)-\epsilon \bfJ^\NA(\tilde{\mcX}^\ac_\xi, \tilde{\mcL}^\ac_\xi)\\
&=&\bfD^\NA(\mcX'_\xi, \mcL'_\xi)-\epsilon \bfJ^\NA(\mcX'_\xi, \mcL'_\xi).
\end{eqnarray*} 

Set $E=p^*K_{\mcX'}-q^*K_{\mcX^s}=\sum_{i=1}^q e_i E_i$ with $e_1\le \cdots \le e_q$. Then $E\ge 0$ by the negativity lemma. Set $\hat{\mathcal{L}}_\lambda=-p^* K_{\mcX'/\bC}+\lambda E$. Applying the diagram and notations similar to \eqref{eq-diagtwist} in the first step to $(\hat{\mcX}, \hat{\mcL}_\lambda)$, we get:
\begin{eqnarray*}
&&V b\cdot \frac{d}{d\lambda}\left(\bfD^\NA(\hat{\mcX}_\xi, (\hat{\mcL}_\lambda)_{\xi})-\epsilon \bfJ^\NA(\hat{\mcX}_\xi, (\hat{\mcL}_\lambda)_\xi)\right)\\
&=&
V \cdot \frac{d}{d\lambda}\left((1-\epsilon)\frac{\tilde{\phi}_{\lambda, b\xi}^{\cdot n+1}}{n+1}-\epsilon \tilde{\psi}^{\cdot n}\cdot \tilde{\phi}_{\lambda, b\xi} +\lambda e_1\right)\\
&=&-\sum_{i=1}^q \left((1-\epsilon)\tilde{\phi}_{\lambda, b\xi}^{\cdot n}+\epsilon \tilde{\psi}^{\cdot n}\right)\cdot (e_i-e_1)E_i\le 0.
\end{eqnarray*}
After integration we get:
\begin{eqnarray}
d \left(\bfD^\NA(\mcX^\ac_\xi, \mcL^\ac_\xi)-\epsilon \bfJ^\NA(\mcX^\ac_\xi, \mcL^\ac_\xi)\right)
&=& \bfD^\NA(\hat{\mcX}_\xi, (\hat{\mcL}_0)_\xi)-\epsilon \bfJ^\NA(\hat{\mcX}_\xi, (\hat{\mcL}_0)_\xi)\nonumber \\
&\ge& \bfD^\NA(\hat{\mcX}_\xi, (\hat{\mcL}_1)_\xi)-\epsilon \bfJ^\NA(\hat{\mcX}_\xi, (\hat{\mcL}_1)_\xi)
\nonumber \\
&=&\bfD^\NA(\mcX^s_\xi, \mcL^s_\xi)-\epsilon \bfJ^\NA(\mcX^s_\xi, \mcL^s_\xi). \label{eq-step3}
\end{eqnarray}
\end{enumerate}
Finally, combining the inequalities \eqref{eq-step1}, \eqref{eq-step2} and \eqref{eq-step3} from the above three steps, we get the conclusion.

\end{proof}

\section{Some properties of reductive groups}\label{app-Yu}
{\large{\bf Jun Yu}\footnote{BICMR, Peking University, junyu@bicmr.pku.edu.cn}}

\vskip 4mm
\begin{prop}\label{prop-Yu1}
Let $G$ be a connected reductive complex Lie group and $K$ be a maximal compact subgroup of $G$. Then
we have $N_G(K)=C(G)\cdot K=C(G)_0\cdot K$ where $C(G)_0$ is the identity component of the center $C(G)$ of $G$.
\end{prop}
\begin{proof}
\Blue{First we can decompose $G=C(G)_0\cdot G_1\cdot G_2\cdots G_s$ where $G_1, \dots, G_s$ are simple factors of $G$. By the connectedness assumption, this follows from the corresponding decomposition of the reductive Lie algebra $\mathfrak{g}=Lie(G)$ (see \cite[Corollary 6.4, Theorem 6.24]{Kir08}). }
Write $K_i=K\cap G_i$, $K_0=
K\cap C(G)_0$. Then $K=K_0\cdot K_1\cdots K_s$ and each $K_i$ is a maximal compact subgroup of $G_i$ ($1\leq i\leq s$). 
Clearly $C(G)_0\subset N_G(K)$. 

Conversely, if $g=g_1\cdot g_2\cdots g_s$ (\Blue{with $g_i\in G_i$}) normalizes $K$, then \Blue{$K=gKg^{-1}=K_0\cdot \prod_{i} g_i K_i g_i^{-1}$ which implies $g_i K_i g_i^{-1}=K\cap G_i=K_i$, i.e. each $g_i$ normalizes $K_i$}. Hence it suffices to show that 
$N_{G_i}(K_i)=K_i$ for each $i (1\le i\le s)$. By this discussion, we may assume that $G$ itself is simple. Write $H=N_G(K)$. 
Then $H$ is a closed subgroup of $G$, and $K$ is a normal subgroup of $H$.

Since $G$ is assumed to be simple, the only Lie subalgebras of $\mathfrak{g}={\rm Lie}(G)$ \Blue{containing $\mathfrak{k}=
{\rm Lie}(K)$} are $\mathfrak{g}$ and $\mathfrak{k}$. Thus $\mathfrak{h}={\rm Lie}(H)=\mathfrak{g}$ or $\mathfrak{k}$. 
When $\mathfrak{h}=\mathfrak{g}$, then $H=G$ which is impossible.

When $\mathfrak{h}=\mathfrak{k}$, $H$ is also compact. Then for any $x\in H$, $\Ad(x)\in {\rm GL}(\mathfrak{g})$ is elliptic 
(i.e. \Blue{the} eigenvalues of $\Ad(x)$ all have norm 1). On the other hand, we have the Cartan decomposition $G=K\exp(\mathfrak{p}_{0})$  
where $\mathfrak{p}_{0}$ is the orthogonal complement of $\mathfrak{k}$ in $\mathfrak{g}$ with respect to the Killing form. 
Since for any $g\in\exp(\mathfrak{p}_{0})$, $\Ad(g)$ has positive real eigenvalues, $H\cap\exp(\mathfrak{p}_{0})=1$. Then
\begin{equation*} 
H=H\cap G=H\cap K\exp(\mathfrak{p}_{0})=K\cap(H\cap\exp(\mathfrak{p}_{0}))=K. 
\end{equation*} 
\end{proof}

\begin{prop}\label{prop-Yu2}
Let $G$ be a connected complex reductive Lie group, and $K_1, K_2$ be two maximal compact subgroups.
Assume that $K_1, K_2$ have a common maximal torus $T$.
Set $T_\bC=C_G(T)$ which is a maximal torus of $G$.
Then the following hold true:
\begin{enumerate}[(1)]
\item $K_2=t K_1 t^{-1}=: \Ad(t)K_1$ for some $t\in T_\bC$.
\item If $K_2=\Ad(t) K_1$, then $K_1=K_2$ if and only if $t\in T$.
\end{enumerate}
\end{prop}
\begin{proof}
\begin{enumerate}[(1)]
\item It is well-known that any two maximal compact subgroups of $G$ are conjugate. Thus there exists $g\in G$ such that 
$K_2=\Ad(g)K_1$. Then $\Ad(g)T$ and $T$ are maximal tori of $K_2$. Hence there exists $k_2\in K_2$ such that $\Ad(g)T=
\Ad(k_2)T$. Set $g'=k_2^{-1}g$. Then \begin{equation*}
\Ad(g')K_1=\Ad(k_2)\Ad(g)K_1=\Ad(k_2^{-1})K_2=K_2
\end{equation*}
and
\begin{equation*}
\Ad(g')T=\Ad(k_2^{-1})\Ad(g)T=\Ad(k_2)^{-1}\Ad(k_2)T=T.
\end{equation*}
Thus $g'\in N_G(T)$. It is well-known that $T_\bC:=C_G(T)$ is a maximal torus of $G$ and
\begin{equation*}
N_G(T)=N_{K_2}(T)\cdot T_\bC.
\end{equation*}
Write $g'=n\cdot t$ for $n\in N_{K_2}(T)$ and $t\in T_\bC$. Then
\begin{equation*}
K_2=\Ad(n^{-1})K_2=\Ad(n^{-1})\Ad(g')K_1=\Ad(n^{-1}g')K_1=\Ad(t) K_1.
\end{equation*}
\item Set $\fg={\rm Lie}(G)$ and $\ft_\bC={\rm Lie}(T_\bC)$. Then one has a root space decomposition:
\begin{equation*}
\fg=\ft_\bC\bigoplus \left(\bigoplus_{\alpha\in \Delta}\fg_\alpha\right),
\end{equation*}
where $\Delta=\Delta(\fg, \ft_\bC)$ are roots of $\fg$ with respect to $\ft_\bC$ and $\fg_\alpha$ is the root 
space of $\alpha$. It is well-known that each $\fg_\alpha$ has dimension one. Chose $0\neq X_\alpha\in \fg_\alpha$ 
for any $\alpha\in\Delta$. Choose a positive system $\Delta^+\subset \Delta$. It is well-known that 
\begin{equation}\label{eq-LieK1}
\fk_1:={\rm Lie}(K_1)=\ft\bigoplus \left(\bigoplus_{\alpha\in \Delta^+}\big(\bR(X_\alpha+a_\alpha X_{-\alpha})\oplus
\bR\mathbf{i}(X_\alpha+b_\alpha X_\alpha)\big)\right)
\end{equation} 
for some constants $a_\alpha, b_\alpha\in\mathbb{C}^{\times}$ with $a_\alpha\neq b_\alpha$. 

Set $\fa$ to be the orthogonal complement of $\ft$ in $\ft_\bC$ and $A={\rm exp}(\fa)$. Then $T_\bC=AT$. Assume 
$\Ad(t)K_1=K_1$. Clearly $\Ad(t_1)K_1=K_1$ for $t_1\in T\subset K_1$. So one may assume that $t=a\in A$. For any 
$\alpha\in \Delta^+$, $\alpha(a)>0$. Then the Lie algebra of $\Ad(t)K_1=\Ad(a)K_1$ is equal to: 
\begin{equation}\label{eq-AdtK1}
\ft\bigoplus\left(\bigoplus_{\alpha\in \Delta^+}\big(\bR(X_\alpha+a_\alpha \alpha(a)^{-2}X_{-\alpha})\oplus \bR\mathbf{i}
(X_\alpha+b_\alpha \alpha(a)^{-2}X_{-\alpha})\big)\right). 
\end{equation} 
For it to be equal to $\fk_1$, one must have $\alpha(a)^{-2}=1$ for all $\alpha\in \Delta^+$. Then $a=1$. 
\end{enumerate}
\end{proof}
\begin{prop}\label{prop-Yu3}
Let $G$ be a connected complex reductive Lie group, and $K_1, K_2$ be two maximal compact subgroups.
Assume that $K_1, K_2$ have a common compact subgroup $K$ that in turn contains a maximal compact torus $T$ of $G$.
%Set $T_\bC=C_G(T)$ which is a maximal torus of $G$.
Then $K_2=t K_1 t^{-1}$ for some $t\in C(K_\bC)$ (the center of $K_\bC$).
%\begin{enumerate}[(1)]
%\item $K_2=t K_1 t^{-1}=: \Ad(t)K_1$ for some $t\in T_\bC$.
%\item If $K_2=\Ad(t) K_1$, then $K_1=K_2$ if and only if $t\in T$.
%\end{enumerate}
\end{prop}
\begin{proof}
We use the same notations as in the proof of the last proposition.
By Proposition \ref{prop-Yu2}, there exists $t\in T_\bC$ such that $K_2=t K_1 t^{-1}$. We just need to show that $t\in C(K_\bC)$. Similar to \eqref{eq-LieK1}, we have the decomposition
\begin{eqnarray*}
\fk:={\rm Lie}(K)=\ft\bigoplus \left(\bigoplus_{\alpha\in \Delta'^+}\big(\bR(X_\alpha+a_\alpha X_{-\alpha})\oplus
\bR\mathbf{i}(X_\alpha+b_\alpha X_\alpha)\big)\right),
\end{eqnarray*}
where $\Delta'^+$ is a positive system for ${\rm Lie}(K_\bC)$ with respect to $\ft_\bC$. Because $K_1\subseteq K$, $\fk$ embeds into $\fk_1$ via the inclusion $\Delta'^+\subseteq \Delta^+$. By using the expression in \eqref{eq-AdtK1}, we see that
the Lie algebra of $K_2={\rm Ad}(t)K_1$ contains ${\rm Lie}(K)$ if and only if
$\alpha(a)^{-2}=1$ for all $\alpha\in \Delta'^+$. This holds if and only if $t\in C(K_\bC)$.

\end{proof}

\vskip 3mm

\noindent
Department of Mathematics, Purdue University, West Lafayette, IN 47907-2067

%\noindent
%{\it E-mail address:} li2285@purdue.edu
\noindent
{\it Current address:} Department of Mathematics, Rutgers University, Piscataway, NJ 08854-8019.

\noindent
{\it E-mail address:} chi.li@rutgers.edu

\vskip 2mm

%\noindent
%School of Mathematical Sciences and BICMR, Peking University, Yiheyuan Road 5, Beijing, P.R.China, 100871

%\noindent {\it E-mail address:} tian@math.princeton.edu

%\vskip 2mm

%\noindent
%School of Mathematical Sciences, Zhejiang University, Zheda Road 38, Hangzhou, Zhejiang,
%310027, P.R. China

%\noindent {\it E-mail address:} wfmath@zju.edu.cn


\begin{thebibliography}{999999}

\bibitem
{ACKPZ}
P. Ahag, U. Cegrell, S. Ko\l dziej, H.H. Pham and A. Zeriahi, Partial pluricomplex energy and integrability exponents of plurisubharmonic functions. Adv. in Math. \textbf{222} (2009), 2036-2058. 


\bibitem
{Akh95}
D. N. Akhiezer. Lie group actions in complex analysis, no. E27 in Aspects of Mathematics (Friedr. Vieweg \& Sohn, Braunschweig), 1995.

\bibitem
{ABHLX}
J. Alper, H. Blum, D. Halpern-Leistner, C. Xu. Reductivity of the automorphism group of K-polystable Fano varieties. arXiv:1906.03122.

\bibitem
{AHS08}
K. Altmann, J. Hausen, H. S\"{u}ss. Gluing affine torus actions via divisorial fans. Transform. Groups {\bf 13} (2008), no. 2, 215–242.

\bibitem
{And01}
M. Andreatta. Actions of linear algebraic groups on projective manifolds and minimal model program, Osaka J. Math. {\bf 38} (2001), 151-166.

\bibitem%[Berm15]
{Berm15}
R. Berman. K-stability of ${\bf Q}$-Fano varieties admitting K\"{a}hler-Einstein metrics, Invent. Math. {\bf 203} (2015), no.3, 973-1025.

\bibitem%[BB14]
{BB17}
R. Berman, R. Berndtsson. Convexity of the K-energy on the space of K\"{a}hler metrics, J. Amer. Math. Soc. 30 (2017), 1165-1196.

\bibitem%[BBEGZ]
{BBEGZ}
R. Berman, S. Boucksom, P. Eyssidieux, V. Guedj, A. Zeriahi. K\"{a}hler-Einstein metrics and the K\"{a}hler-Ricci flow on log Fano varieties. J. Reine Angew. Math. {\bf 751} (2019), 27-89.

%\bibitem
%{BBGZ}
%R. Berman, S. Boucksom, V. Guedj, A. Zeriahi, A variational approach to complex Monge-Amp\`{e}re equations. Publ. Math. Inst. Hautes \'{E}tudes Sci. {\bf 117} (2013), 179-245. 

\bibitem
{BFJ16}
S. Boucksom, C. Favre and M. Jonsson. Singular semipositive metrics in non-Archimedean geometry. J. Algebraic Geom. {\bf 25} (2016), 77–139.


\bibitem%[BBJ15]
{BBJ15}
R. Berman, S. Boucksom, M. Jonsson. A variational approach to the Yau-Tian-Donaldson conjecture, arXiv:1509.04561v1.

\bibitem
{BBJ18}
R. Berman, S. Boucksom, M. Jonsson. A variational approach to the Yau-Tian-Donaldson conjecture, to appear in J. Amer. Math. Soc, arXiv:1509.04561v3.

\bibitem
{BCHM}
C. Birkar, P. Cascini, C. D. Hacon, and J. McKernan. Existence of minimal models for varieties of log general type, J. Amer. Math. Soc. {\bf 23} (2010), 405–468.

\bibitem%[BDL17]
{BDL17}
R. Berman, T. Darvas, and C. H. Lu. Convexity of the extended K-energy and the large time behaviour of the weak Calabi flow, Geom. Topol. 21 (2017), 2945-2988..



\bibitem%[Bern15]
{Bern15}
B. Berndtsson. A Brunn-Minkowski type inequality for Fano manifolds and some uniqueness theorems in K\"{a}hler geometry. {Invent. Math.} 
\textbf{200} (2015), no.1, 149-200.



\bibitem%[BJ17]
{BlJ20}
H. Blum and M. Jonsson. Thresholds, valuations and K-stability, Adv. Math. \textbf{365} (2020), 107062, arXiv:1706.04548.





\bibitem%[BC11]
{BC11}
S. Boucksom, H. Chen, Okounkov bodies of filtered linear series, Compositio Math. 147 (2011) 1205-1229.


\bibitem
{BEGZ}
S. Boucksom, P. Eyssidieux, V. Guedj, and A. Zeriahi, Monge-Amp\`{e}re equations in big cohomology classes,
Acta Math., {\bf 205} (2010), 199–262.

\bibitem
{BFFU15}
S.  Boucksom, T.  de  Fernex,  C.  Favre  and  S.  Urbinati.  Valuation  spaces  and multiplier ideals on singular varieties, Recent advances in algebraic geometry, 29-51, London Math. Soc. Lecture Note Ser., 417, Cambridge Univ. Press, Cambridge, 2015.

\bibitem%[BFJ08]
{BFJ08}
S. Boucksom, C. Favre, and M. Jonsson. Valuations and plurisubharmonic singularities. Publ. RIMS {\bf 44} (2008), 449-494.

\bibitem%[BHJ17]
{BHJ17}
S. Boucksom, T. Hisamoto and M. Jonsson. Uniform K-stability, Duistermaat-Heckman measures and singularities of pairs, Ann. Inst. Fourier (Grenoble) 67 (2017), 87-139. arXiv:1504.06568.

\bibitem%[BHJ16]
{BHJ19}
S. Boucksom, T. Hisamoto and M. Jonsson. Uniform K-stability and asymptotics of energy functionals in K\"{a}hler geometry, J. Eur. Math. Soc. (JEMS) \textbf{21} (2019), no. 9, 2905-2944.

\bibitem
{BoJ18a}
S. Boucksom, M. Jonsson. Global pluripotential theory over a trivially valued field. arXiv:1801.08229v2.

\bibitem
{BoJ18b}
S. Boucksom, M. Jonsson. A non-Archimedean approach to K-stability. arXiv:1805.11160v1.




\bibitem
{Bri87}
M. Brion. Sur l'image de l'application moment. S\'{e}minaire d'Alg\`{e}bre Paul Dubreil et Marie-Paul Malliavin Lect. Notes Math., vol. 1296, Springer-Verlag, 1987.

\bibitem
{BSU13}
M. Brion,  P. Samuel,  and V.  Uma, Lectures  on  the  structure  of  algebraic groups  and  geometric  applications, CMI Lecture Series in Mathematics, vol. 1, Hindustan Book Agency, New Delhi; Chennai Mathematical Institute (CMI), Chennai, 2013.

%\bibitem
%{BX18}
%H. Blum, C. Xu. Uniqueness of K-polystable degenerations of Fano varieties, arXiv:1812.03538.

%\bibitem%[Che99]{Che99}
%J. Cheeger. Differentiability of Lipschitz functions on metric measure spaces. Geom. Funt. Anal. {\bf 9}, 428-517 (1999).

%\bibitem
%{CMM17}
%D. Coman, X. Ma and G. Marinescu, Equidistribution for sequences of line bundles on normal K\"{a}hler spaces, Geom. Topol. {\bf 21} (2017) 923-962.

%\bibitem
%{BM97}
%E. Bierstone, P. Milman, Canonical desingularization in characteristic zero by blowing up the maximum strata of a local invariant, Invent. Math. {\bf 128}
%(1997), 207-302.


\bibitem%[CDS15]
{CDS15}
X.X. Chen, S. K. Donaldson and S. Sun, K\"{a}hler-Einstein metrics on Fano manifolds, I-III, J. Amer. Math. Soc. {\bf 28} (2015), 183-197, 199-234, 235-278.

%\bibitem%[Cro80]
%{Cro80}
%C. Croke. Some isoperimetric inequalities and consequences. Ann. Sci. E. N. S., Paris, {\bf 13}: 419-435, 1980.

\bibitem
{CGZ13}
D. Coman, V. Guedj and A. Zeriahi, Extension of plurisubharmonic functions with growth control, J. Reine Angew. Math. {\bf 676} (2013), 33-49.

\bibitem
{DaS16}
V. Datar, G. Sz\'{e}kelyhidi, K\"{a}hler-Einstein metrics along the smooth continuity method, Geom. Funct. Anal. {\bf 26} (2016), no.4, 975-1010.

\bibitem%[Dar14]
{Dar15}
T. Darvas. The Mabuchi geometry of finite energy classes. Adv. Math. {\bf 285} (2015), 182-219.

\bibitem
{Dar17a}
T. Darvas. Weak geodesic rays in the space of K\"{a}hler potentials and the class $\cE(X, \omega)$, J. Inst. Math. Jussieu {\bf 16} (2017), no.4, 837-858.

\bibitem
{Dar17}
T. Darvas. Metric geometry of normal K\"{a}hler spaces, energy properness, and existence of canonical metrics, Int. Math. Res. Not. (IMRN) (2017), 6752-6777.

\bibitem
{DL19}
T. Darvas, C. Lu. Geodesic stability, the space of rays, and the uniform convexity in Mabuchi geometry, arXiv:1810.04661.

\bibitem
{DR17}
T. Darvas, Y. Rubinstein. Tian's properness conjectures and Finsler geometry of the space of K\"{a}hler metrics. J. Amer. Math. Soc. \textbf{30} (2017), 347-387.


\bibitem%[Dem92]
{Dem92}
J.-P. Demailly. Regularization of closed positive currents and intersection Theory. J. Alg. Geom. \textbf{1} (1992), 361-409.

\bibitem{DNG18}
E. Di Nezza, V. Guedj. Geometry and topology of the space of K\"{a}hler metrics, Composition Math. \textbf{154} (2018) 1593-1632.

\bibitem%[Der16]
{Der16}
R.Dervan, Uniform stability of twisted constant scalar curvature K\"{a}hler metrics, Int. Math. Res. Not. (2016) no. 15, 4728-4783.

%\bibitem%[DT92]
%{DT92}
%Ding, W.Y. and Tian, G.:  K\"{a}hler-Einstein metrics and the generalized
%Futaki invariant. Invent. Math., {\bf 110} (1992), no. 2, 315-335.
\bibitem
{Din88}
W.Y. Ding. Remarks on the existence problem of positive K\"{a}hler-Einstein metrics, Math. Ann. \textbf{282} (1988), 463-471. 

\bibitem%[Don02]
{Don02}
S. Donaldson, Scalar curvature and stability of toric varieties, J. Differential Geom. \textbf{62} (2002), no.2, 289-349.

\bibitem
{Don05}
S. Donaldson, Lower bounds of the Calabi functional, J. Differential Geom. \textbf{70} (2005), 453-472. 

%\bibitem%[Don12a]
%{Don12a}
%S.Donaldson. K\"{a}hler metrics with cone singularities along a divisor. Essays in mathematics and its applications. 49-79, Springer, Heidelberg, 2012.

%\bibitem%[Don12b]
%{Don12b}
%S.Donaldson. Stability, birational transformations and the K\"{a}hler-Einstein problem. Surveys in Differential Geometry, Vol XVII, International Press 2012.

%\bibitem%[DS14]
%{DS14}
%S. Donaldson, S. Sun, Gromov-Hausdorff limits of K\"{a}hler manifolds and algebraic geometry, Acta Math. {\bf 213} (2014), no. 1, 63-106.

\bibitem%[EGZ09]
{EGZ09}
P. Eyssidieux, V. Guedj, A. Zeriahi. Singular K\"{a}hler-Einstein metrics. J. Amer. Math. Soc. {\bf 22} (2009), 607-639.

\bibitem
{FN80}
J. E. Fornaess, R. Narasimhan. The Levi problem on complex spaces with singularities. Math. Ann. {\bf 248} (1980), 47-72.


\bibitem
{Fuj18}
K. Fujita, Optimal bounds for the volumes of K\"{a}hler-Einstein Fano manifolds. Amer. J. Math. 140 (2018), no. 2, 391-414.

\bibitem%[Fuj19a]
{Fuj19a}
K. Fujita, A valuative criterion for uniform K-stability of $\bQ$-Fano varieties, J. Reine Angew. Math. 751 (2019), 309-338.

\bibitem%[Fuj19b]
{Fuj19b}
K. Fujita, Uniform K-stability and plt blowups, Kyoto J. Math. 59 (2019), no.2, 399-418.

%\bibitem
%{FO18}
%K. Fujita, Y. Odaka. On the K-stability of Fano varieties and anticanonical divisors, Toholu Math. J. (2) 70 (2018), no. 4, 511-521.

%\bibitem%[Fuj17b]
%{Fuj17b}
%K. Fujita, Openness results for uniform K-stability, arXiv:1709.08209.

%\bibitem%[Fut83]
%{Fut83}
%A. Futaki. An obstruction to the existence of Einstein K$\ddot{a}$hler metrics, Inventiones Mathematicae (1983), 437-443.

%\bibitem%[GP16]
%{GP16}
%H. Guenancia, M. P\v{a}un, Conic singularities metrics with prescribed Ricci curvature: the case of general cone angles along normal crossing divisors. J. Diff. Geom., {\bf 103} (2016), no.1, 15-57.
%\bibitem%[Sha99]{Sha99}
%N. Shanmugalingam. Harmonic functions on metric spaces. 1999.

\bibitem{HL20}
J. Han, C. Li. On the Yau-Tian-Donaldson conjecture for generalized K\"{a}hler-Ricci soliton equations, arXiv:2006.00903.

\bibitem{GZ17}
V. Guedj, A. Zeriahi. Degenerate complex Monge-Amp\`{e}re equations, book, 470pp, EMS Tracts in Mathematics (2017).

\bibitem{Gol19}
A. Golota. 
Delta-invariants for Fano varieties with large automorphism groups, arXiv:1907.06261.

\bibitem%[GZ07]
{GZ07}
V. Guedj, A. Zeriahi. The weighted Monge-Amp\`{e}re energy of quasiplurisubharmonic functions, J. Funct. Anal. {\bf 250} (2007), no. 2, 442-482.

\bibitem
{His16a}
T. Hisamoto. Orthogonal projection of a test configuration to vector fields. arXiv:1610.07158.

\bibitem
{His16b}
T. Hisamoto. Stability and coercivity for toric polarizations. arXiv:1610.07998.

\bibitem
{His19}
T. Hisamoto. Mabuchi's soliton metric and relative D-stability. arXiv:1905.05948.
%\bibitem%[JMR16]
%{JMR16}
%T. Jeffres, R. Mazzeo, Y. Rubinstein. K\"{a}hler-Einstein metrics with edge singularities, with an appendix by C. Li and Y. Rubinstein, Ann. of Math. (2) {\bf 183} (2016), no. 1, 95-176.

%\bibitem
%{Hum81}
%J.E. Humphreys. Linear algebraic groups. Graduate Texts in Mathematics, No.21. Springer-Verlag, New York-Heidelberg, 1981. 

\bibitem
{Kir08}
A. Kirillov Jr., An introduction to Lie groups and Lie algebras, Cambridge Studies in Advanced Mathematics, vol. 113, Cambridge University Press, Cambridge, 2008. 

\bibitem
{Kol07}
J. Koll\'{a}r. Lectures on resolution of singularities, Annals of Mathematics Studies, 166, Princeton University Press, Princeton, NJ, 2007.

\bibitem
{KM98}
J. Koll\'{a}r,  S. Mori.  Birational  geometry  of  algebraic  varieties.  Cambridge Tracts in Mathematics, 134. Cambridge University Press, 1998.

\bibitem%[JM12]
{JM12}
M. Jonsson, M. Musta\c{t}\u{a}. Valuations and asymptotic invarians for sequences of ideals. Ann. Inst. Fourier {\bf 62} (2012), no.6, 2145-2209.

%\bibitem%[Li80]
%{Li80}
%P. Li, On the Sobolev constant and the $p$-spectrum of a compact Riemannian manifold. Ann. Sci. E.N.S., Paris, {\bf 13}, pp. 451-468, 1980.

%\bibitem%[Li12]
%{Li12}
%C. Li. K\"{a}hler-Einstein metrics and K-stability. Ph.D. Thesis, Princeton, June 2012.

\bibitem
{Li18}
C. Li. Minimizing normalized volume of valuations. Math. Z., Vol. 289, 2018, no. 1-2, 491-513. 

\bibitem%[Li15]
{Li17}
C. Li, K-semistability is equivariant volume minimization, Duke Math. J. {\bf 166}, number 16 (2017), 3147-3218. %, arXiv:1512.07205.

\bibitem
{Li20}
C. Li, Geodesic rays and stability in the cscK problem, arXiv:2001.01366.

\bibitem%[LX14]
{LX14}
C. Li and C. Xu. Special test configurations and K-stability of Fano varieties, Ann. of Math. (2) {\bf 180} (2014), no.1, 197-232.

\bibitem
{LTW17}
C. Li, G. Tian and F. Wang. On Yau-Tian-Donaldson conjecture for singular Fano varieties, arXiv:1711.09530.

\bibitem
{LTW19}
C. Li, G. Tian and F. Wang. The uniform version of Yau-Tian-Donaldson conjecture for singular Fano varieties. arXiv:1903.01215.

\bibitem%[LWX18]
{LWX18}
C. Li, X. Wang, C. Xu. Algebraicity of metric tangent cones and equivariant K-stability, arXiv:1805.03393.


\bibitem
{LXZ21}
Y. Liu, C. Xu and Z. Zhuang. Finite generation for valuations computing stability thresholds and applications to K-stability, arXiv:2102.09405. 
%\bibitem%[LX16]
%{LX16}
%C. Li, C. Xu, Stability of valuations and Koll\'{a}r components, arXiv:1604.05398.

%\bibitem%[NTZ15]
%{NTZ15}
%G. La Nave, G. Tian and Z. Zhang, Bounding diameter of singular K\"{a}hler metric, American Journal of Mathematics, 
%Volume {\bf 139}, no. 6, pp.1693-1731. arXiv:1503.03159.

%\bibitem%[Mat57]
%{Mat57}
%Y. Matsushima. Sur la structure du group d'hom\'{e}omorphismes analytiques d'une certaine 
%variet\'{e} Kaehl\'{e}rinne. Nagoya Math. J., {\bf 11}, 145-150 (1957).

%\bibitem
%{MR12}
%R. Mazzeo, Y. Rubinstein, The Ricci continuity method for the complex Monge-Amp\`{e}re equation, with applications to K\"{a}hler-Einstein edge metrics. 
%C. R. Math. Acad. Sci. Paris {\bf 350} (2012), no. 13-14, m693-697.


\bibitem
{Mab86}
T. Mabuchi.  K-energy maps integrating Futaki invariants. Tohoku Math. J. (2) {\bf 38}, No. 4 (1986), 575-593.



\bibitem
{Oda13}
Y. Odaka, The GIT stability of polarized varieties via discrepancy, Ann. of Math. 177 (2013), 645-661.

\bibitem
{Oko96}
A. Okounkov. Brunn-Minkowski inequality for multiplicities, Invent. Math. {\bf 125}, 405-411 (1996).

\bibitem
{Pas17}
 B. Pasquier.  Birational geometry of G-varieties, online lecture notes, 2017.

\bibitem
{LTZ19}
Y. Li, G. Tian, X. Zhu. Uniform K-stability modulo a group for stable pairs, in preparation.

\bibitem
{Mil17}
J. Milne. Algebraic Groups. The Theory of Group Schemes of Finite Type Over a Field. Cambridge Studies in Advanced Mathematics, 
No. 170. Cambridge University Press, Cambridge, 2017. 


\bibitem
{PRS08}
D. H. Phong, J. Ross, and J. Sturm. Deligne pairings and the Knudsen-Mumford expansion, J. Differential Geom. {\bf 78} (2008), 475-496.

%\bibitem%[RZ11]
%{RZ11}
%X. Rong, Y. Zhang, Continuity of extremal transitions and flops for Calabi-Yau manifolds, J. Differential Geometry, {\bf 89} (2011) 233-269.

%\bibitem%[She15]
%{She15}
%L. Shen, $C^{2,\alpha}$-estimate for conical K\"{a}hler-Ricci flow, arXiv:1412.2420.

%\bibitem%[Son14]
%{Son14}
%J. Song, Riemannian geometry of K\"{a}hler-Einstein currents, arXiv:1404.0445.

%\bibitem%[SSY]
%{SSY16}
%C. Spotti, S. Sun, C-J. Yao. Existence and deformations of K\"{a}hler-Einstein metrics on smoothable $\bQ$-Fano varieties,
%Duke Math. J. {\bf 165}, 16 (2016), 3043-3083.

\bibitem
{Ree89}
D. Rees. Izumi's theorem. In Commutative algebra (Berkeley, CA, 1987), 407-416. Math. Sci. Res. Inst. Publ., 15. Springer, New York. 1989. 

\bibitem
{Suss13}
H. S\"{u}{\ss}. K\"{a}hler-Einstein metrics on symmetric Fano T-varieties, Adv. in Math. \textbf{246} (2013), 100-113. 

\bibitem
{Sze15}
G. Sz\'{e}kelyhidi. Filtrations and test configurations, with an appendix by S. Boucksom, Math. Ann. {\bf 362} (2015), no. 1-2, 451-484.

%\bibitem%[Tia90]
%{Tia90}
%G. Tian, On Calabi's conjecture for complex surfaces with positive first Chern class. Invent. Math. {\bf 101}, (1990), 101-172.

%\bibitem%[Tia96]
%{Tia94}
%G. Tian, K\"ahler-Einstein metrics on algebraic manifolds. 
%Transcendental methods in algebraic geometry (Cetraro, 1994), 143–185, Lecture Notes in Math.,
%1646, Springer, Berlin, 1996.

\bibitem
{Tia00}
G.  Tian, Canonical  metrics  in  K\"{a}hler  geometry.  Lectures  in  Mathematics  ETH  Z\"{u}rich.Birkh\"{a}user Verlag, Basel (2000).


%\bibitem
%{Tia94}
%G. Tian. The K-energy on hypersurfaces and stability. Communications in Analysis and Geometry, no. 2 (1994), 239-265.


\bibitem%[Tia97]
{Tia97}
G. Tian. K\"{a}hler-Einstein metrics with positive scalar curvature. Invent. Math. {\bf 137} (1997), 1-37.



%\bibitem%[Tia12]
%{Tia12}
%G. Tian, Existence of Einstein metrics on Fano manifolds, in {\it Metric and Differential Geometry}, Progr. Math., 297, pp. 119-159. Birkh\"{a}user/Springer, Basel, 2012.
\bibitem
{Tia12}
G. Tian. Existence of Einstein metrics on Fano manifolds, in : Metric and Differential Geometry (X.-Z. Dai et al., Eds.), Springer, 2012, pp.119-159.

\bibitem%[Tia15]
{Tia15}
G. Tian, K-stability and K\"{a}hler-Einstein metrics. Comm.
Pure Appl. Math. {\bf 68} (7) (2015), 1085-1156. %Corrigendum: Comm. Pure Appl. Math. 68 (2015), no. 11, 2082-2083.


%\bibitem%[Tia17]
%{Tia17}
%G. Tian, A third derivative estimate for Monge-Ampere equations with conical singularities, Chinese Annals
%of Mathematics, Series B, Volume {\bf 38}, Issue 2, pp 687-694.

%\bibitem%[TW17a]
%{TW17a}
%G. Tian, F. Wang, On the existence of conic Kaehler-Einstein metrics, to appear.

%\bibitem%[TW17b]
%{TW18}
%G. Tian, F. Wang, Cheeger-Colding-Tian theory for conic K\"{a}hler-Einstein metrics, arXiv:1807.07209.

%\bibitem%[Wan12]
%{Wan12}
%X. Wang, Heights and GIT weight, Math. Res. Lett. {\bf 19} (2012), 909-926.

\bibitem
{WN12}
D. Witt Nystr\"{o}m, Test configuration and Okounkov bodies, Compos. Math. {\bf 148} (2012), no. 6, 1736-1756.

\bibitem
{Zhu19}
Z. Zhu. A note on equivariant K-stability, arXiv:1907.07655.


\end{thebibliography}
\end{document}